\documentclass{article}

\usepackage{graphicx}
\usepackage{psfrag}
\usepackage{amsfonts}
\usepackage{amsmath,amssymb}
\usepackage{float}
\usepackage{color}
\usepackage{multirow}
\usepackage{subfig}
\graphicspath{{./Figures/}}

\newcommand{\QED}{\hspace*{\fill}\rule{2.5mm}{2.5mm}}
\newtheorem{theorem}{Theorem}[section]
\newtheorem{lemma}[theorem]{Lemma}

\newtheorem{corollary}[theorem]{Corollary}

\newenvironment{proof}{{\bf Proof\ }}{\QED\\}

\numberwithin{equation}{section}

\newcommand{\ds}{\displaystyle}
\renewcommand{\vec}[1]{\mbox{\boldmath $#1$}}
\newcommand{\R}{\mathbb{R}}
\newcommand{\C}{\mathbb{C}}

\renewcommand{\le}{\leqslant}
\renewcommand{\ge}{\geqslant}
\newcommand {\go}  {{\cal O}}
\newcommand {\po}  {\mbox{\tiny{${\cal O}$}}}
\newcommand {\gol}[1]  {{\cal O}(L^{#1})}
\newcommand {\pr}  {\backsimeq}
\newcommand {\vpr}  {\varpropto\!\!}
\newcommand{\w}{\omega}
\newcommand{\eps}{\varepsilon}
\newcommand{\wm}{\omega_{m}}
\newcommand{\wmm}{\omega_{M}}
\newcommand{\km}{k_{m}}
\newcommand{\kmm}{k_{M}}

\newcommand{\xo}{x_0}
\newcommand{\xoo}{x_0^2}
\newcommand{\ac}{|c|}
\newcommand{\z}{(\xoo+4\nu^2k^2)}

\newcommand{\soc}{\mathrm{s}(c)}

\renewcommand{\max}{\mathrm{max}}

\renewcommand{\Re}{\mathrm{Re} \,}
\renewcommand{\Im}{\mathrm{Im} \,}

\newcommand{\tp}{\tilde{p}}
\newcommand{\tq}{\tilde{q}}
\newcommand{\tk}{\tilde{k}}
\newcommand{\tx}{\tilde{x}}
\newcommand{\tz}{\tilde{z}}

\newcommand{\bk}{\breve{k}}
\newcommand{\bx}{\breve{x}}
\newcommand{\bz}{\breve{z}}
\newcommand{\bw}{\breve{\omega}}
\newcommand{\yo}{Y_0}

\newcommand{\td}{\widetilde{D}}
\newcommand{\cc}{{\cal C}}
\newcommand{\ccsw}{{\cal C}_{sw}}
\newcommand{\ccse}{{\cal C}_{se}}
\newcommand{\ccn}{{\cal C}_{n}}
\newcommand{\cce}{{\cal C}_{e}}
\newcommand{\ccw}{{\cal C}_{w}}
\newcommand{\ccwi}{{\cal C}_{w}^\infty}

\title{Optimized Schwarz Waveform Relaxation for\\ Advection Reaction
  Diffusion Equations in Two Dimensions}

\author{Daniel Bennequin \thanks{Institut de Math\'ematiques de Jussieu, Universit\'e Paris 7, B\^atiment Sophie Germain,  
75205 Paris Cedex 13, France.
}\and Martin J. Gander\thanks{Section de Math\'ematiques, Universit\'e de Gen\`eve, 2-4 rue du Li\`evre, CP 240, CH-1211
Gen\`eve, Switzerland.} \and Loic Gouarin\thanks{Laboratoire de Math\'ematiques d'Orsay, Universit\'e Paris-Sud 11, 91405 Orsay Cedex, France} \and
  Laurence Halpern\thanks{Laboratoire Analyse, G\'eom\'etrie \& Applications
UMR  7539 CNRS, Universit\'e PARIS 13, 93430 VILLETANEUSE,
FRANCE}}

\begin{document}
\maketitle
\begin{abstract}
Optimized Schwarz Waveform Relaxation methods have been developed over
the last decade for the parallel solution of evolution problems.  They
are based on a decomposition in space and an iteration, where only
subproblems in space-time need to be solved. Each subproblem can be
simulated using an adapted numerical method, for example with local
time stepping, or one can even use a different model in different
subdomains, which makes these methods very suitable also from a
modeling point of view. For rapid convergence however, it is important
to use effective transmission conditions between the space-time
subdomains, and for best performance, these transmission conditions
need to take the physics of the underlying evolution problem into
account.  The optimization of these transmission conditions leads to
mathematically hard best approximation problems of homographic
functions.  We study in this paper in detail the best approximation
problem for the case of linear advection reaction diffusion equations
in two spatial dimensions.  We prove comprehensively best
approximation results for transmission conditions of Robin and Ventcel
(higher order) type, which can also be used in the various limits for
example for the heat equation, since we include in our analysis a
positive low frequency limiter both in space and time. We give for
each case closed form asymptotic values for the parameters which can
directly be used in implementations of these algorithms, and which
guarantee asymptotically best performance of the iterative methods.
We finally show extensive numerical experiments including cases not
covered by our analysis, for example decompositions with cross
points. In all cases, we measure performance corresponding to our
analysis.
\end{abstract}


\section{Introduction}

Schwarz waveform relaxation algorithms are parallel algorithms to
solve evolution problems in space time. They were invented
independently in \cite{GS:1998:STC} and \cite{Giladi:2002:STD}, see
also \cite{Gander:2002:OSW}, based on the earlier work in
\cite{Bjorhus:1995:DDS}, and are a combination of the classical
waveform relaxation algorithm from \cite{Lelarasmee:1982:WRM} for the
solution of large scale systems of ordinary differential equations,
and Schwarz methods invented in \cite{Schwarz:1870:UGA}.  Modern
Schwarz methods are among the best parallel solvers for steady partial
differential equations, see the books
\cite{Smith:1996:DPM,Quarteroni:1999:DDM,Toselli:2004:DDM} and
references therein. Waveform relaxation methods have been analyzed for
many different classes of problems recently: for fractional
differential equations see \cite{Jiang:2013:WRM}, for singular
perturbation problems see \cite{Zhao:2012:OTC}, for differential
algebraic equations see \cite{Bai:2011:OCC}, for population dynamics
see \cite{GerardoGiorda:2008:BWR}, for functional differential
equations see \cite{ZubikKowal:1999:WRF}, and especially for partial
differential equations, see
\cite{Janssen:1996:MWR,Janssen:1996:MWR2,VanLent:2002:MWR} and the
references therein.  For the particular form of Schwarz waveform
relaxation methods, see \cite{Anfray:2002:NTC,
  Gander:2005:OSW, Daoud:2007:OSW, Daoud:2007:OSWFB, Jiang:2008:SWR,
  Zhang:2010:SWR,Garbay:2010:AOA, Ltaief:2010:AOA, caetano:2010:swr,
  Wu:2011:CAO,Liu:2011:WRF, Liu:2012:APW}.  These algorithms have also
become of interest in the moving mesh R-refinement strategy, see
\cite{Haynes:2006:ASW, Haynes:2008:AMM, Gander:2012:DDA}, and
references therein.

Schwarz waveform relaxation methods however exhibit only fast
convergence, when optimized transmission conditions are used, as first
shown in \cite{Gander:1999:OCO}, and then treated in detail in
\cite{Martin:2005:AOS,Gander:2007:OSW,Bennequin:2009:HBA,Binh:2010:SWR}
for diffusive problems, and \cite{Gander:2003:OSWW,Gander:2004:ABC}
for the wave equation, see also \cite{Gander:2004:OWR,Gander:2009:OWR}
for circuit problems, and \cite{Audusse:2011:OSW} for the primitive equations. With optimized transmission conditions, the
algorithms can be used without overlap, and optimized transmission
conditions turned out to be important also for Schwarz algorithms
applied to steady problems, for an overview, see
\cite{Gander:2006:OSM} and references therein.  In order to make such
algorithms useful in practice, one needs simply to use formulas for
the optimized parameters, which can then be put into implementations
and lead to fast convergent algorithms, without having to think about
optimizing transmission conditions ever again.

The purpose of this paper is to provide such formulas for a general
evolution problem of advection reaction diffusion type. The analysis
required to solve the associated optimization problems is substantial,
and only asymptotic techniques lead to easy to use, closed form
formulas. We also use and extend more general, abstract results for
best approximation problems, which appeared in
\cite{Bennequin:2009:HBA}. In particular, we remove a compactness
condition which remained in \cite{Bennequin:2009:HBA} in the case of
overlap. We obtain with our analysis the best choice of Robin
transmission conditions, and also higher order transmission conditions
called Ventcel conditions (after the Russian mathematician
A. D. Ventcel, also spelled Venttsel, Ventsel or Wentzell
\cite{ventcel}), both for the case of overlapping and non-overlapping
algorithms. We give complete proofs of optimality, generalizing
one-dimensional results given in \cite{Gander:2007:OSW} and
\cite{Bennequin:2009:HBA}. We also illustrate our results with
numerical experiments.

\section{Model Problem and Main Results}
\label{sec:modelproblem}

We study the optimized Schwarz waveform relaxation algorithm for the
time dependent advection reaction diffusion equation in $\Omega\subset
\R^2$,
\begin{equation}\label{eq:ModelProblem}
  \mathcal{L}u:=\partial_tu +\vec{a}\cdot \nabla u -\nu \Delta u+bu=f,\quad
    \mbox{in $\Omega \times (0,T)$,}
\end{equation}
where $\nu>0$, $b\ge 0$ and $\vec{a}=(a,c)^T$, and suitable boundary
conditions need to be prescribed on the boundary of $\Omega$, which
will however not play an important role, and we will not mention this
further. In order to describe the Schwarz waveform relaxation
algorithm, we decompose the domain into J non-overlapping subdomains
$U_j$, and then enlarge them, if desired, in order to obtain an
overlapping decomposition given by subdomains $\Omega_j$.  The
interfaces between subdomain $\Omega_i$ and $\Omega_j$ are then
defined by $\Gamma_{ij}=\partial \Omega_i \cap\overline{U}_j$.  The
algorithm for such a decomposition calculates then for $n=1,2,\ldots$
the iterates $(u_j^n)$ defined by
\begin{equation}\label{eq:ModifiedSchwarz}
  \begin{array}{rcl}
    \mathcal{L}u_i^n & = & f\quad \mbox{in $\Omega_i\times(0,T)$}\\
	u_i^n(\cdot,\cdot,0)  & = & u_0 \quad \mbox{in $\Omega_i$},\\
    \mathcal{B}_{ij}u_i^n  & = &
    \mathcal{B}_{ij}u_j^{n-1} \mbox{ on $\Gamma_{ij}\times (0,T)$},
  \end{array}
\end{equation}
where the $\mathcal{B}_{ij}$ are linear differential operators in
space and time, and initial guesses $\mathcal{B}_{ij}u_j^0$ on
$\Gamma_{ij}\times (0,T) $ need to be provided.

There are many different choices for the operators ${\cal B}_{ij}$.
Choosing for ${\cal B}_{ij}$ the identity leads to the classical Schwarz
waveform relaxation method, which needs overlap for convergence.
Zeroth or higher order differential conditions lead to optimized
variants, which also converge without overlap, see for example
\cite{Gander:2007:OSW} and \cite{Bennequin:2009:HBA}, where a complete
analysis in one dimension was performed. We study here in detail the
case where the transmission operators are of the form
\begin{equation}\label{eq:ventcel}
 {\cal B}_{ij} =(\nu\nabla-
              \frac{\vec{a}}{2})\cdot \vec{n_i}
              +\frac{s}{2},\qquad
     s= p + q (\partial_t+c \partial_y-\nu \Delta_y).
\end{equation}
If $q=0$, these are Robin transmission conditions, whereas for $q\ne
0$, they are called Ventcel transmission conditions. In the ideal
case where $\Omega=\R^2$ is decomposed into two half spaces
$\Omega_1=(-\infty,L)\times \R$ and $\Omega_2=(0,\infty)\times \R$, we
can compute explicitly the error in each subdomain at step $n$ as a
function of the initial error. We use Fourier transforms in time and
in the direction $y$ of the boundary, with $\omega$ and $k$ the
Fourier variables. The convergence factor $\rho(\omega,k,p,q,L)$ of
algorithm (\ref{eq:ModifiedSchwarz}), which gives precisely the error
reduction of each error component in $\omega$ and $k$ for a given
choice of parameters $p$ and $q$ and overlap $L$, can in this case be
computed in closed form (see \cite{Gander:2007:OSW}),
\begin{equation}\label{GeneralConvFactor}
  \rho(\omega,k,p,q,L)=\frac{p+q(\nu k^2+i (\omega+c k))-
  \sqrt{x_0^2+4\nu(\nu k^2+i (\omega+c k))}}
 {p+q(\nu k^2+i (\omega+c k))+\sqrt{x_0^2+4\nu(\nu k^2+i (\omega+c k))}}
    e^{- \frac{L\sqrt{x_0^2+4\nu(\nu k^2+i (\omega+c k))}}{2\nu}},
\end{equation}
where we denote by $\sqrt{\ }$ the standard branch of the square root
with positive real part, $x_0^2:=a^2+4\nu b$ and
$i=\sqrt{-1}$. Computing on a (uniform) grid, we assume that the
maximum frequency in space is $\kmm=\frac{\pi}{h}$ where $h$ is the
local mesh size in $x$ and $y$, and the maximum frequency in time is
$\wmm=\frac{\pi}{\Delta t}$, and that we also have estimates for the
lowest frequencies $\km$ and $\wm$ from the geometry, see for example
\cite{Gander:2006:OSM} for estimates, or for a more precise analysis
see \cite{Gander:2010:OTI}.  We also assume that the mesh sizes in
time and space are related either by $\Delta t =C_h h$, or $\Delta t
=C_h h^2$, corresponding to a typical implicit or explicit time
discretization of the problem.

Defining $D:=\{(\omega,k), \wm\le |\omega|\le \wmm,\,\km\le |k|\le
\kmm \}$, the parameters $(p^*,q^*)$ which give the best convergence
factor are solution of the best approximation problem
\begin{equation}\label{BestApprox}
   \inf_{(p,q)\in\C^2}\ \sup_{(\omega,k)\in
     D}|\rho(\omega,k,p,q,L)|=\sup_{(\omega,k)\in D}|\rho(\omega,k,p^*,q^*,L)|
     =: \delta^*(L).
\end{equation}
To motivate the reader, we outline in Table \ref{tab:convfactor} the
asymptotic behavior of the convergence factors, which can be achieved
by optimization.
We use here the notation $Q\pr h$ or $Q =\vpr(h)$ if there exists $C
\ne 0$ such that $Q\sim C h$.
\begin{table}
\centering
\renewcommand{\arraystretch}{1.5}
\begin{tabular}{|c|c|c|}
\hline
Method  & No overlap & Overlap $L$\\
\hline
\hline
 Dirichlet &
1& $1- \vpr(L)$
\\
\hline
Robin &
$1- \vpr(\sqrt{h})$& $1- \vpr(\sqrt[3]{L})$
\\
\hline
Ventcel &
 $1- \vpr(\sqrt[4]{h})$ & $1- \vpr(\sqrt[5]{L})$
\\
\hline
\end{tabular}
\caption{The asymptotically optimized convergence factors $\delta^*(L)$. }
\label{tab:convfactor}
\end{table}

In what follows, we will often use the quantity
$$
  \bar{k} =|c|\,\frac{\sqrt{(c^2+x_0^2)^2+16\nu^2\wm^2}
    -(c^2+x_0^2)} {8\nu^2\wm}.
$$
By a direct calculation, we see that $0\le \bar{k} |c|\le \wm$, and we
define the function
\begin{equation}\label{phi}
    \begin{array}{l}
    \varphi(k,\xi):=2\sqrt{2}\sqrt{\sqrt{(x_0^2+4\nu^2 k^2)^2+16\nu^2\xi^2}
    +x_0^2+4\nu^2 k^2},
    \end{array}
\end{equation}
and the constant
\begin{equation}\label{ConstA}
  A = \left\{
    \begin{array}{ll}
    \varphi(\bar{k},-\wm+|c|\bar{k})&\mbox{if $\km\le\bar{k}$,}\\
   \varphi(\km,-\wm+|c|\km) &
   \mbox{if $\bar{k}\le \km\le \frac{1}{|c|}\wm$,}\\
   \varphi(\km,0) & \mbox{if $\km\ge \frac{1}{|c|}\wm $.}\\
    \end{array}\right.
\end{equation}
We state in the following two subsections the main theorems which we
will prove in this paper, for both overlapping and non-overlapping
variants of the algorithm.

\subsection{Robin Transmission Conditions}

\begin{theorem}[Robin Conditions without Overlap]\label{th:toutno}
For small $h$ and small $\Delta t$, the best approximation problem
(\ref{BestApprox}) with $L=0$ has a unique solution
$(p_0^*(0),\delta_0^*(0))$, which is given asymptotically by
\begin{equation}\label{BestPRhoNoOverlap}
  p_0^*(0)\sim\sqrt{\frac{A}{Bh}},\quad \delta^*_0(0)\sim 1-\frac{1}{2}\sqrt{ABh},
\end{equation}
where $A$ is defined in (\ref{ConstA}), and
\begin{equation}\label{BestPRhoNoOverlapB}
  B = \left\{\begin{array}{ll}
    \frac{2}{\nu\pi} & \mbox{if $\Delta t=C_h h$},\\
    C\frac{\sqrt{2d}}{\nu\pi} & \mbox{if $\Delta t=C_h h^2$,
    $d:=\nu \pi C_h$,
      $C=\left\{\begin{array}{ll}
          1 & \mbox{if $d <d_0$},\\
    \sqrt{\frac{d+\sqrt{1+d^2}}{1+d^2}} & \mbox{if $d\ge d_0$,}
		\end{array}\right.$}
 \end{array}\right.
\end{equation}
where $d_0\approx 1.543679$ is the unique real root of the
polynomial $d^3-2d^2+2d-2$.
\end{theorem}
Partial results in the spirit of this theorem were already obtained
earlier:
\begin{enumerate}
  \item If $\km=\wm=0$, all three cases in (\ref{ConstA}) coincide,
    since $\bar{k}=0$, and the constant $A$ simplifies to $A=4x_0$,
    and we find the case analyzed in \cite{Halpern:2009:OSW}.
  \item If $\km$ and $\wm$ do not both vanish simultaneously, and we
    are in the case of the heat equation, $a=0$, $b=0$, $c=0$,
    $\nu=1$, we also obtain $\bar{k}=0$, and
    $A=4\sqrt{2\left(\sqrt{\km^4+\wm^2}+\km^2\right)}$, the case
    analyzed in \cite{Binh:2010:SWR}.  Note that the stability
    constraint for the heat equation discretized with a finite
    difference scheme is $4\nu\Delta t \le h^2$, which with our
    notation implies that $d\le \pi/4 \sim 0.7854$, a value smaller
    than $d_0$, and hence the constant $C$ in
    \eqref{BestPRhoNoOverlapB} is equal to 1.

\end{enumerate}

For the algorithm with overlap, $L>0$, we treat two asymptotic cases:
the continuous case deals with the small overlap parameter $L$ only,
while the discrete case involves also the grid parameters. In the
continuous case, we consider the parameters $\wmm$ and $\kmm$ to be
equal to $+\infty$.

\begin{theorem}[Robin Conditions with Overlap, Continuous]\label{th:touto}
For small overlap $L > 0$, the best approximation problem
\eqref{BestApprox} on $D^\infty:=\{(\omega,k), \wm\le |\omega|\le
+\infty,\,\km\le |k|\le +\infty\}$ has a unique solution
\begin{equation}
  p_{0,\infty}^*(L)\sim\frac12\sqrt[3]{\frac{\nu A^2}{L}},\quad
     \delta_{0,\infty}^*(L)\sim 1-\frac{A}{2p_{0,\infty}^*(L)}
,
\end{equation}
where $A$ is defined in (\ref{ConstA}).
\end{theorem}

If the overlap is fixed, the above analysis gives the behavior of the
best parameter when $h$ and $\Delta t$ tend to zero. However, the
overlap contains in general a few grid points only, and then the
discretization also needs to be taken into account:
\begin{theorem}[Robin Conditions with Overlap, Discrete]\label{th:toutobounded}
For small $\Delta t$ and $h$, for $L\pr h$, the best approximation
problem \eqref{BestApprox} on $D$ has a unique solution
\begin{equation}\label{RobinOverlapDiscrete}
\begin{array}{ll}
  \begin{array}{rl}
  \mbox{for $\Delta t \pr h^2$}:&
    p_0^*(L)  \sim  p^*_{0,\infty} (L)
    ,\\
    \mbox{for $\Delta t \pr h$}:&
    p_0^*(L)\sim\ds \frac{p^*_{0,\infty} (L)}{\sqrt[3]{2}},
    \end{array}
    &
    \delta_0^*(L)  \sim   1-\ds\frac{A}{2p_{0}^*(L)}.
  \end{array}
  \end{equation}
\end{theorem}

\subsection{Ventcel Transmission Conditions}

In order to present the theorems, we need to define two auxiliary
functions: first
$$
  g(t)=\frac{2t-\sqrt{t^2+1}}{t^2+1},
$$
and we denote for $Q < g_0\approx 0.3690$ by $t_2(Q)$ the only root of
the equation $g(t)=Q$ larger than $t_0=\sqrt{54+6\sqrt{33}}/6\approx
1.567618292$. Next we also define
\begin{equation}\label{eq:PdeQ}
P(Q)=
\left\{
\begin{array}{ll}
\sqrt{1+\sqrt{t_2(Q)^2+1}}( \frac{1}{\sqrt{t_2(Q)^2+1}} +Q)
&\mbox{if }Q < g_1\approx 0.3148, \\
 1 +Q
&\mbox{if }Q> g_1.
\end{array}
\right.
\end{equation}
\begin{theorem}[Ventcel Conditions without Overlap]\label{th:pqnoverlap}
The best approximation problem has for $L~=~0$ a unique solution
$(p_1^*(0),q_1^*(0))$, given by
\begin{equation}\label{eq:pqnoverlap}
\begin{array}{llllll}
  \mbox{for $\Delta t =C_hh$ and $\frac{AC_h}{8} < 1$}:&
    p_1^*(0)  \sim  \frac{1}{2}\sqrt[4]{\frac{\nu\pi A^3}{4h}},&\ 
    q_1^*(0) \sim  \frac{8ph}{\pi A}, \\
  \mbox{for $\Delta t =C_hh$ and $\frac{AC_h}{8} > 1$}:&
    p_1^*(0)  \sim  \sqrt[4]{\frac{\nu\pi
        A^2}{2C_h(P(\frac{8}{C_hA}))^2h}}, &
    q_1^*(0)  \sim  \frac{8ph}{\pi A}, \\
  \mbox{for $\Delta t =C_hh^2$}:&
    p_1^*(0)  \sim  \frac{1}{2}\sqrt[4]{\frac{\nu\pi A^3}{4Ch}\sqrt{\frac{2}{d}}},\quad
   &  q_1^*(0)  \sim   \frac{8Cph}{\pi A}\sqrt{\frac{d}{2}},\\[1em]
    &\ds \delta_1^*(0) \sim  1-\ds\frac{A}{2p_1^*(0)}. & &
\end{array}
\end{equation}
Here again $A$ is the constant defined in \eqref{ConstA}, $d$ and $C$ are
the constants defined in \eqref{BestPRhoNoOverlapB}.
\end{theorem}

\begin{theorem}[Ventcel Conditions with Overlap, Continuous]\label{th:toutopq}
For small overlap $L > 0$, the best approximation problem
\eqref{BestApprox} on $D^\infty$ has the unique solution
\begin{equation}
  p_{1,\infty}^*(L)\sim \frac12\sqrt[5]{\frac{ \nu A^4}{8L}},\quad
  q_{1,\infty}^*(L)\sim 4 \sqrt[5]{\frac{\nu^2L^3}{2A^2}},\quad
     \delta_{1,\infty}^*(L)\sim1-\ds\frac{A}{2p_{1,\infty}^*(L)},
\end{equation}
where $A$ is defined in (\ref{ConstA}).
\end{theorem}

\begin{theorem}[Ventcel Conditions with Overlap, Discrete]
\label{th:toutoboundedpq}
For small $\Delta t$ and $h$, for $L\pr h$, the best approximation
problem \eqref{BestApprox} on $D$ has a unique solution
\begin{equation}
\begin{array}{ll}
  \begin{array}{rll}
  \mbox{for $\Delta t \pr h^2$}:&
    p_1^*(L)  \sim  p^*_{1,\infty }(L),&
    q_1^*(L) \sim  q^*_{1,\infty }(L), 
       \\
  \mbox{for $\Delta t\pr h$}:&
    p_1^*(L)   \sim   2^{-\frac15}p^*_{1,\infty }(L),&
  q_1^*(L)   \sim   2^{\frac35} q^*_{1,\infty }(L), 
\end{array}
  &\delta_1^*(L) \sim 1-\ds\frac{A}{2p_{1}^*(L)}.     
\end{array}
\end{equation}
\end{theorem}

\section{Abstract Results}

We now recall the abstract results on the best approximation problem
(\ref{BestApprox}) from \cite{Bennequin:2009:HBA}, and present an
important extension, which allows us to remove a compactness
assumption in the overlapping case. We start by rewriting the
convergence factor (\ref{GeneralConvFactor}) in the form
\begin{equation}\label{eq:rho}
  \rho(z,s,L)=\ds  \frac{s-z}{s+z} \
   e^{- \frac{Lz}{2\nu}},\quad
  z :=\sqrt{x_0^2+4\nu(\nu k^2+i (\omega+c k))},\quad s =p+q(\nu k^2+i (\omega+c k)).
\end{equation}
In order to separate real and imaginary parts of the square root, we
introduce the change of variables ${\mathcal T}: (k,\w)\mapsto
z=x+iy$, which transforms the domain $D$ into $\td=\td_+\cup
\overline{\td_+}$, with $\td_+\subset \R_+\times\R_+$, as illustrated
in Figure \ref{Fig:ChangeOfVars}.
\begin{figure}
  \centering
  \includegraphics[width=0.45\textwidth]{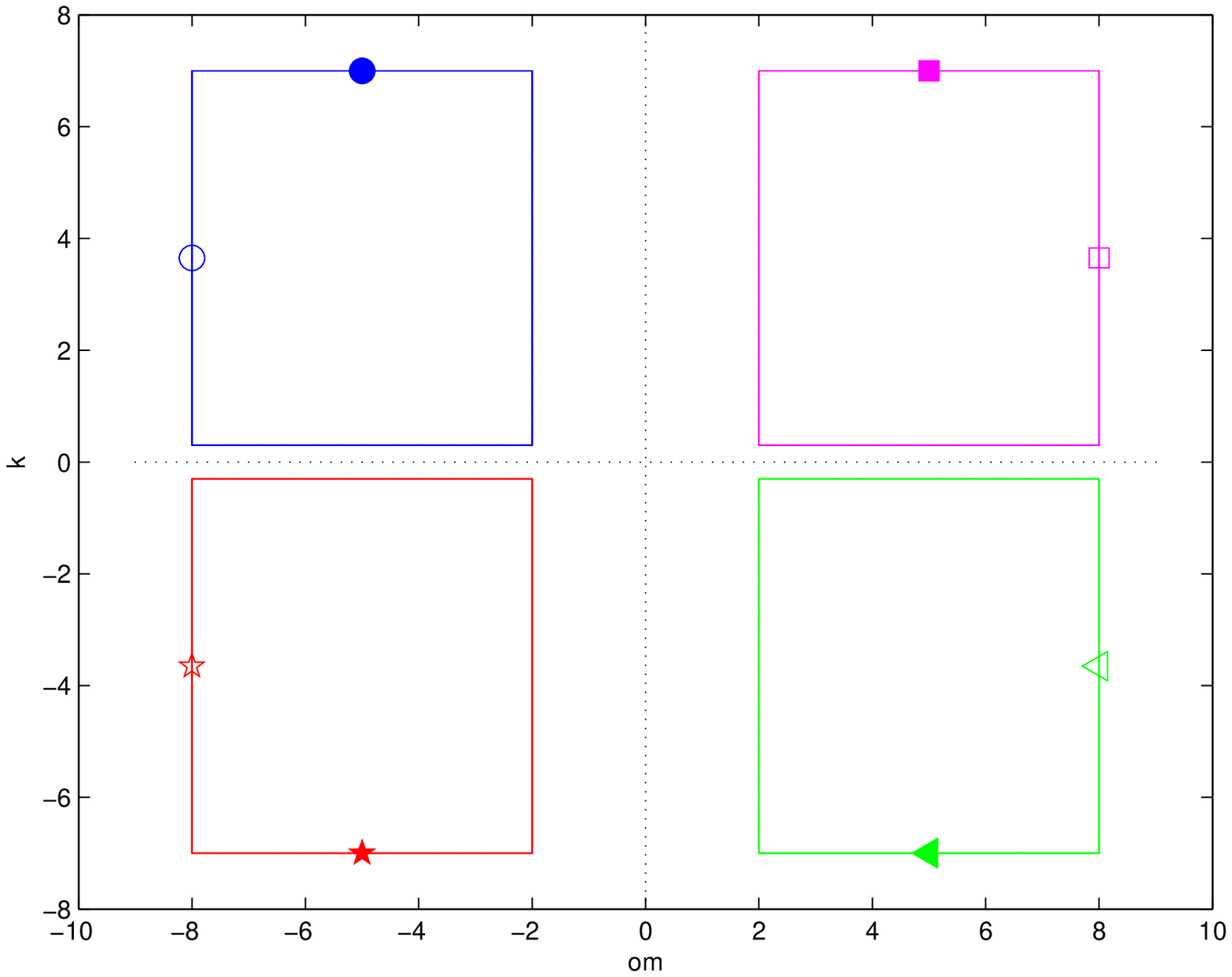}
  \includegraphics[width=0.45\textwidth]{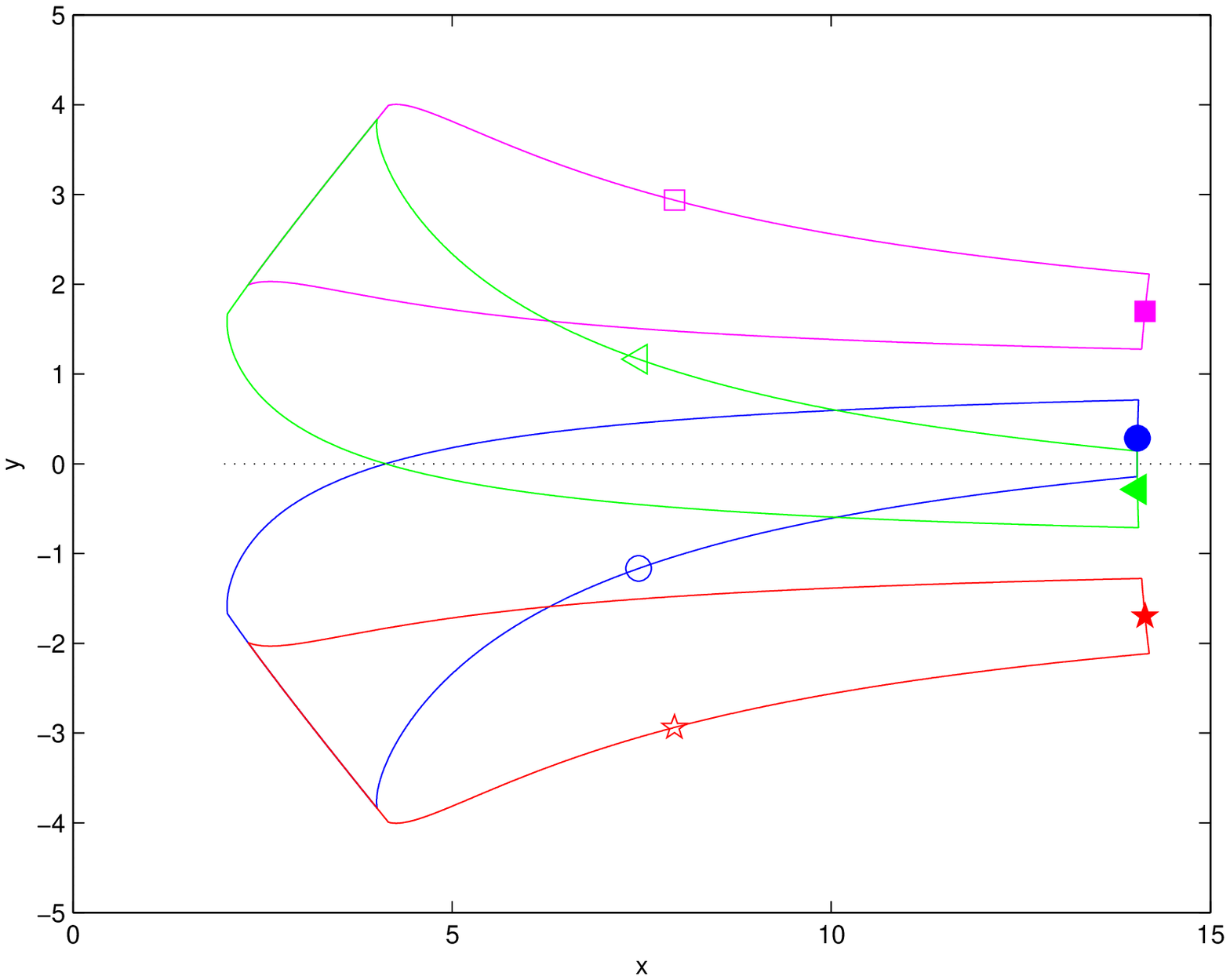}
  \caption{How the change of variables to simplify the convergence factor
    transforms the frequency domains}
  \label{Fig:ChangeOfVars}
\end{figure}
The domain $\td_+$ is compact, and lies below the line $x=y$, as
one can see from the coordinates $(x,y)=(\Re {\mathcal T}(k,\w), \Im
{\mathcal T}(k,\w))$, which satisfy
\begin{subequations}\label{eq:systemxy}
  \begin{eqnarray}
    x^2-y^2& = & \xoo+4\nu^2k^2,\label{eq:systemxy1}\\
    2xy    & = &4\nu (\w+ck). \label{eq:systemxy2}
  \end{eqnarray}
\end{subequations}
We further assume that the coefficients and parameters satisfy
\begin{equation}\label{eq:hypcoefficients}
\mbox{either\quad $\xoo+4\nu^2\km^2 \ne 0$,\quad  or\quad $\wm \ne 0$},
\end{equation}
which implies that there exists an $\alpha>0$ such that
\[
  \forall z\in\td, \quad \Re z \ge \alpha > 0.
\]

We also use the notation $\rho_0(z,p,q):=\frac{s-z}{s+z}$,
$\rho(z,p,q,L):=\rho_0(z,p,q)e^{-Lz/2\nu}$. The min-max problem
(\ref{BestApprox}) in the new $(x,y)$-coordinates takes now
the simple form
\begin{equation}\label{BestApproxtilde}
  \ds \inf_{(p,q)\in\C^2}\ \sup_{z\in
     \td}|\rho(z,p,q,L)| =\sup_{z\in
     \td}|\rho (z,p^*,q^*,L)|
     =: \delta^*(L).
\end{equation}
For convenience, we will also use the notation $R_0(\w,k,p,q)$ or
$R_0(z,p,q)$ for $|\rho_0(z,p,q)|^2 $, and
$R(\w,k,p,q,L)=R(z,p,q,L)=R_0(z,p,q)e^{-Lx/\nu}$.

\subsection{Robin Transmission Conditions}

In this case, we set $q=0$, and we will simply use the above notation
without the parameter $q$ in the arguments, writing for instance
$\rho(z,p,L)$, $\rho_0(z,p)$, \textit{etc.}. We also call the minimum in
the Robin case $\delta_0^*(L)$.

We start with the non-overlapping case, $L=0$, where there is a nice
geometric interpretation of the min-max problem
(\ref{BestApproxtilde}): for a given point $z_o\in \C$ and a parameter
$\delta\in\R$, we introduce the sets
\begin{equation}\label{circle}
  {\cal C}(z_0,\delta)=\{z\in \C;\ \left|\frac{z-z_0}{z+z_0}\right|=\delta\},\,
  \bar{\cal D}(z_0,\delta)=\{z\in \C; \ \left|\frac{z-z_0}{z+z_0}\right|\le \delta\}.
\end{equation}
Note that ${\cal C}(z_0,\delta)$ is a circle centered at
$\frac{1+\delta^2}{1-\delta^2}z_0$, cutting the $x-$axis at the points
$\frac{1-\delta}{1+\delta} z_0$ and $\frac{1+\delta}{1-\delta} z_0$,
and $\bar{\cal D}(z_0,\delta)$ is the associated disk. Now because of
the form of the convergence factor $\rho_0(z,p,q)=\frac{s-z}{s+z}$,
$(p^*,\delta^*)$ is a solution of the min-max problem
(\ref{BestApproxtilde}) if and only if for any $z$ in $\td$, $z$ is in
$\bar{\cal D}(p^*,\delta^*)$. This means geometrically that the
solution of the min-max problem (\ref{BestApproxtilde}) is represented
by the smallest circle centered on the real axis which contains
$\td$. We will use this interpretation as a guideline in the analysis,
also for the overlapping case!

\begin{theorem}\label{th:gennonoverlap}
  For any set of coefficients such that \eqref{eq:hypcoefficients} is
  satisfied, and $\kmm$ and $\wmm$ being finite, the min-max problem
  \eqref{BestApproxtilde} with $L=0$ has a unique solution
  $(\delta^*_0(0),p^*_0(0))$ with $\delta^*_0(0) < 1$. The optimized
  parameter $p^*_0(0)$ is real and positive, and any strict local minimum
  on $\R$ of the real function
 \begin{equation}\label{BestApproxtildereal}
   F_0(p)=\sup_{z\in \td_+}|\rho_0(z,p)|
  \end{equation}
  is the global minimum.
 \end{theorem}
\begin{proof}
  Since $\td$ is compact, and with the assumption
  \eqref{eq:hypcoefficients} we have $\Re z \ge \alpha > 0$ with
  $\alpha= \sqrt{\xoo +4\nu^2\km^2}$ in the first case of
  \eqref{eq:hypcoefficients} or $\alpha=\sqrt{2\nu\wm}$ in the second
  case, we can use directly the analysis in \cite{Bennequin:2009:HBA}
  for polynomials of degree zero to get existence and uniqueness. The
  fact that the optimized parameter must be real follows directly from
  the symmetry of $\td$ with respect to the $x$-axis and the geometric
  interpretation, and finally that any strict local minimum is the global
  minimum follows as in \cite{Bennequin:2009:HBA}.
\end{proof}

In \cite{Bennequin:2009:HBA} one can also find a proof of the
existence of a solution to the min-max problem (\ref{BestApproxtilde})
in the overlapping case, and uniqueness is shown for $L$ small enough,
such that
\[
   \delta^*(L) e^{\frac{L}{2\nu} \sup_{z \in \td} \Re z }< 1.
\]
This constraint imposes that $\td$ is bounded in the $x$ direction. We
show now that this constraint is not necessary, using the fact that
in $\td$ the real part of $z$ is strictly larger than the absolute
value of its imaginary part.
\begin{theorem}\label{th:genoverlap}
  For any $L$, for $\kmm$ and $\wmm$ finite or not, and with the
  assumption (\ref{eq:hypcoefficients}), the min-max
  problem \eqref{BestApproxtilde} has a unique solution
  $(\delta^*_0(L),p^*_0(L))$. The optimized parameter $p^*_0(L)$ is
  real, positive, and any strict local minimum on $\R$ of the real
  function
  \begin{equation}\label{BestApproxtilderealoverlap}
    F_L(p)=\sup_{z\in  \td_+}|\rho(z,p,L)|
  \end{equation}
  is the global minimum.
\end{theorem}
\begin{proof}
By Theorem 2.8 in \cite{Bennequin:2009:HBA}, we know that a (possibly
complex) solution $p^*=p^*_1+ip^*_2$ of \eqref{BestApproxtilde} exists.
We now compute explicitly the modulus of the convergence factor,
\[
  |\rho_0(z,p)|^2=\ds\frac{(x-p_1)^2+(y-p_2)^2}{(x+p_1)^2+(y+p_2)^2}.
\]
We first note that for any $z$, and any $(p_1,p_2)$ with $p_1 > 0$, we
have $|\rho_0(z,-p_1+ip_2)| > |\rho_0(z,p_1+ip_2)|$, and therefore we
must have $p^*_1>0$. Next, in order to show that $|p^*_2| \le p^*_1$,
we assume the contrary, $|p^*_2| > p^*_1$, to reach a contradiction
(in particular this means that $p^*_2 \ne 0$). We calculate the
gradient,
\[
  \begin{array}{l}
  \partial_{p_1}|\rho_0(z,p)|^2= -4\ds\frac{x(x^2+y^2+p_2^2-p_1^2)-2yp_1p_2}
  {((x+p_1)^2+(y+p_2)^2)^2},\\
  \partial_{p_2}|\rho_0(z,p)|^2= -4\ds\frac{y(x^2+y^2+p_1^2-p_2^2)-2xp_1p_2}
  {((x+p_1)^2+(y+p_2)^2)^2},
  \end{array}
\]
which gives, with $\eps=\mbox{sign}(p^*_2)$,
\[
(\partial_{p_1}-\eps \partial_{p_2})|\rho_0(z,p^*)|^2=
-4\ds\frac{(x-\eps y)(x^2+y^2+2\eps p_1p_2)+(x+\eps y)(p_2^2-p_1^2)}
{((x+p_1)^2+(y+p_2)^2)^2} \,<  \, 0,
\]
where we used the fact that $x>|y|$ as we noted earlier (see Figure
\ref{Fig:ChangeOfVars}). This shows that $|\rho_0(z,p)|e^{-Lx/2\nu}$
decays in the neighborhood of $p^*$, in the direction $(1,-\eps)$, if
$|p^*_2| > p^*_1$, which is in contradiction with the fact that the
minimum is reached, and hence we must have $|p^*_2| \le p^*_1$.

Now for any $z$ in $\td$, since $x>|y|$ and $|p^*_2| \le p^*_1$,
we have
\[
  \ds \Re \frac{p}{z} = \frac{xp^*_1+yp^*_2}{|z|^2}\, > \, 0.
\]
This allows us to prove that the set of best approximations is convex:
consider the disk defined in \eqref{circle}. We have seen that
$(p^*(L),\delta^*(L))$ is a solution of the best approximation problem
\eqref{BestApproxtilde}, if and only if for any $z$ in $\td$, $z$ is
also in $\bar{\cal D}(p^*(L),\delta^*(L) e^{Lx/2\nu})$, which is
equivalent by dividing numerator and denominator by $z$ to saying that
$p^*/z$ belongs to $\bar{\cal D}(1,\delta^*(L) e^{Lx/2\nu})$. For any
$z$ in $\td$, either $\delta^*(L) e^{Lx/2\nu} < 1$ and thus $p^*/z$ is on the
inside of the disk $\bar{\cal
  D}(1,\delta^*(L) e^{Lx/2\nu})$ which is convex, or $\delta^*(L)
e^{Lx/2\nu} \ge 1$ and thus $p^*/z$ is outside of the disk
$\bar{\cal D}(\delta^*(L) e^{Lx/2\nu},1)$. Now since the circle with
$z_0=1$ cuts the $x$-axis only on the negative half line, see the
explicit calculation after (\ref{circle}), the outside of the
disk contains the half-plane $x \ge 0$, which is also convex.

Using the convexity, we can now show uniqueness: let $p^*$ and
$\tilde{p}^*$ be two solutions of the best approximation problem with
associated $\delta^*$. For a given $z$ in $\td$, in the first case,
$p^*/z $ and $\tilde{p}^*/z $ are both inside the disk, which is
convex. In the second case, they both belong to the half-plane $x \ge
0$, which is also convex, because by assumption
(\ref{eq:hypcoefficients}) the real part of $z$, and hence with the
properties on $p^*=p_1^*+ip_2^*$ also the real parts of $p^*/z$ and
$\tilde{p}^*/z$ are strictly positive. In both cases therefore, any
point $p/z$ in the segment joining $p^*/z $ and $\tilde{p}^*/z $ is
also in the disk $\bar{\cal D}(1,\delta^*(L) e^{Lx/2\nu})$, which
means that $\sup_{z\in \td}\left|\frac{z-p}{z+p}e^{-Lz/2\nu}\right|
\le \delta^*(L)$. Since $\delta^*(L)$ is the minimum, $p$ is also a
minimizer. To conclude the proof of uniqueness, we can use now Theorem
2.11 and the proof of Theorem 2.12 from \cite{Bennequin:2009:HBA},
using a classical equioscillation argument.

To see that the minimizer is real, we use again the symmetry of $\td$
with respect to the real axis, and the results on the strict local
minimum implying the global minimum follows as in the non-overlapping
case.

\end{proof}

\subsection{Ventcel Transmission Conditions}

For the case of Ventcel conditions, $q\ne 0$, we use the abstract
results from \cite{Bennequin:2009:HBA}.
\begin{theorem}\label{th:gennonoverlappq}
  For any set of coefficients such that the assumption
  \eqref{eq:hypcoefficients} is satisfied, and with $\kmm$ and $\wmm$
  finite, the min-max problem \eqref{BestApproxtilde} with $L=0$ has a
  unique solution $(\delta^*_1(0),p^*_1(0),q^*_1(0))$ with
  $\delta_1^*(0)<1$. The coefficients $p^*_1(0)$and $q^*_1(0))$ are
  real, and any strict local minimum in $\R_+\times\R_+$ of the real
  function
  \begin{equation}\label{BestApproxtildereal2}
    F_0(p,q)=\sup_{z\in \td_+}\left|\rho_0(z,p,q)\right|
  \end{equation}
  is the global minimum.
\end{theorem}
\begin{theorem}\label{th:genoverlappq}
  For any $L >0$, for $\kmm$ and $\wmm$ finite or not, and with the
  assumption (\ref{eq:hypcoefficients}) the min-max problem
  \eqref{BestApproxtildepq} has a solution.
  \begin{itemize}
     \item If $\td$ is compact and $L$ sufficiently small, the
       solution is unique and any strict local minimum of the real
       function
       \begin{equation}\label{BestApproxtildereal2overlap}
         F_L(p,q)=\sup_{z\in \td_+}\left|\rho(z,p,q,L)\right|
       \end{equation}
       is the global minimum.
    \item If $\td$ is not compact, but $L$ sufficiently small, if
      $F_L$ has a strict local minimum in $\R_+\times\R_+$, it is the
      unique global minimum.
  \end{itemize}
\end{theorem}

\subsection{Outline of the Analysis}

The abstract theorems in the previous subsections provide a guideline
for the proof of the main results in section \ref{sec:modelproblem}:
\begin{enumerate}
\item The existence and uniqueness is guaranteed by the
  abstract results.
\item The convergence factor being analytic on the compact $D$, its
  maximum is reached on the boundary. We thus study the variations of
  $R$ for fixed $p$ and $q$, on the exterior boundaries of
  $\td_+$. Due to the complexity of the problem, this study must be
  asymptotic, assuming asymptotic properties of $p$ and $q$.
\item There are two local maxima in the Robin case, and three local
  maxima in the Ventcel case. We prove that there exists a value
  $\bar{p}$ (resp. $(\bar{p},\bar{q})$) such that these two
  (resp. three) values coincide. The corresponding points $z$ are
  called equioscillation points.
\item We give the asymptotic values of these points and $\bar{p}$
  (resp. $(\bar{p},\bar{q})$).
\item We prove that $\bar{p}$ (resp. $(\bar{p},\bar{q})$) is a strict
  local minimizer for the function $F$.
\item We again invoke the abstract results to show that the strict
  local minimizer is in fact the global minimizer.
\end{enumerate}
Note that point 3 is not at all easy, since many cases have to be
analyzed.  We will treat the cases $\Delta t =C_h h$ and $\Delta t
=C_h h^2$ in the same paragraphs. But for the clarity of the paper, we
treat the Robin and Ventcel cases separately.

\subsection{Study of the Boundaries of the Frequency Domain}

The boundaries of $\td_+$ are all branches of the same function
$(\w,k)\mapsto z=x+iy$. Combining the equations \eqref{eq:systemxy},
we see that $x$, $y$ also satisfy the equation
\begin{equation}\label{eq:systemxy2b}
   x^2+y^2    = \sqrt{(\xoo+4\nu^2k^2)^2+16\nu^2 (\w+ck)^2},
\end{equation}
which, together with the constraints $ x \ge 0, y\ge 0$, gives us a
closed form parametric representation for $\td_+$:
\begin{equation}\label{eq:systemxy3}
  \left\{\begin{array}{lcl}
    x & = & \sqrt{
    \frac{1}{2}\sqrt{(\xoo+4\nu^2k^2)^2+16\nu^2 (\w+ck)^2}
    +\frac{1}{2}(\xoo+4\nu^2k^2)},\\
    y & = &\sqrt{
     \frac{1}{2}\sqrt{(\xoo+4\nu^2k^2)^2+16\nu^2 (\w+ck)^2}
     -\frac{1}{2}(\xoo+4\nu^2k^2)}.
  \end{array}\right.
\end{equation}
The boundary curves $\w\mapsto (x(\w,k), y(\w,k))$ for $k=\km$ or
$k=\kmm$ are hyperbolas, as one can see directly from
(\ref{eq:systemxy1}). They are shown in Figure \ref{fig:squarerootdomainscp},
\begin{figure}
  \centering
  \begin{tabular}{cc}
  \psfrag{z1}{\footnotesize $z_1$}
  \psfrag{z2}{\footnotesize $z_2$}
  \psfrag{z3}{\footnotesize $z_3$}
  \psfrag{z4}{\footnotesize $z_4$}
  \psfrag{cw}{\footnotesize $\ccw$}
  \psfrag{cn}{\footnotesize $\ccn$}
  \psfrag{csw}{\footnotesize $\ccsw$}
  \psfrag{ce}{\footnotesize $\cce$}
  \includegraphics[height=0.35\textwidth]{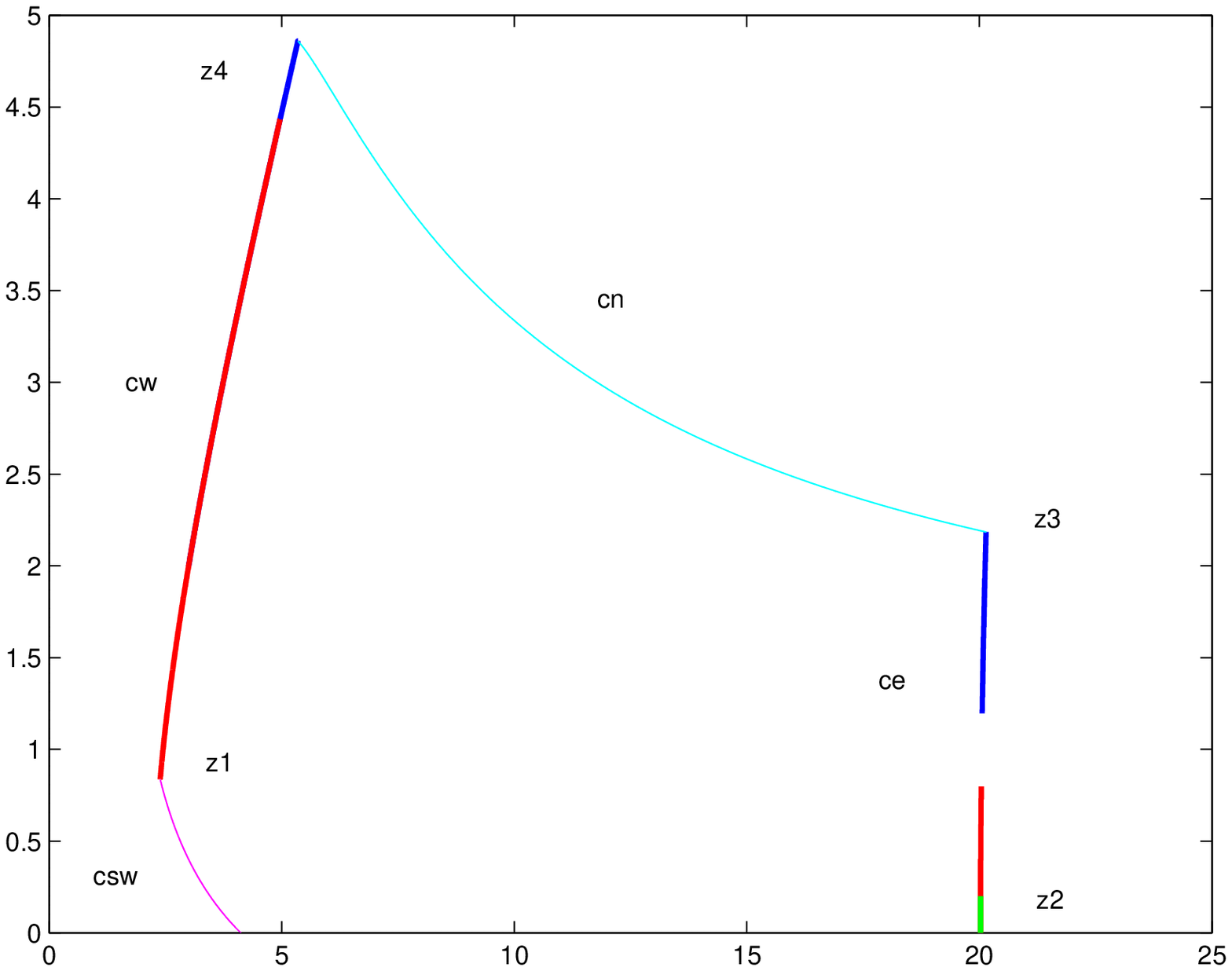}
  &
  \psfrag{z1}{\footnotesize $z_1$}
  \psfrag{z2}{\footnotesize $z_2$}
  \psfrag{z3}{\footnotesize $z_3$}
  \psfrag{z4}{\footnotesize $z_4$}
  \psfrag{cw}{\footnotesize $\ccw$}
  \psfrag{cn}{\footnotesize $\ccn$}
  \psfrag{csw}{\footnotesize $\ccsw$}
  \psfrag{cse}{\footnotesize $\ccse$}
  \psfrag{ce}{\footnotesize $\cce$}
  \includegraphics[height=0.35\textwidth]{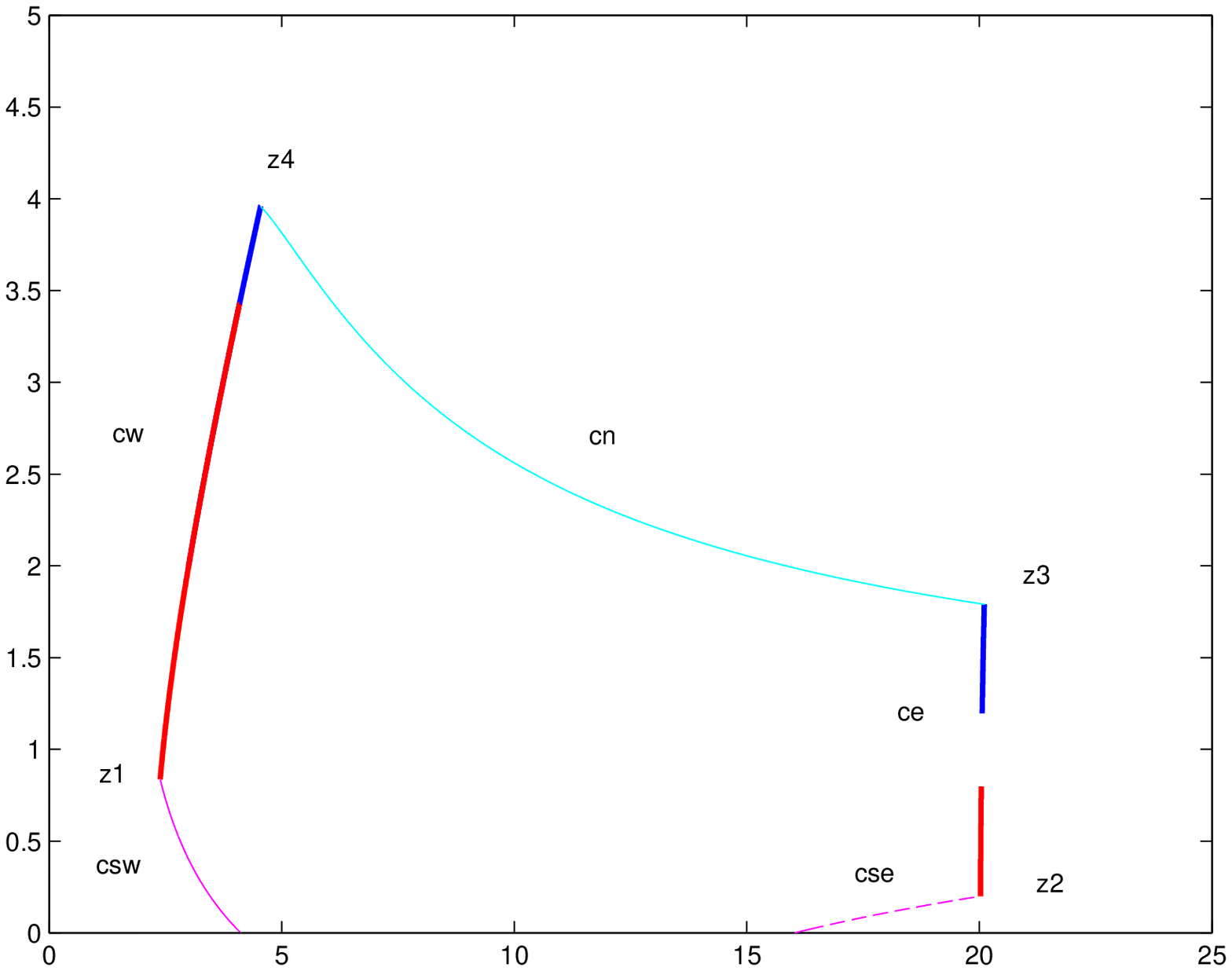}\\
  $|ck_m|< \omega_m,\ |ck_M|< \omega_M$ &
  $|ck_m|< \omega_m,\ |ck_M|> \omega_M$\\
  \psfrag{z1}{\footnotesize $z_1$}
  \psfrag{z2}{\footnotesize $z_2$}
  \psfrag{z3}{\footnotesize $z_3$}
  \psfrag{z4}{\footnotesize $z_4$}
  \psfrag{cw}{\footnotesize $\ccw$}
  \psfrag{cn}{\footnotesize $\ccn$}
  \psfrag{csw}{\footnotesize $\ccsw$}
  \psfrag{ce}{\footnotesize $\cce$}
  \includegraphics[height=0.35\textwidth]{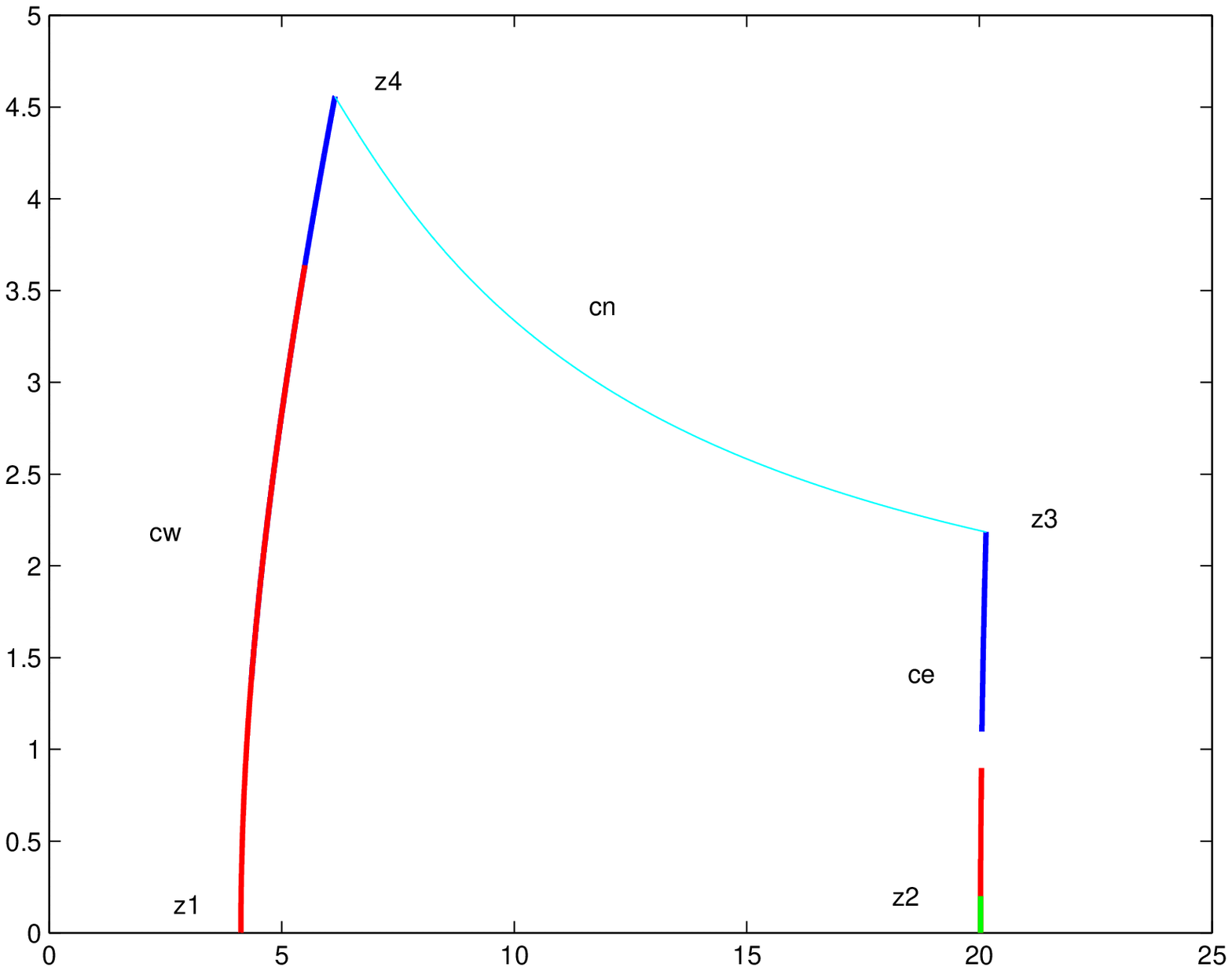}
  &
  \psfrag{z1}{\footnotesize $z_1$}
  \psfrag{z2}{\footnotesize $z_2$}
  \psfrag{z3}{\footnotesize $z_3$}
  \psfrag{z4}{\footnotesize $z_4$}
  \psfrag{cw}{\footnotesize $\ccw$}
  \psfrag{cn}{\footnotesize $\ccn$}
  \psfrag{csw}{\footnotesize $\ccsw$}
  \psfrag{cse}{\footnotesize $\ccse$}
  \psfrag{ce}{\footnotesize $\cce$}
  \includegraphics[height=0.35\textwidth]{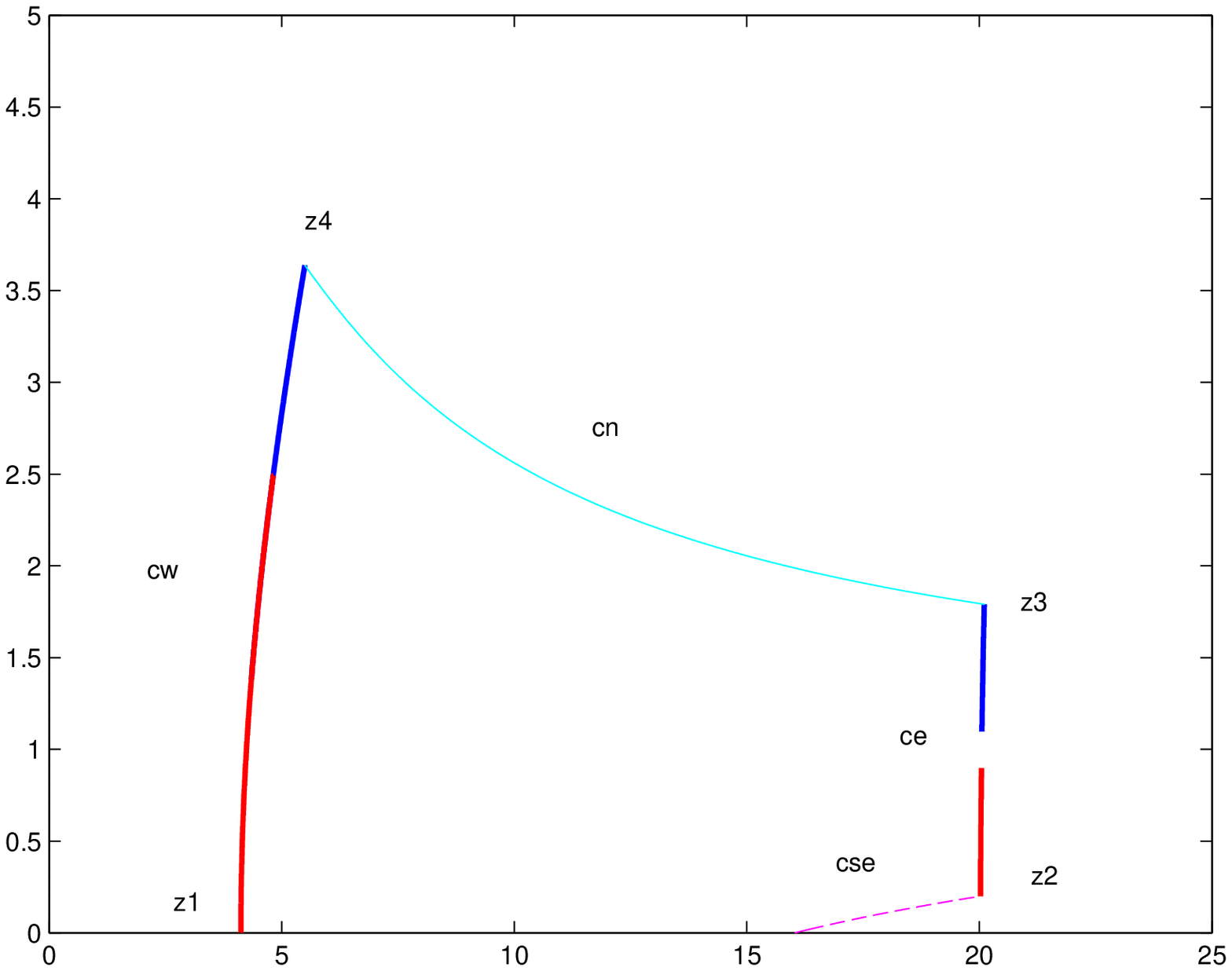}\\
  $|ck_m|> \omega_m,\ |ck_M|< \omega_M$&
  $|ck_m|> \omega_m,\ |ck_M|> \omega_M$
  \end{tabular}
  \caption{Illustration of the domain $\td_+$ in the $(x,y)$ plane}
  \label{fig:squarerootdomainscp}
\end{figure}
and using $\soc$ to denote the sign of $c$, the boundary on the left
(west) is given by
\begin{equation}\label{eq:hypwest}
\ccw=
\textcolor[rgb]{0.22,0.26,0.78}{z([\wm,\wmm],\soc\km)}
\cup \textcolor[rgb]{0.98,0.00,0.00}{z([\max(\wm,|c|\km),\wmm],-\soc\km)}
\end{equation}
and the boundary on the right (east) is given by
\begin{equation}\label{eq:hypeast}
\cce=
\textcolor[rgb]{0.98,0.00,0.00}{z([-min(|c|\kmm,\wmm),-\wm],\soc\kmm)}
\cup \textcolor[rgb]{0.22,0.26,0.78}{z([\wm,\wmm],\soc\kmm)}
\cup \textcolor[rgb]{0.00,0.50,0.00}{z([|c|\kmm,\wmm],-\soc\kmm)},
\end{equation}
with the convention that $[a,b]= \emptyset$ whenever $a > b$.
The corner points of $\td_+$ are
$$
\begin{array}{l}
z_1=z(\max(\wm,|c|\km),-\soc\km),\\
z_2= z(-\min(\wmm,|c|\kmm), \soc\km),\\
z_3=z(\wmm,\soc\kmm),\\
z_4=z(\wmm,\soc\km).
\end{array}
$$
In order to complete the boundary of $\td_+$, we analyze now the
curves at constant $\w$. The northern curve joins $z_3$ and $z_4$,
\begin{equation}\label{eq:hypnorth}
\ccn=
z(\wmm,\soc[\km,\kmm]).
\end{equation}
The southern curve can have two components, which are
\begin{equation}\label{eq:hypsouth}
\ccsw=
z(\wm,-\soc[\km,\frac{\wm}{|c|}]),\quad
\ccse=z(-\wmm,\soc[\frac{\wmm}{|c|},\kmm]).
\end{equation}

\begin{theorem}\label{lem:convexitycurve}
The curve $k\mapsto (x(\w,k), y(\w,k))$ has a vertical tangent in the
first quadrant if and only if $\w > 0$. It is reached for
\begin{equation}\label{eq:formulaforkbp}
  \tk_1(\w)=\frac{c}{8\nu^2\w}
  \left(\xoo+c^2-\sqrt{(\xoo+c^2)^2+16\nu^2\w^2}\right).
\end{equation}
It has a horizontal tangent in the first quadrant if and only if $\w >
0$. It is reached for
    \begin{equation}\label{eq:formulaforkb2}
  \tk_2(\w)=\frac{c}{8\nu^2\w}
  \left(\xoo+c^2+\sqrt{(\xoo+c^2)^2+16\nu^2\w^2}\right).
\end{equation}
For $\w c=0$, the curve is monotone.
\end{theorem}
\begin{proof}
We fix $\w$ and differentiate \eqref{eq:systemxy} in $k$ to obtain
\begin{equation}\label{eq:systemxpyp}
  \begin{pmatrix}
    x & -y\\
    y &  x
  \end{pmatrix}
  \hfill
  \begin{pmatrix}
    \partial_k x\\
    \partial_k y
  \end{pmatrix}
  =
  \hfill
   2\nu
  \begin{pmatrix}
    2\nu k\\
    c
  \end{pmatrix},
\end{equation}
or equivalently
\begin{equation}\label{eq:systemxpypsolved}
  \begin{pmatrix}
   \partial_k x\\
   \partial_k y
  \end{pmatrix}
  =\ds
  \frac{2\nu}{x^2+y^2}
  \begin{pmatrix}
    2\nu k x +cy \\
   -2\nu k y +cx
  \end{pmatrix}.
\end{equation}
We first search vertical tangent lines. From
\eqref{eq:systemxpypsolved}, we see that $\partial_k x=0$ if and only if
\begin{equation}\label{eq:cnstangentehoriz}
     2\nu k x +cy=0.
\end{equation}
Multiplying \eqref{eq:cnstangentehoriz} successively by $x$ and $y$
and substituting $xy$ from \eqref{eq:systemxy2} gives the system
\begin{equation}\label{eq:systemxpypsolved2}
  \begin{array}{lcl}
    x^2&=& -\ds\frac{c}{k}(\w+kc),\\
    y^2&=& -\ds 4\nu^2 \frac{k}{c}(\w+kc).
  \end{array}
\end{equation}
Replacing into the expression \eqref{eq:systemxy1} for $x^2-y^2$ gives
the equation for $kc$ (we keep $kc$ since $kc$ has a sign)
\begin{equation}\label{eq:equationforkb}
  Q_{\w}(kc):= \ 4\frac{\nu^2}{c^2}\w (kc)^2-(c^2+\xoo)(kc)-\w c^2=0.
\end{equation}
The polynomial $Q_{\w}$ has one negative solution $c\tk_1(\w)$, and
one positive solution $c\tk_2(\w)$, given in
(\ref{eq:formulaforkbp},\ref{eq:formulaforkb2}). For $k$ to yield a solution of
\eqref{eq:systemxpypsolved2} in $x >0, y > 0$, we must have $\w+kc >
0$ and $kc < 0$.  We compute $Q_{\w}(-\w )= \w (\xoo+4\frac{\nu^2
  \w^2}{c^2})$, which has the sign of the leading coefficient in
$Q_{\w}$. This proves that $-\w$ is outside the interval defined by the
roots, i.e.
$$
\begin{cases}
-\w < c\tk_1(\w)< 0 < c\tk_2(\w) &\mbox{if }\w >0,\\
 c\tk_1(\w)< 0 < c\tk_2(\w)< -\w &\mbox{if }\w <0.
 \end{cases}
$$
Therefore, $\w+c\tk_1(\w) > 0\iff \w > 0$, and there is a unique point where the
tangent is vertical, and this point is given by $k=\tk_1(\w)$.

We now search for horizontal tangent lines. By
\eqref{eq:systemxpypsolved}, we see that $\partial_k y=0$ if and only
if
\begin{equation}\label{eq:cnstangentevert}
  -2\nu k y +cx=0.
\end{equation}
Proceeding as before when we obtained (\ref{eq:systemxpypsolved2}), we
get the system
\begin{equation}\label{eq:cnstangentevert2}
\begin{array}{lcl}
  x^2&=& \ds 4\nu^2 \frac{k}{c}(\w+kc),\\
  y^2&=& \ds\frac{c}{k}(\w+kc),
\end{array}
\end{equation}
and $kc$, together with $\w+kc$, must be positive, which is the case
if $kc$ is the positive root of $Q_{\w}$, yielding $\tk_2$.
Therefore, there is a unique point where the tangent is horizontal,
which is given by $k=\tk_2(\w)$.

If $\w=0$ and $c \ne 0$,
a direct computation shows that
\[
  \partial_k x =
  \frac{4\nu^2 k(x^2+c^2)}{x(x^2+y^2)}\ > 0,\quad
  \partial_k y =
  \frac{2\nu c(x_0^2+y^2)}{x(x^2+y^2)}\ > 0,
\]
which implies that $\mbox{sign}(\partial_k x)=\mbox{sign}(k)$ and
$\mbox{sign}(\partial_k y)=\mbox{sign}(c)$. Since with $\omega=0$ we
have from (\ref{eq:systemxy2}) that $k$ and $c$ have the same sign,
and hence $\frac{dy}{dx}=\frac{\partial_k y}{\partial_k x}>0$, we obtain
that the curve is monotone.

Suppose now $c=0$, $\w \ne 0$. Using \eqref{eq:systemxpypsolved}, we
obtain directly $\frac{dy}{dx}=\frac{\partial_k y}{\partial_k
  x}=-\frac{y}{x}<0$, and again the curve is monotone.

Finally, if $c=\w=0$, we obtain from (\ref{eq:systemxy2}) that
$y(x)=0$, going from $x=\xo$ to infinity, which is also monotone.
\end{proof}

\begin{corollary}\label{cor:tangentes}
The northern curve $\ccn$ has a horizontal tangent, at
$\tz_2=z(\wmm,\tk_2(\wmm))$, if and only if $|\tk_2(\wmm)| \in
[\km,\kmm]$.

For $\km \le \wm/|c|$, the southern curve $\ccsw$ has a vertical
tangent, at $\tz_1=z(\wm,\tk_1(\wm))$, if and only if $|\tk_1(\wm)| \in
[\km,\wm/|c|]$.
\end{corollary}

\begin{proof}
  The results follow directly from Theorem \ref{lem:convexitycurve}.
\end{proof}

We show in Figure \ref{NewFig} an example where the two points $\tk_1$
and $\tk_2$ are part of $\tilde{D}_+$.
\begin{figure}
  \centering
  \psfrag{z1}{\footnotesize $z_1$}
  \psfrag{z2}{\footnotesize $z_2$}
  \psfrag{z3}{\footnotesize $\hspace{-2.5em}z_3$}
  \psfrag{z4}{\footnotesize $z_4$}
  \psfrag{cw}{\footnotesize $\ccw$}
  \psfrag{cn}{\footnotesize $\ccn$}
  \psfrag{csw}{\footnotesize $\ccsw$}
  \psfrag{ce}{\footnotesize $\cce$}
  \psfrag{zt3}{\footnotesize $\hspace{1.8em}\tz_1$}
  \psfrag{zt4}{\footnotesize $\tz_2$}
  \includegraphics[height=0.35\textwidth]{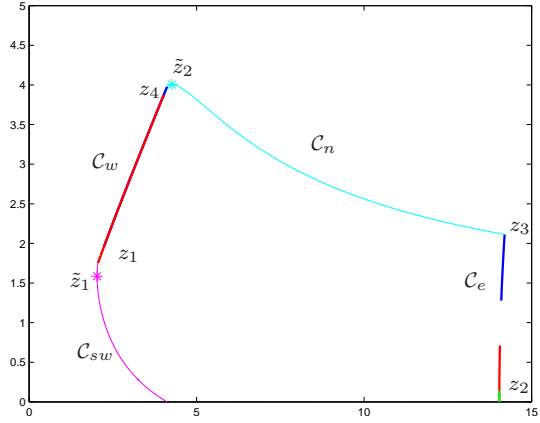}
  \caption{Illustration of the domain $\td_+$ in the $(x,y)$ plane
    with the two special points $\tz_1$ and $\tz_2$ defined in Corollary \ref{cor:tangentes}}
  \label{NewFig}
\end{figure}

Note that for $\wmm$ large, we have from (\ref{eq:formulaforkb2}) that
$$
  \tk_2(\wmm)= \frac{c}{2\nu}(1+\frac{\xoo+c^2}{4\nu\wmm})+\go(\wmm^{-2}).
$$
Therefore a sufficient condition for $\tz_2$ to belong to the northern
curve for $\omega_M$ large is $\km < \frac{|c|}{2\nu}$.

The next lemma gives the asymptotic expansions for the corner points
of $\td_+$, $z_1:=z(\wm,-\soc\km)$ if $\ac\km<\wm$ and $z_1:=z(\wm,-
\wm/c)$ if $\ac\km>\wm$, $z_3:=z(\wmm,\soc\kmm)$, and
$z_4:=z(\wmm,\soc\km)$, and also for other important points on
the boundary of $\td_+$.
\begin{lemma}\label{lem:cornerasymptotic}
The corner points $z_j$ of $\td_+$ have for $\kmm$ and
$\wmm$ large the asymptotic expansions
\begin{equation}\label{eq:asymptz3soc}
\begin{array}{rcl}
    z_1&=&\sqrt{x_0^2+4\nu^2\km^2+4i \nu\,
    \max( \wm-\ac \km,0))},
\\[3mm]
z_3  &\sim&\begin{cases}
    2\nu\kmm + i(\ac+ \frac{\wmm}{\kmm})
    &\mbox{if $\wmm \pr\kmm $,}\\
  2\nu\kmm
\sqrt{1+ i \frac{\wmm}{\nu\kmm^2}}
    &\mbox{if $\wmm\pr\kmm^2 $,}
\end{cases}
\\
  z_4&\sim & \sqrt{2\nu\wmm}(1+ i)
.
\end{array}
\end{equation}
\[
\]
We furthermore have the expansions for the horizontal tangent point
\[
  \tk_2(\wmm) \sim  \frac{c}{2\nu},\quad \tz_2(\wmm)\sim
   \sqrt{2\nu\wmm}(1+ i).
\]
\end{lemma}
\begin{proof}
  All expansions are obtained by direct calculations.
\end{proof}

We now define the south-western point and the northern point as
\begin{equation}\label{eq:defsw}
\begin{array}{rcl}
z_{sw}&=&
\begin{cases}
z_1 &\mbox{if } |c\km| < \wm
\mbox{
 or if } (|c\km| >  \wm \mbox{ and } |\tk_1(\wm)| \not\in [\km,\frac{\wm}{c}]), \\
\tz_1=z(\wm,\tk_1(\wm)) &\mbox{if } |c\km| >  \wm \mbox{ and } |\tk_1(\wm) | \in  [\km,\frac{\wm}{c}].
\end{cases}
\\[3mm]
z_{n}&=&
\begin{cases}
z_4 &\mbox{if } |\tk_2(\wmm)| \not\in [\km,\kmm] ,\\
\tz_2=z(\wmm,\tk_2(\wmm)) &\mbox{if }|\tk_2(\wmm)|  \in [\km,\kmm].
\end{cases}
\end{array}
\end{equation}

\section{Optimization of Robin Transmission Conditions}

This section is devoted to the proofs of Theorems \ref{th:toutno},
\ref{th:touto} and \ref{th:toutobounded}. The existence and uniqueness
of the minimizers are guaranteed by the abstract Theorems
\ref{th:gennonoverlap} and \ref{th:genoverlap}; we therefore focus
in each case on the characterization of a strict local minimum, which
will also provide the asymptotic results.

\subsection{The Nonoverlapping Case}

\paragraph{Proof of Theorem \ref{th:toutno} (Robin Conditions Without
  Overlap):} by Theorem \ref{th:gennonoverlap}, the best approximation
problem \eqref{BestApproxtilde} on $\td$ has a unique solution
$(p^*_0(0),\delta^*_0(0))$, which is the minimum of the real function
$F_0$ in \eqref{BestApproxtildereal}. To characterize this minimum, we
are guided by the geometric interpretation of the min-max problem: we
search for a circle containing $\td_+$, centered on the real positive
half line, and tangent in at least two points.  From numerical
insight, we make the ansatz that $p^*_0(0) \pr\sqrt{2\nu\kmm} $, which
we will validate a posteriori by the uniqueness result from Theorem
\ref{th:gennonoverlap}.

{\bf\em Local Maxima of the Convergence Factor:} We start by analyzing
the variation of $R_0(\w,k,p)=|\rho_0(\w,k,p)|^2$ on the boundary
curves $\cce$ ($k=\km$) and $\ccw$ ($k=\kmm$).
\begin{lemma}\label{lem:variationsw}
  For $\kmm$ large, and $ p \pr\sqrt{2\nu\kmm}$, we have
  \begin{enumerate}
  \item the maximum of $R_0$ on $\cce$
    is attained for $z=z_3$.
  \item the maximum of $R_0$ on $\ccw$
    is attained for $z=z_4$ or $z=z_1$.
\end{enumerate}
\end{lemma}
\begin{proof}
Computing the partial derivative of $R_0$ with respect to $\w$
using the chain rule, we obtain
\[
  \begin{array}{rcl}
    \partial_\w R_0(\w,k,p)
%
&=& \ds
     8\nu p y \,\frac{3x^2- y^2 -p^2}{|z(z+p)^2|^2},
  \end{array}
\]
which we rewrite, using the definitions of $x$ and $y$ in
\eqref{eq:systemxy3}, as
\begin{equation}\label{Rw}
  \partial_\w R_0(\w,k,p)= \frac{8p\nu y} {|z|^2\,|z+p|^4}\left(
    \sqrt{(\xoo+4\nu^2k^2)^2+16\nu^2 (\w+ck)^2}
    +2(\xoo+4\nu^2k^2)-p^2\right).
\end{equation}
We look now at the two boundary curves separately:
\begin{itemize}
\item $|k|=\kmm$: with the asymptotic assumptions, $p^2\ll
2(\xoo+4\nu^2\kmm^2)$, and the factor on the right is therefore
positive. Since $y$ is non-negative, $\partial_\w R_0(\cdot,\kmm,p)$
does not change sign, and the convergence factor $R_0$ is thus
increasing in $\w$. Its maximum is attained at $z_3$.

\item $|k|=\km$: the right hand side of \eqref{Rw} vanishes if $y=0$,
  which leads to a first root
  \[
    \w_1(k):=-c k,
  \]
  and also if the factor on the right in \eqref{Rw} vanishes, which
  happens if and only if
  \[
    \sqrt{(\xoo+4\nu^2k^2)^2+16\nu^2 (\w+ck)^2}=p^2-2 (\xoo+4\nu^2k^2),
  \]
where the right hand side is positive, since $|k|=\km$ and we have the
asymptotic assumption on $p$. By squaring, this equality is equivalent
to
\[
  16\nu^2 (\w+ck)^2 = (p^2-2(\xoo+4\nu^2k^2))^2-(\xoo+4\nu^2k^2)^2
   = (p^2-3(\xoo+4\nu^2k^2))(p^2-(\xoo+4\nu^2k^2)).
\]
Under the asymptotic assumption on $p$, the right hand side is
positive, and we can therefore obtain two further real roots
\[
  \begin{array}{l}
    \w_2(k):=-\ds c k +\frac{1}{4\nu}\sqrt{(p^2-2\z)(p^2-3\z)}, \\
    \w_3(k):=-\ds c k -\frac{1}{4\nu}\sqrt{(p^2-2\z)(p^2-3\z)}.
 \end{array}
\]
The three values $\w_j(\km)$, $j=1,2,3$, which lead to a vanishing
derivative, can be ordered, $\w_3(\km) < \w_1(\km) < \w_2(\km)$.
Looking at the behavior of the derivative of $R$ in \eqref{Rw} for
$\w$ large, we see that $\w_1(\km)$ must be a maximum, whereas
$\w_2(\km)$ and $\w_3(\km)$ represent minima.  For $k=-\soc\km$,
$\w_1(k)=\ac\km$ belongs to the western curve only if $\wm\le \ac\km$,
see (\ref{eq:hypwest}), and it is precisely on the boundary.  The
maximum of $R_0$ is therefore always attained on the boundary of the
western curve.
\end{itemize}
\end{proof}

We next analyze the variation of $R_0$ on the exterior boundary curves
of $\td_+$ when $\w$ is fixed. We start with the case $\omega=\omega_m$:
\begin{lemma}\label{lem:variationsk}
  For $\km\le \wm/\ac$, and large $p $, the derivative of $k\mapsto
  R(\wm,k,p)$ vanishes at a single point $\tk_3(p) \sim \tk_1(\wm)$,
  yielding a maximum at $\tz_3(p)=z(\wm,\tk_3(p))$, and
  $$
  \sup_{z\in\ccsw}R_0(z,p)=
  \begin{cases}
  R_0(z_1,p)&\mbox{if $|\tk_3(p)| \le \km$},\\
   R_0(\tz_3(p),p)&\mbox{if $|\tk_3(p)| \le \km$}.
  \end{cases}
  $$
\end{lemma}
\begin{proof}
As in the previous proof, we start by computing the partial derivative
\begin{equation}\label{Nw}
  \begin{array}{rcl}
  \partial_k R(\wm,k,p) &=& \ds 4 p\,
   \frac{(x^2-y^2-p^2)\partial_kx
    +2xy\partial_ky}{|z+p|^4}
    \ =\  8p\nu \ \frac{N_\w(k)}{|z|^2\,|z+p|^4},
     \\[3mm]
    N_\w(k)&=&\ds (x^2-y^2-p^2)(2\nu\,k\,x+cy)+2x\,y (-2\nu\,k\,y +cx).
\end{array}
\end{equation}
For $k$ in $-\soc[\km,\frac{\wm}{\ac}]$, $N_{\wm}(k)\sim -p^2
\partial_kx $ if $\partial_kx \ne 0$. If $|\tk_1(\wm)|\le \km$,
$\partial_k x$ has a constant sign in the interval, and $R_0(\wm,k,p)$
is a decreasing function of $x$, reaching therefore its maximum at
$z_1$. If $|\tk_1(\wm)|> \km$, $\partial_k x$ changes sign in
the interval, and so does $ N_{\wm}(k)$: there is a value
$\tk_3(p)\sim \tk_1(\wm)$ such that $N_{\wm}(\tk_3(p))=0$. At that
point $R_0$ is maximal.
\end{proof}
It finally remains to study the case were $\w=\wmm$.
\begin{lemma}\label{lem:definitionoftk4}
Suppose that $\wmm$ and $\kmm$ are large, with $\wmm \pr \kmm^{\alpha}
$, $\alpha=1$ or $2$, and $p\pr \sqrt{\kmm} $. If $p <
\sqrt{4\nu\wmm}$, $k\mapsto R(\wmm,k,p)$ has a single maximum at
$\tz_4=z(\wmm,\tk_4(\wmm,p),p)$. It is given asymptotically by
\begin{equation}\label{eq:valueofkb0}
   \tk_4(\wmm,p) \sim
   \begin{cases}
     \ds\frac{ c}{2\nu} \ \ds
     \frac{ 4\nu\wmm- p^2}{4\nu\wmm+p^2}
     &\mbox{if $\alpha=1$,}\\
     \ds \frac{c}{2\nu}
     &\mbox{if $\alpha=2$.}
  \end{cases}
\end{equation}
We then have the following two results:
\begin{enumerate}
 \item If $p > \sqrt{4\nu\wmm}$ or if $p < \sqrt{4\nu\wmm}$ and $|\km| \ge |\tk_4(\wmm,p)|$, then
     $$\ds\sup_{z\in\ccn}R_0(z,p)=\max (R_0(z_3,p),R_0(z_4,p)).$$
 \item If $p < \sqrt{4\nu\wmm}$ and $|\km| \le |\tk_4(\wmm,p)|$, then
   $$\ds\sup_{z\in\ccn}R_0(z,p)= \max
   (R_0(z_3,p),R_0(\tz_4(\wmm,p),p)).$$
\end{enumerate}
\end{lemma}
\begin{proof}
We study the variations of $N_{\wmm}$ defined in \eqref{Nw}, for $\soc
k\in[\km,\kmm]$. Since we are on $\ccn$, $k$ has the sign of $c$, see
(\ref{eq:hypnorth}), which implies that $\partial_k x $ has the sign
of $c$, as seen from (\ref{eq:systemxpypsolved}).  We now study
separately the two cases $\wmm\pr \kmm $ and $\wmm\pr \kmm^2 $:
\begin{description}
\item[$\bullet$]Case $\wmm\pr \kmm $: we need to study the three cases
  $k\pr\kmm^\alpha$ for $\alpha<\frac{1}{2}$, $\alpha=\frac{1}{2}$ and
  $\frac{1}{2}<\alpha<1$:
  \begin{description}
    \item[$\checkmark$ $k\pr \kmm^{\alpha} ,\ \alpha <\frac12$:] we
      obtain from \eqref{eq:systemxy3} that $x\sim y\sim
      \sqrt{2\nu\wmm}$, and \eqref{eq:systemxy1} shows that
      $x^2-y^2\sim \xoo+4\nu^2k^2\ll p^{2} $, which gives
      \[
        \begin{split}
          N_{\wmm}(k)&\sim \sqrt{2\nu\wmm}(-p^2(2\nu k + c )
          +4\nu\wmm(-2\nu k +c))\\
       & \sim \sqrt{2\nu\wmm}(-2\nu k(p^2 +4\nu\wmm)- c(p^2-4\nu\wmm) ).
        \end{split}
      \]
      Since $k$ has the same sign as $c$, this last quantity has the
      sign of $-c$ if $p> \sqrt{4\nu\wmm}$. $|\rho|$ is therefore a
      decreasing function of $x$. If $p < \sqrt{4\nu\wmm}$, the right
      hand side vanishes for
      $$
         k_0=\frac{c}{2\nu}\frac{4\nu\wmm-p^2}{4\nu\wmm+p^2}=\go(1).
      $$
      Therefore it has the sign of $c$ if $|k|\le |k_0|$, and the
      opposite sign otherwise. By the intermediate values theorem, $
      N_{\wmm}$ vanishes for $\tk_4\sim k_0$, where a local maximum
      occurs.
    \item[$\checkmark$ $ k\pr \kmm^\frac12$:] in this case,
       \[
          N_{\wmm}(k)\sim 2\nu k\wmm x(x^2-3y^2-p^2 )=
          2\nu k\wmm x
          (2(2\nu k)^2-\sqrt{(2\nu k)^4+(4\nu\wmm)^2} -p^2 ).
      \]
      The right hand side vanishes for
      $$
        k_0'=\frac{\soc}{2\sqrt{3}\nu}
          \sqrt{2p^2+\sqrt{p^4+3(4\nu\wmm)^2}} \pr \kmm^\frac12,
      $$
      and changes sign. Therefore, $N_{\wmm}$ vanishes for $\tk_4'\sim
      k'_0$, where a local minimum occurs.
    \item[$\checkmark$ $k\pr\kmm^\alpha $, $\frac12 < \alpha \le 1$:]
      In this case we see from \eqref{eq:systemxy1} that $x^2-y^2 \gg
      p^2$,
      $$
      z\sim\sqrt{4\nu^2k^2+4i\nu \wmm}
      \sim 2\nu |k| +i   \frac{\wmm}{|k|},
      $$
      and the leading order term in $N_{\wmm}$ is
      \[
         N_{\wmm}(k) \sim 4\nu^2k^2 (2\nu kx)
        +4\nu \wmm (-2\nu \wmm +2\nu c k)\sim (2\nu k)^4\soc.
      \]
    \end{description}
    In conclusion, if $ p^2\ge 4\nu\wmm$, $|\rho|$ has a single
    extremum, which is a minimum, and $\sup_{k\in\soc[\km,\kmm]}
    R_0(\wmm,k,p)=\max(R_0(\wmm,\soc\kmm,p),R_0(\wm,\soc\kmm,p))$.  If
    $ p^2\le 4\nu\wmm$, there is a maximum at $\tk_4\sim
    \frac{c}{2\nu}\frac{4\nu\wmm-p^2}{4\nu\wmm+p^2}$. If it is inside
    the segment, then $\sup_{k\in\soc[\km,\kmm]}
    R_0(\wmm,\cdot,p)=\max(R_0(\wmm,\soc\kmm,p),R_0(\wmm,\tk_4,p))$.

\item[$\bullet$]Case $\wmm\pr\kmm^{2} $: we study the cases
  $k\pr k_M^\alpha$ for $\alpha=0$, $0<\alpha<1$ and $\alpha=1$ separately:
  \begin{description}
    \item[$\checkmark$ $k\pr 1 $:] we have $x\sim y\sim
      \sqrt{2\nu\wmm}$, and in $N_{\wmm}$ the dominant term is
      $2xy(-2\nu k y +c x)$, which vanishes at $\tk_2(\wmm)$, from
      which we conclude that for $| k |< |\tk_2(\wmm)|$, $\soc
      N_{\wmm}(k)$ is positive, and negative for $| k
      |>|\tk_2(\wmm)|$. Therefore a local maximum is reached in the
      neighbourhood of $\tk_2(\wmm)$.
    \item[$\checkmark$ $k\pr\kmm^{\alpha} $, $0< \alpha < 1$:] we have
      again $x\sim y\sim \sqrt{2\nu\wmm}$, and the dominant term in
      $N_{\wmm}$ is $2x y (-2\nu k y)$, and
      \[
          N_{\wmm} \sim -8\nu \wmm k x  .
       \]
    \item[$\checkmark$ $k\pr\kmm $:] we have now $x\sim 2\nu|k|$,
      $y\sim \frac{\wmm}{|k|}\pr\kmm$, and the dominant term in
      $N_{\wmm}$ is
      \[
       N_{\wmm} \sim2\nu \frac{x}{k}(x^2-3y^2)  \sim2\nu k x(4\nu^2k^4-3\wmm^2) .
      \]
      Hence $\soc N_{\wmm}$ is negative for small $k$, and becomes positive for
      $k > \sqrt{\frac{\sqrt3}{2\nu}\wmm}$. $R_0(\omega_M,\cdot,p)$
      therefore reaches a minimum in the neighborhood of
      $\sqrt{\frac{\sqrt3}{2\nu}\wmm}$.
  \end{description}
  In conclusion, there is a maximum at $\tk_4\sim
  \tk_2(\wmm)\sim\frac{c}{2\nu}$. If this value is inside the segment,
  then $\sup_{k\in\soc[\km,\kmm]} R_0(\wmm,\cdot,p)=
  \max(R_0(\wmm,\soc\kmm,p),R_0(\wmm,\tk_4,p))$.  Otherwise
  $\sup_{\soc[\km,\kmm]}
  R_0(\wmm,\cdot,p)=\max(R_0(\wmm,\soc\km,p),R_0(\wmm,\soc\kmm,p))$.
\end{description}
The conclusion of the Lemma now follows directly from the conclusion of
the two cases.
\end{proof}
From the above analysis, we see that there are three local maxima of
$R_0(\w,k,p)$:
\begin{equation}\label{eq:pointsencompet}
\begin{array}{ll}
  \mbox{southwest}&\tz_{sw}=
  \begin{cases}
z_1 &\mbox{if } |c\km| < \wm,\\
z_1 &\mbox{if } |c\km| >  \wm \mbox{ and } |\tk_3(p)| \not\in [\km,\frac{\wm}{\ac}], \\
\tz_3(p) &\mbox{if } |c\km| >  \wm \mbox{ and } |\tk_3(p)| \in [\km,\frac{\wm}{\ac}],
\end{cases}
\\
\mbox{northwest}&
\tz_n=
\begin{cases}
z_4 &\mbox{if } p > \sqrt{4\nu \wmm},\\
z_4 &\mbox{if }  p < \sqrt{4\nu \wmm}\mbox{ and } |\tk_4(\wmm,p)| \not\in [\km,\kmm] ,\\
\tz_4(\wmm,p) &\mbox{if } p < \sqrt{4\nu \wmm} \mbox{ and } |\tk_4(\wmm,p)| \in [\km,\kmm],
\end{cases}
\\
  \mbox{northeast}&z_3,\\
\end{array}
\end{equation}
where $\tz_3$ comes from Lemma \ref{lem:variationsk} and $\tz_4$ comes
from Lemma \ref{lem:definitionoftk4}.

We investigate now the asymptotic behavior of the convergence factor
for large $\kmm$, in order to see which of the candidates of local
maxima $\tz_{sw}$, $\tz_n$ and $z_3$ will be important. Since
$\tz_{sw} \pr 1$, for $p\pr\sqrt{\kmm} $, the convergence factor at
$\tz_{sw}$ behaves asymptotically like
\[
  \rho_0(\tz_{sw},p)=\frac{\tz_{sw}-p}{\tz_{sw}+p}\sim -1 +2\frac{\tz_{sw}}{p},\quad
    |\rho(\tz_{sw},p)|\sim 1 -2\frac{x_{sw}}{p}.
\]
For $\tz_n$, we have $k\pr 1$ and $\w=\wmm$. Therefore $\tz_n\sim
\sqrt{2\nu\wmm}(1+i)$ and the convergence factor at $\tz_n$ behaves
asymptotically like
\[
  \rho_0(\tz_n,p)\sim\frac{1+i-\frac{p}{\sqrt{2\nu\wmm}}}
                        {1+i+\frac{p}{\sqrt{2\nu\wmm}}}.
\]
We thus need to distinguish two cases for $\rho_0(\tz_n,p)$:
\begin{enumerate}
  \item If $\wmm \pr \kmm $, $|\rho(\tz_n,p)|$ is asymptotically a
    constant smaller than 1, which shows that the modulus is smaller
    than $1$ independently of $\wmm$, and thus also independent of
    $\kmm$. Therefore, for $\kmm$ large enough, the convergence factor
    at $\tz_n$ is smaller than the convergence factor at $\tz_{sw}$,
    where it tends to $1$, and we do not need to take it into account
    in the min-max problem.

  \item If $\wmm \pr \kmm^2 $, then
    $\frac{p}{\sqrt{2\nu\wmm}}=\po(1)$, and the convergence factor at
    $\tz_n$ is asymptotically
    \[
      |\rho_0(\tz_n,p)|\sim 1-\frac{p}{\sqrt{2\nu\wmm}},
    \]
    which means it could be important in the min-max problem.
\end{enumerate}
We finally study the convergence factor at the last point $z_3$, and
again have to distinguish two cases:
\begin{enumerate}
  \item If $\wmm \pr \kmm $, $z_3\sim
    2\nu\kmm+i\frac{\wmm+\ac\kmm}{\kmm}$ and the convergence factor at
    $z_3$ behaves asymptotically like
    \[
      |\rho_0(z_3,p)|\sim\frac{x_3-p}{x_3+p}\sim 1- \frac{p}{\nu\kmm},
    \]
    which means it needs to be taken into account.
  \item If $\wmm =\frac{\nu \kmm^2}{d}$ then $z_3\sim
    2\nu\kmm\sqrt{1+\frac{i}{d}}$ and the convergence factor behaves
    asymptotically like
    \[
      |\rho_0(z_3,p)|\sim 1-\sqrt{\frac{d(d+\sqrt{1+d^2})}{2(1+d^2)}}\frac{p}{\nu\kmm},
    \]
    again possibly important for the min-max problem.
\end{enumerate}

{\bf\em Determination of the Global Minimizer by Equioscillation:} We now
compare the various points where the convergence factor can attain a
maximum, in order to minimize the overall convergence factor by an
equilibration process. We need to consider again the two basic cases
of an implicit or explicit time integration scheme:
\begin{enumerate}
  \item If $\wmm \pr\kmm $, for large $\wmm$, large $\kmm$ and
    $p\pr\sqrt{\kmm}$, the maximum of $|\rho_0|$ is reached at either
    $\tz_{sw}$ or $z_3$. We therefore consider the difference
    $|\rho_0(\tz_{sw},p)|-|\rho_0(z_3,p)|$, which is asymptotically
    equal to $2(\frac{p}{2\nu\kmm} -\frac{x_{sw}}{p})$. Depending on
    the relative values of $\frac{p^2}{2\nu\kmm}$ and $x_{sw}$, this
    difference can be positive or negative. Therefore, as a function
    of $p$, we can make it vanishes in the region $p \pr \sqrt{\kmm}$.
  \item If $\wmm =\frac{\nu \kmm^2}{d}$, then the point
    $\tz_n$ comes into play: we compute asymptotically the difference
\[
  |\rho_0(z_3,p)|-|\rho_0(\tz_n,p)|\sim \frac{p}{\nu\kmm}\sqrt{\frac{d}{2}} \left(1-\sqrt{\frac{d+\sqrt{1+d^2}}{1+d^2}}\right).
\]
The sign of this quantity is governed by the value of $d$ with
respect to $d_0$:
$$
\begin{cases}
\mbox{If $d > d_0$},|\rho_0(z_3,p)| > |\rho_0(\tz_n,p)|,\\
\mbox{If $d <d_0$},|\rho_0(z_3,p)| < |\rho_0(\tz_n,p)|.
\end{cases}
$$
Hence there is again a value of $p$ such that $|\rho_0
(\tz_{sw},p)|=\max(|\rho_0(z_3,p)|,|\rho_0(\tz_n,p)|)$.
\end{enumerate}
In order to obtain an explicit formula to equilibrate the convergence
factor at two maxima, we get after a short calculation that $|\rho_0|$
equioscillates at the generic points $Z_1$ and $Z_2$ (\textit{i.e.}
$|\rho_0(Z_1,p)|=|\rho_0(Z_2,p)|$) if and only if
$$
p=\sqrt{\frac{\Re Z_1|Z_2|^2-\Re Z_2|Z_1|^2}{\Re Z_2-\Re Z_1}}.
$$
Therefore we can define a unique $\bar{p}_0^*$ for both asymptotic
regimes by the equioscillation equations
\begin{equation}\label{eq:definitionpbar}
\begin{cases}
\wmm\pr\kmm &|\rho_0 (\tz_{sw},\bar{p}^*_0)=|\rho_0 (z_3,\bar{p}^*_0)|,\\
\wmm=\frac{\nu \kmm^2}{d} &
\begin{cases}
d>d_0 &|\rho_0(\tz_{sw},\bar{p}^*_0)|=|\rho_0(z_3,\bar{p}^*_0)|,\\
d<d_0 &|\rho_0(\tz_{sw},\bar{p}^*_0)|=|\rho_0(\tz_n,\bar{p}^*_0)|.\\
\end{cases}
\end{cases}
\end{equation}
In the first two cases, we get $
\bar{p}^*_0=\sqrt{\frac{\tx_{sw}|z_3|^2-x_3|\tz_{sw}|^2}{x_3-\tx_{sw}}}
$
and in the third case we obtain $
\bar{p}^*_0=\sqrt{\frac{\tx_{sw}|z_N|^2-x_N|\tz_{sw}|^2}{x_N-\tx_{sw}}}.
$
Since $\tz_{sw}$ is bounded, we obtain the asymptotic results
$$
\begin{cases}
\wmm\pr \kmm  &\bar{p}_0^*\sim \sqrt{\frac{x_{sw}|z_3|^2}{x_3}},\\
\wmm=\frac{\nu \kmm^2}{d} &
\begin{cases}
d>d_0 &\bar{p}_0^*\sim \sqrt{\frac{s\tx_{sw}|z_3|^2}{x_3}},\\
d<d_0 &\bar{p}_0^*\sim \sqrt{\frac{\tx_{sw}|\tz_n|^2}{\tx_n}},
\end{cases}
\end{cases}
$$
which imply
\begin{equation}\label{eq:formulasp}
\begin{cases}
\wmm\pr \kmm  &\bar{p}_0^*\sim \sqrt{2\nu\kmm x_{sw}},\\
\wmm=\frac{\nu \kmm^2}{d}  &
\begin{cases}
d>d_0  &\bar{p}_0^*\sim \sqrt{2\nu\kmm x_{sw} \sqrt{\frac{2(1+d^2)}{d(d+\sqrt{1+d^2})}} },\\
d<d_0 &\bar{p}_0^*\sim \sqrt{2\nu\kmm x_{sw}\sqrt{\frac{2}{d}}}.\\
\end{cases}
\end{cases}
\end{equation}
We now need to prove that the values of the Robin parameter
$\bar{p}_0^*$ we obtained by equioscillation are indeed local minima:
\begin{lemma}
For $\delta p$ sufficiently small and $p=\bar{p}_0^*+\delta p$
$$
F_0(p)-F_0(\bar{p}_0^*)=
\max(\delta p
\partial_p|\rho_0(\tz_{sw}(\bar{p}_0^*),\bar{p}_0^*)|,
\delta p
\partial_p|\rho_0(\tz_n(\wmm,\bar{p}_0^*),\bar{p}_0^*))|+
\po(\delta p).
$$

\end{lemma}
\begin{proof}
Consider for example the last case in \eqref{eq:definitionpbar}, when $\tz_{sw}=\tz_3(p)$ and $\tz_n=\tz_4(\wmm,p)$. By continuity,
$$
F_0(p)=\max(|\rho_0(\tz_3(p),p)|,
|\rho_0(\tz_4(\wmm,p),p)|).
$$
By the Taylor formula,
$$
\begin{array}{rcl}
|\rho_0(\tz_3(p),p)|&=&
|\rho_0(\tz_3(\bar{p}_0^*),\bar{p}_0^*)|
+ \delta p
(
\partial_p\tz_3(\bar{p}_0^*)
\partial_k|\rho_0(\tz_3(\bar{p}_0^*),\bar{p}_0^*))|
+
\partial_p|\rho_0(\tz_3(\bar{p}_0^*),\bar{p}_0^*))|
+\po(\delta p)\\
&=&
|\rho_0(\tz_3(\bar{p}_0^*),\bar{p}_0^*)|
+ \delta p
\partial_p|\rho_0(\tz_3(\bar{p}_0^*),\bar{p}_0^*)|
+\po(\delta p),
\end{array}
$$
since $\partial_k|\rho_0(\tz_3(\bar{p}_0^*),\bar{p}_0^*))|=0$. In the same way,
$$
|\rho_0(\tz_4(\wmm,p),p)|=
|\rho_0(\tz_4(\wmm,\bar{p}_0^*),\bar{p}_0^*)|
+ \delta p
\partial_p|\rho_0(\tz_4(\wmm,\bar{p}_0^*),\bar{p}_0^*)|
+\po(\delta p).
$$
Therefore
$$
F_0(p)-F_0(\bar{p}_0^*)=
\max(\delta p
\partial_p|\rho_0(\tz_3(\bar{p}_0^*),\bar{p}_0^*)|,
\delta p
\partial_p|\rho_0(\tz_4(\wmm,\bar{p}_0^*),\bar{p}_0^*))|+
\po(\delta p),
$$
which gives the lemma in this particular case. For the
case where the extremum is reached at a corner of the domain, the
argument is even simpler, since then no derivative in $k$ occurs.
\end{proof}

The derivative of $R_0$ in $p$ is given by
$$
\partial_p R_0(z,p)=\frac{-4x(|z|^2-p^2)}{|z+p|^4}.
$$
For $p=\bar{p}_0^*$, $z=\tz_{sw}$ , the numerator is equivalent to
$4xp^2$, whereas for $z=\tz_n$, it is equivalent to $-4x
|z|^2$. Therefore
$\partial_p|\rho_0(\tz_{sw}(\bar{p}_0^*),\bar{p}_0^*)|\times
\partial_p|\rho_0(\tz_{sw}(\wmm,\bar{p}_0^*),\bar{p}_0^*))| < 0$, and
$F_0(p)-F_0(\bar{p}_0^*) < 0$: $\bar{p}_0^*$ is a strict local minimizer
of $F_0$.

By Theorem \ref{th:gennonoverlap}, $\bar{p}_0^*$ is the global
minimizer, and therefore coincides with $ {p}_0^*(0)$.  In order to
conclude the proof of Theorem \ref{th:toutno}, we can replace in
\eqref{eq:formulasp} the term $x_{sw}$ by the notation $A/4$ from the
theorem, to obtain
$$
  \delta_0^*(L)=\left|\frac{\tz_{sw}-p}{\tz_{sw}+p}\right| \sim
    1-2\frac{x_{sw}}{p}=1-\frac{A}{2p}.
$$
The proof of Theorem  \ref{th:toutno} is now complete.

\subsection{The Overlapping Case} \label{sec:overlap}

We address now the two overlapping cases, and prove Theorem
\ref{th:touto} for the continous algorithm, and Theorem
\ref{th:toutobounded} for the discretized algorithm. By Theorem
\ref{th:genoverlap}, we know already that there is a unique minimizer
in both cases, which we now again characterize by equioscillation.

\paragraph{Proof of Theorem \ref{th:touto} (Robin Conditions with Overlap,
  Continuous):} we denote the unique minimizer of $F_L$ by
$p^*_{0,\infty}(L)$. As in the non-overlapping case, the maximum over
the whole domain is reached on the boundary $\cc=\ccwi\cup\ccsw$ of
$\td_+^\infty$, which is represented in Figure
\ref{fig:squarerootdomainsunbounded} for the three possible
configurations of the boundary.
\begin{figure}
\centering
\subfloat[$\km \le |\tk_1(\wm)|$]{
  \psfrag{z1}{\footnotesize $z_1$}
  \psfrag{z4cas1}{\footnotesize $z_4$ case 1}
  \psfrag{z4cas2}{\footnotesize $z_4$ case 2}
  \psfrag{tzp2}{\footnotesize $\tz'_2$}
  \psfrag{ztp3}{\footnotesize $\tz'_3$}
  \psfrag{zt3}{\footnotesize $\tz_3$}
  \psfrag{cw}{\footnotesize $\ccw$}
  \psfrag{csw}{\footnotesize $\ccsw$}
  \includegraphics[height=0.25\textwidth]{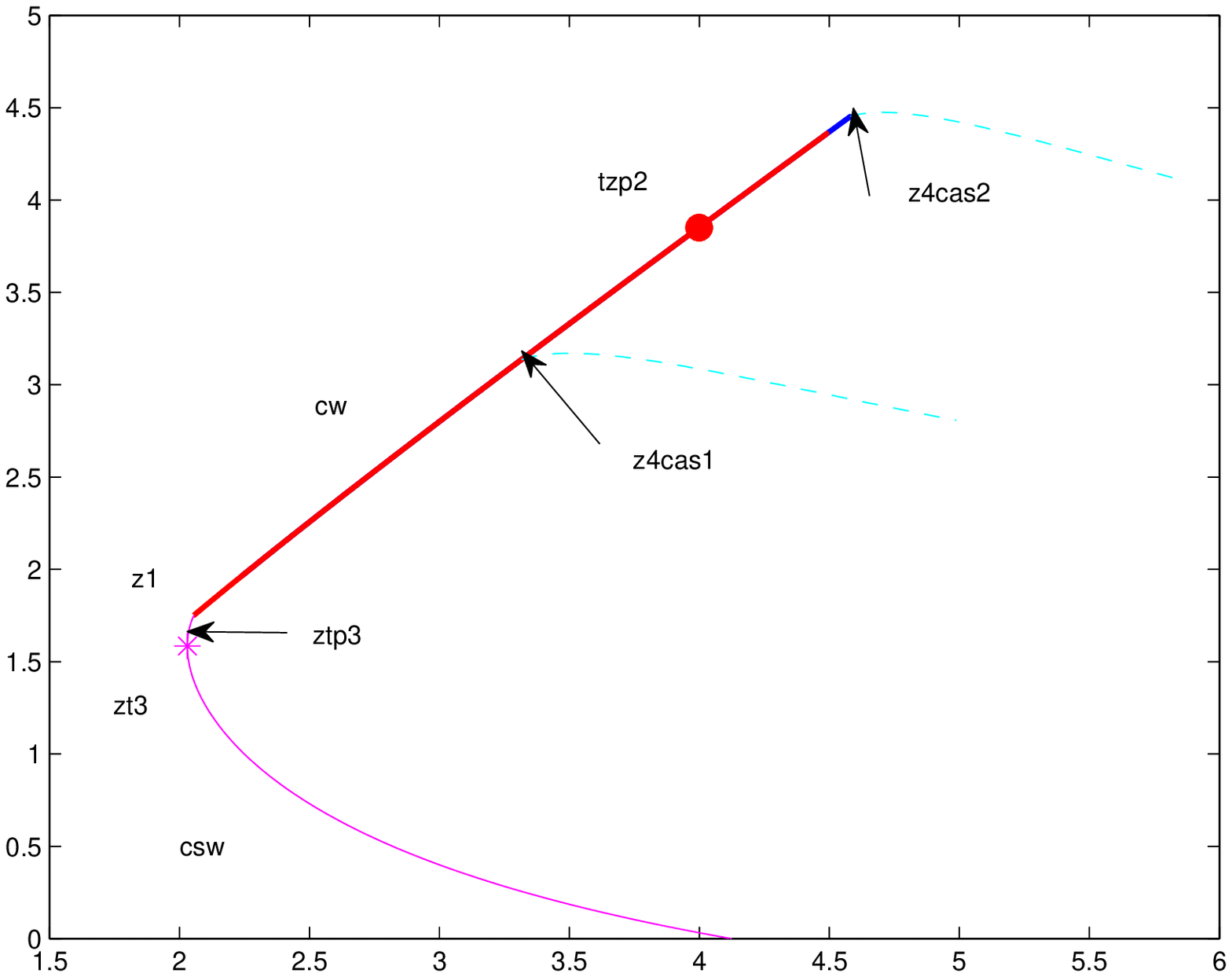}
  \label{fig:squarerootdomainsunbounded1}}
  \subfloat[$|\tk_1(\wm)| \le \km \le \frac{\wm}{\ac}$]{
  \psfrag{z1}{\footnotesize $z_1$}
  \psfrag{z4cas1}{\footnotesize $\ \ z_4$ case 1}
  \psfrag{z4cas2}{\footnotesize $\ \ z_4$ case 2}
  \psfrag{tzp2}{\footnotesize $\tz'_2$}
  \psfrag{ztp3}{\footnotesize $\tz'_3$}
  \psfrag{zt3}{\footnotesize $\tz_3$}
  \psfrag{cw}{\footnotesize $\!\!\ccw$}
  \psfrag{csw}{\footnotesize $\ccsw$}
  \includegraphics[height=0.25\textwidth]{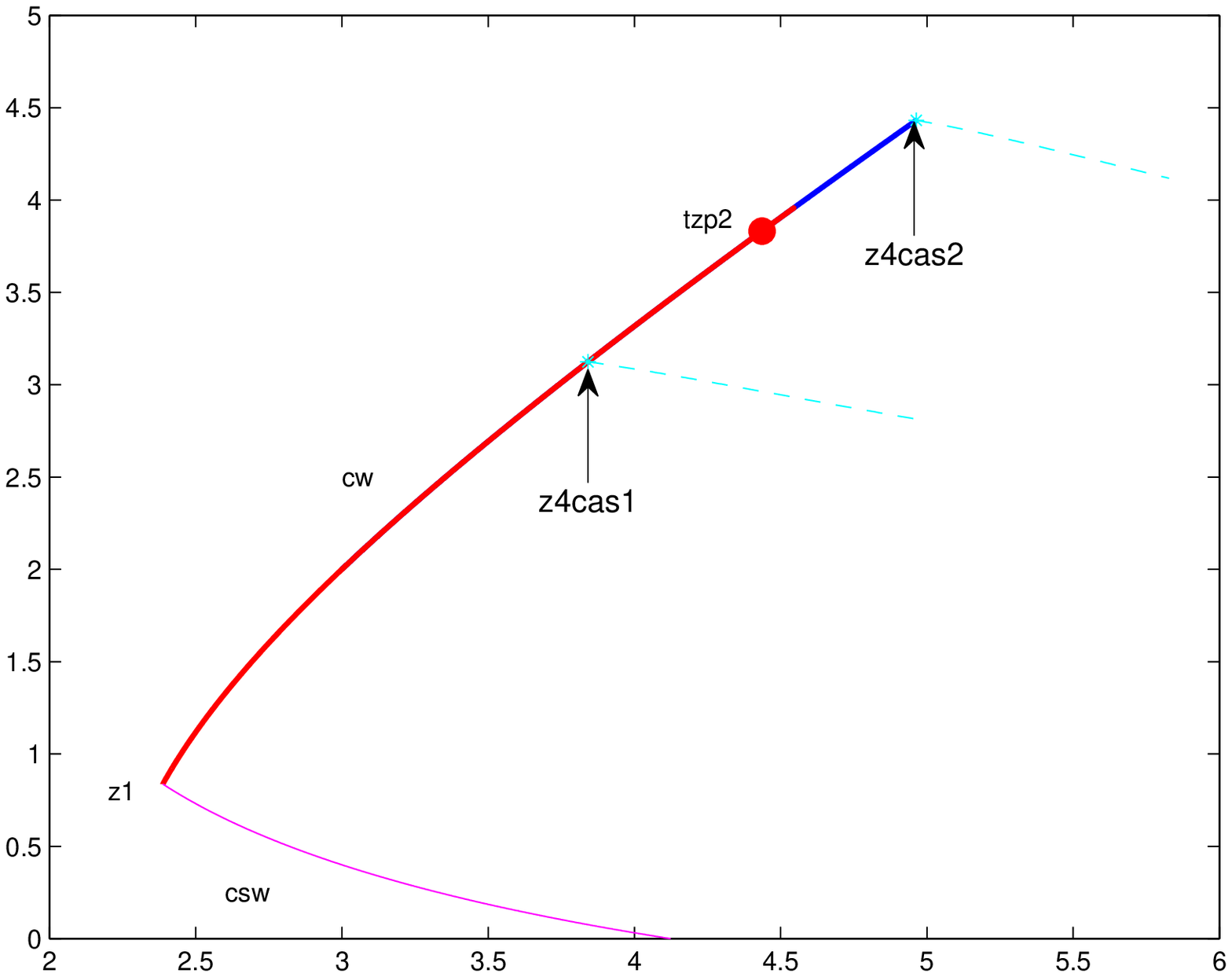}
  \label{fig:squarerootdomainsunbounded2}}
  \subfloat[$\km \ge \frac{\wm}{\ac}$]{
  \psfrag{z1}{\footnotesize $z_1$}
  \psfrag{z4cas1}{\footnotesize $\ \ z_4$ case 1}
  \psfrag{z4cas2}{\footnotesize $\ \ z_4$ case 2}
  \psfrag{tzp2}{\footnotesize $\!\!\!\tz'_2$}
  \psfrag{ztp3}{\footnotesize $\tz'_3$}
  \psfrag{zt3}{\footnotesize $\tz_3$}
  \psfrag{cw}{\footnotesize $\!\!\!\!\!\ccw$}
  \psfrag{csw}{\footnotesize $\ccsw$}
  \includegraphics[height=0.25\textwidth]{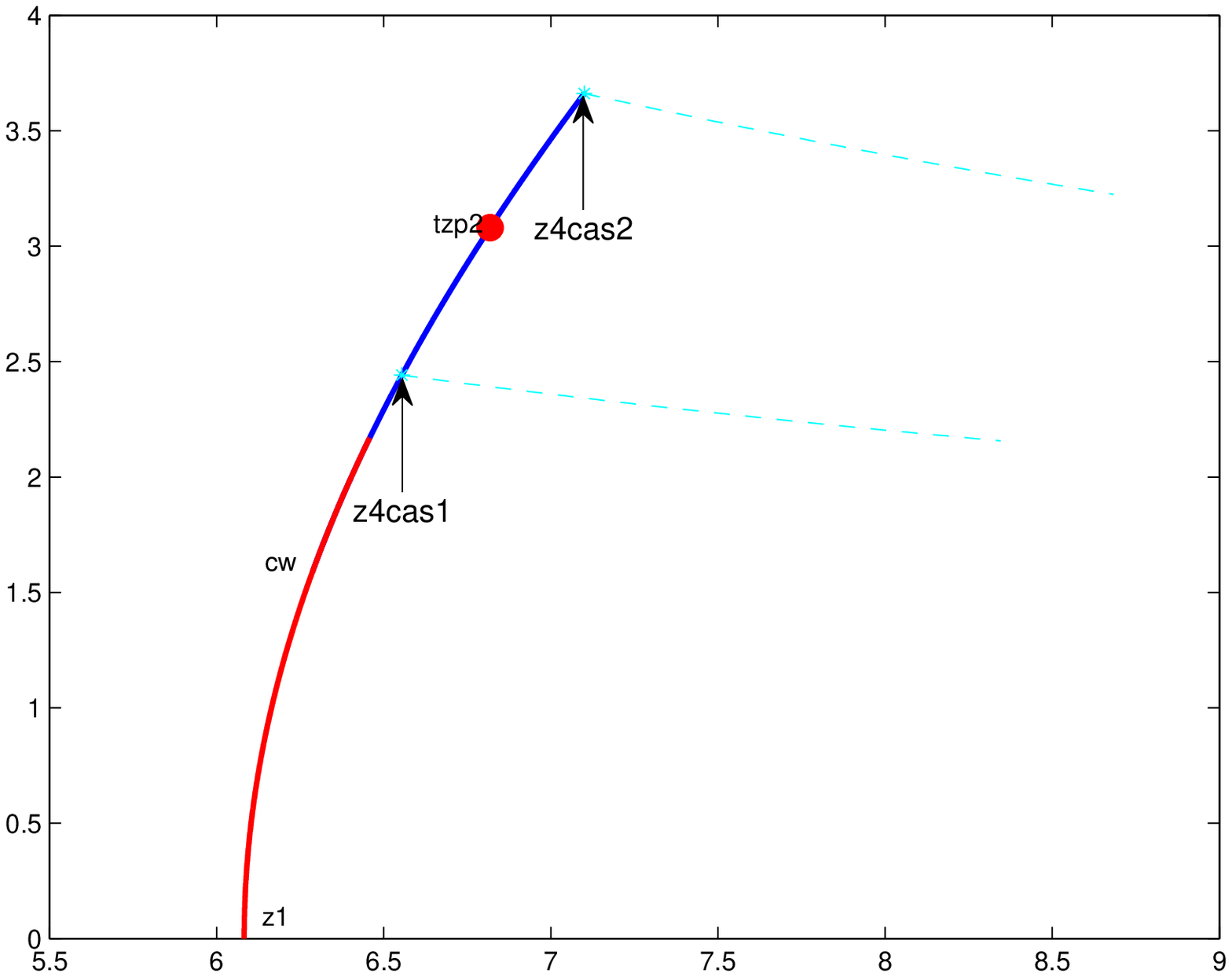}
  \label{fig:squarerootdomainsunbounded3}}
  \caption{Illustration of the domain $\td_+^\infty$ in the $(x,y)$ plane}
  \label{fig:squarerootdomainsunbounded}
\end{figure}
In order to simplify the notation, we use $l:=\frac{L}{2\nu}$. We
start with the variations of the convergence factor
\begin{equation}\label{eq:convfactoroverlap}
R(\w,k,p,\ell)=R_0(\w,k,p)e^{-2\ell x}
\end{equation}
on the west boundary $ \ccwi$. Calculating the partial derivative of
$R$ with respect to $\w$ leads to
\begin{equation}\label{eq:deriveomegaconvfactoroverlap1}
\begin{array}{rcl}
  \partial_\w R(\w,k,p,\ell) & = &
(\partial_\w R_0(\w,k,p)-2\ell R_0(\w,k,p)\partial_\w x(\w,k))e^{-2\ell x}\\
   & = &
\ds\frac{4\nu y}{|z|^2|z+p|^4} S_{k}(x,y,p,\ell),
\end{array}
\end{equation}
where we introduced the function
\[
  S_{k}(x,y,p,\ell)=2p(3x^2-y^2-p^2) -\ell \,|z^2-p^2|^2=
    2p(3x^2-y^2-p^2) -\ell \,[(x^2-y^2-p^2)^2 +4x^2y^2].
\]
The root $y=0$ of $\partial_\w R(\wm,k,p,\ell)$ corresponds to
$\w=-c\km$, which is possible only if $|\wm|\le |c\km|$.

We study now $S_{\km}(x,y,p,\ell)$.  Replacing $y^2=x^2-\alpha^2
=x^2-\xoo-4\nu^2\km^2$ from (\ref{eq:systemxy1}), we get
\[
  \tilde{S}_{\km}(x,p,\ell):=
   2p(2x^2+\alpha^2-p^2)-\ell \,((\alpha^2-p^2)^2 +4x^2(x^2-\alpha^2)),
\]
which is now a second order polynomial in $x^2$,
\begin{equation}\label{eq:deriveomegaconvfactoroverlap2}
  \tilde{S}_{\km}(x,p,\ell)=-4\ell x^4+4(\alpha^2\ell+p)x^2
    -(p^2-\alpha^2)(2p+\ell(p^2-\alpha^2)).
\end{equation}
The following lemma gives the asymptotic behavior of the roots of this
polynomial:
\begin{lemma}\label{lem:asymptoverlapinfini}
For small $\ell$, large $p$ with $\ell p$ small,
$\tilde{S}_{\km}(x,p,\ell)$ has two distinct real roots,
$$
\tx_1'(p,\ell) \sim \frac{p}{\sqrt{2}},\quad
\tx_2'(p,\ell)\sim \ds\sqrt{\frac{p}{\ell}}.
$$
The first root is the real part of a minimum of the convergence
factor, and the second root is the real part of a maximum of the
convergence factor, say at $\tz'_2$. We thus obtain that
\[
  \ds\sup_{z\in\ccwi} |\rho(z,p,\ell)|=\max(|\rho(z_1,p,\ell)|,|\rho(\tz'_2(p,\ell),p,\ell)|).
\]
\end{lemma}
\begin{proof}
The discriminant of the second degree polynomial $\tilde{S}_{\km}$ and
its leading asymptotic part under the conditions of Theorem
\ref{th:touto} are
 \[
  \ds \Delta=  4(\Delta_a + 2\alpha^2\ell p\, (2+\ell p)); \qquad
  \Delta_a=4p^2(1-2\ell p-\ell^2 p^2).
\]
Since $\Delta \sim \Delta_a$, $\tilde{S}_{\km}$ has two roots with
asymptotic behavior
\[
  \tx'_1 \sim \frac{p}{\sqrt{2}},\quad \tx'_2\sim \ds\sqrt{\frac{p}{\ell}}.
\]
For $\tz'_j=\tx'_j+i\sqrt{(\tx_j')^2-(\xoo+4\nu^2\km^2)}$, which we
obtain from (\ref{eq:systemxy1}), we compute
\[
|\rho(\tz'_1,p,\ell) |\sim \sqrt{\frac{1-\sqrt{2}}{1+\sqrt{2}}},\qquad
|\rho(\tz'_2,p,\ell) |\sim 1-2\sqrt{p\ell},
\]
and $|\rho(\tz'_1,p,\ell)| < |\rho(\tz'_2,p,\ell) |$ for small $\ell p$.
\end{proof}

We analyze now the cases in Figure
\ref{fig:squarerootdomainsunbounded} in detail:
\begin{itemize}
  \item Figure \ref{fig:squarerootdomainsunbounded3}, $|\wm| < |c\km|$:
  As $\w$ runs through $\R$, $z$ runs through the full hyperbola, and
$
\sup_{z\in\tilde{\cal D}_+^{\infty}} |\rho(z,p,\ell)|=\max(|\rho(z_1,p,\ell)|,|\rho(\tz_2,p,\ell)|).
$

  \item Figure \ref{fig:squarerootdomainsunbounded1} and
    \ref{fig:squarerootdomainsunbounded2}, $|\wm| > |c\km|$: to study
    the variation of $R$ on $\ccsw=z(\wm,-\soc[\km,\frac{\wm}{\ac}])$,
%
%
%
we compute
\begin{equation}\label{eq:derivekconvfactoroverlap1}
\begin{array}{c}
\partial_k R(\w,k,p,\ell)=
\ds\frac{4\nu}{|z|^2|z+p|^4} S_{\w}(x,y,p,\ell), \\
S_{\w}(x,y,p,\ell):=
2p\{(2\nu k x + c y)(x^2-y^2-p^2)+2xy(-2\nu k y + c x)\}
 -\ell \,(2\nu k x + c y)|z^2-p^2|^2.
\end{array}
\end{equation}
With the same assumptions as in the previous lemma, for any $z$ in $\ccsw$,
\[
S_{\wm}(x,y,p,\ell)\sim
-2p^3(1+\ell p^2)(2\nu k x + c y)=
-2p^3(1+\ell p^2)\partial_k x.
\]
In case of Figure \ref{fig:squarerootdomainsunbounded2}, where
$|\tk_1(\wm)| \le \km $, $\partial_k x$ has a constant sign on the
curve $\ccsw$, see the second case in Corollary \ref{cor:tangentes},
and hence the maximum of $R$ is reached at $z_1$.  In case of Figure
\ref{fig:squarerootdomainsunbounded1}, where $\km \le|\tk_1(\wm)|$,
$\soc S_{\wmm}$ is positive for $\km\le |k| < \tk_1(\wm)$, and
negative for $|k|> \tk_1(\wm)$. It must therefore vanish in a
neighborhood of $\tk_1(\wm)$, where $R$ has a maximum on $\ccsw$, at a
point we call $\tz'_3(p,\ell):=z(\wm,\tk'_3(p,\ell))$, which is
asymptotically equivalent to $\tz_1$ where the vertical tangent
occurs.
\end{itemize}
We now define the point $\tz_{sw}'(p,\ell)$ by
$$
\tz_{sw}'(p,\ell)=
\begin{cases}
z_1&\mbox{if $\wm < \ac\km$ or  $|\tk'_3(p,\ell)| \le \km \le \frac{\wm}{\ac}$},\\
\tz'_3(p,\ell)&\mbox{if   $\km\le|\tk'_3(p,\ell)| \le \frac{\wm}{\ac}$},
\end{cases}
$$
in order to write in compact form
$$
\ds\sup_{z\in\td_+^{\infty}} |\rho(z,p,\ell)|=
\max(|\rho(\tz_{sw}'(p,\ell),p,\ell)|,|\rho(\tz'_2(p,\ell),p,\ell)|) .
$$
Using the asymptotic expansions of $|\rho(\tz'_2,p,\ell)|$ above, and
$|\rho(\tz_{sw}',p,\ell)|\sim 1-2\frac{x_{sw}}{p}$, we see that for
small $\ell$,
$$
  |\rho(\tz'_2,p,\ell)|- |\rho(\tz'_{sw},p,\ell)|\sim 2(\frac{x_{sw}}{p}-\sqrt{p\ell}).
$$
This quantity is positive for $p$ smaller than
$\sqrt[3]{\frac{x_{sw}}{\ell}}$, and negative otherwise. Therefore it
vanishes for one single value of $p$, and we have asymtotically
\begin{equation}\label{eq:poptinf}
  \bar{p}^*_\infty\sim\sqrt[3]{\frac{x_{sw}^2}{\ell}}, \quad
  F_L(\bar{p}^*_\infty)\sim 1-2\sqrt[3]{\ell x_{sw}} .
 \end{equation}
We verify that $\ell \bar{p}^*_\infty$ tends to zero with $\ell$, thus
justifying all previous computations.

The proof can now be completed like for the previous theorem, showing
that $\bar{p}^*_\infty $ is a strict local minimizer and therefore
coincides with the global minimizer ${p}^*_\infty$ according to the
abstract result.

\paragraph{Proof of Theorem \ref{th:toutobounded} (Robin Conditions with
  Overlap, Discrete):} In order to prove the
results for the discretized algorithm, suppose $\ell\pr\kmm^{-1}$, $p $
large , with $\ell p$ small as in Lemma \ref{lem:asymptoverlapinfini}. The
maximum at $\tz'_2$ on ${\cal C}_w^+$ is on the curve ${\cal C}_w$ if
$\tx_2'\sim\sqrt{\frac{p}{\ell}} < x_4\sim \sqrt{2\nu\wmm} $. We see
that
$$
\frac{\tx_2'}{x_4} \sim\sqrt{\frac{p}{2\nu\ell\wmm}}
\pr
\left\{\begin{array}{lll}
\sqrt{p}  &  \gg 1
&\mbox{if $\wmm \pr\kmm $}, \\
\sqrt{\frac{p}{\kmm}} &\ll  1
&\mbox{if $\wmm\pr\kmm^2 $},
\end{array}
\right.
$$
which indicates that the continuous analysis will only be important in
the second case. We study now both cases in detail:
\begin{itemize}
  \item $\wmm\pr\kmm^2 $: Let $p \pr p^*_{0,\infty(L)}$. An asymptotic
    study shows that the derivative in $\w$ on the eastern curve
    $|k|=\kmm$ satisfies
    $$
      S_{\kmm}(z,p,\ell)\sim -\ell( 4 \nu^2\kmm^2 +16\nu^2\w^4) < 0.
    $$
    Therefore the maximum of $|\rho|$ on the east is reached at
    $z_3=z(\wmm,\soc\kmm)$. The same study on the north gives
    $$
      S_{\wmm}(z,p,\ell)\sim -\ell\partial_kx ( (4 \nu^2k^2 )^2+16\nu^2\wmm^4) .
    $$
    The sign of $S_{\wmm}(z,p,\ell)$ is the opposite of the sign of
    $x$, the maximum of $|\rho|$ on $\ccn$ is therefore reached at
    $z_4$.  From this we conclude that all values of $|\rho|$ on
    $\ccn$ and $\cce$ are smaller than the value at $z_4$.  We now
    study the variations of $R$ on the other boundaries. Since $p \pr
    p^*_{0,\infty(L)}$, the conclusions from Lemma
    \ref{lem:asymptoverlapinfini} and after are all valid, there is a
    unique value $\bar{p}^*(\ell)$ of $p$ such that
    $|\rho(\tz'_2,p,\ell)|= |\rho(\tz_{sw}',p,\ell)|$. It is for
    $\ell=\frac{L}{2\nu}$ small asymptotically equivalent to
    $p^*_{0,\infty}(L)$.

  \item $\wmm\pr\kmm$:
    We perform the asymptotic analysis in $\kmm$, assuming $p \ll
    \sqrt{\wmm}$, and study the behavior of the convergence factor on
    all four boundary curves $\ccw$, $\cce$, $\ccsw$ and $\ccn$:


    \begin{description}

    \item[Behavior of $R$ on $\ccw$:] $\tx_2' \gg x_4$, and $R$ has no
      local maximum on ${\cal C}_w$. Therefore
      $$
        \max_{{\cal C}_w}R=\max(R(z_1),R(\tz'_2)).
      $$

    \item[Behavior of $R$ on ${\cal C}_e$:] Since $p \ll \kmm$, using
      that $x\sim2\nu\kmm$, we obtain
      $$
        S_{\kmm}(x,y,p,\ell) \sim -\ell(2\nu\kmm)^4.
      $$
      The maximum of $R$ on the eastern side is therefore reached for
      $z=z_3$.

    \item[Behavior of $R$ on $\ccsw$:] The behavior of $R$ on the
      southern part remains unchanged: for $\km\le \wm/\ac$,
      $p=\go(\sqrt{2\nu\kmm}) $, the maximum of $R(\wm,\cdot,p)$ on
      $-\soc(\km,\wm/\ac)$ is reached at the single point
      $\tz''_3(p,\ell)=z(\wm,\tk''_3(p,\ell))$, whose asymptotic
      behavior is given by $ \tk''_3(p,\ell)\sim \tk_1(\wm).  $ The
      proof is similar to that of Lemma \ref{lem:variationsk}.\\

    \item[Behavior of $R$ on ${\cal C}_n$:] We extend the analysis in
      the proof of Lemma \ref{lem:definitionoftk4} to $S_{\wmm}$ in
      \eqref{eq:derivekconvfactoroverlap1}. The variations of $R$
      are determined by the sign of
    \[
      \begin{split}
       S_{\wmm}(k)&=N_{\wmm}(k)-\frac{\ell}{2p} |z^2-p^2|^2)\,(2\nu k x + c y)\\
       &=2p\left[
   (\xoo+4\nu^2k^2 -p^2 -\frac{\ell}{2p} ((\xoo+4\nu^2k^2 -p^2 )^2+16\nu^2(\wmm+ck)^2))\,(2\nu k x + c y)  \right. \\  &\hspace{20mm}
    \left.+4\nu(\wmm+ck)(-2\nu k y + c x)\phantom{^\frac11}\hspace{-1mm}\right].
       \end{split}
     \]
     Again we have to distinguish three cases for $k\pr
     k_M^\alpha$: $\alpha\le\frac{1}{2}$, $\frac{1}{2}<\alpha<1$ and $\alpha=1$:
  \begin{itemize}
    \item[$\checkmark$] $k=\go(\kmm^\frac12)$: in this case
      $S_{\wmm}(k) \sim 2p N_{\wmm}(k)$, and therefore on the curve
      ${\cal C}_n$, $S_{\wmm}$ vanishes for $\tk'_4\sim \tk_4$ under
      the conditions of case 2 in Lemma \ref{lem:definitionoftk4}, and $R$
      has a maximum there. For $k''_0\pr \sqrt{\kmm}$, $R$ has
      a minimum.

    \item[$\checkmark$] For $k\pr \kmm^\alpha$ with $\frac12< \alpha < 1$, the
      overlap comes into play. We have
      \[
        S_{\wmm}(k) \sim 2p (2\nu k)^4\soc (1-\frac{\ell}{2p}(2\nu k)^2).
      \]
      The right hand side vanishes for $2\nu
      k=\sqrt{\frac{2p}{\ell}}$, and $S_{\wmm}(k)$ vanishes therefore
      in a neighbourhood of that point,
      $$
        \tk_4''\sim\frac1{2\nu }\sqrt{\frac{2p}{\ell}},
      $$
      which corresponds to a maximum of $R$ again.

      \item[$\checkmark$] For $k\pr \kmm $, the overlap dominates, and
        $ S_{\wmm}(k) \sim -\ell (2\nu k)^4\soc $.
  \end{itemize}
  Therefore, there are two local maxima on the curve $\ccn$, and we must
  compare $|\rho|$ at $\tz_n$ defined in \eqref{eq:pointsencompet},
  $$
    |\rho(\tz_n,p)|\sim \left|\frac{(1+i)\sqrt{2\nu\wmm}-p}
      {(1+i)\sqrt{2\nu\wmm}-p}e^{-\ell(1+i)\sqrt{2\nu\wmm}}\right|
    \sim 1 - \frac{p}{\sqrt{2\nu\wmm}},
  $$
  and $|\rho|$ at $\tz''_4=z(\tk''_4,\wmm)$,
  $$
    \tz''_4\sim 2\nu|\tk''_4|\left(1+i\frac{\wmm}{\nu (\tk''_4)^2}\right)
      \sim\sqrt{\frac{2p}{\ell}} (1+i\frac{\nu \ell\wmm}{p})
      \sim \sqrt{\frac{2p}{\ell}},
  $$
  which gives for $|\rho|$ at $\tz''_4$
  $$
    |\rho(\tz''_4,p)|\sim \frac{\sqrt{\frac{2p}{\ell}}-p}
      {\sqrt{\frac{2p}{\ell}}-p} e^{-\sqrt{2p\ell} }\sim
      \frac{1-\sqrt{\frac{p\ell}{2}}}{1+\sqrt{\frac{p\ell}{2}}}(1-
      \sqrt{2p\ell}) \sim 1- 2\sqrt{2p\ell}.
  $$
  Since $\frac{p}{\sqrt{2\nu\wmm}} \gg \sqrt{2p\ell}$, we find
  $$
    \sup_{{\cal C}_n} |\rho(z,p)|=|\rho(\tz''_4,p)|\sim 1- 2\sqrt{2p\ell}.
  $$
  \end{description}
  The rest of the proof is now similar to the proof of the
  nonoverlapping case, except that now the best $p$ equilibrates the
  values of $|\rho|$ at the points $\tz''_4$ and $\tz'_w$, which is
  equivalent to $z_{sw}$. Asymptotically we have
  $$
    |\rho(\tz''_w)|\sim 1-2\frac{x_{sw}}{p},
  $$
  which gives for $p$ and the optimized contraction factor the
  asymptotic values
  $$
    \bar{p}^*(L)= \sqrt[3]{\frac{x_{sw}^2}{2\ell}}, \quad
      \delta^*(L)\sim 1-2\frac{x_{sw}}{\bar{p}^*(L)}.
  $$
\end{itemize}
The full justification that $\bar{p}^*(L)$ is indeed a strict local,
and hence the global optimum is analogous to the nonoverlapping case
and we omit it, and the proof is complete.


\section{Optimization of Ventcel Transmission Conditions}

This section is devoted to the proof of Theorems \ref{th:pqnoverlap},
\ref{th:toutopq} and \ref{th:toutoboundedpq}.  We start with a change
of variables,
$$
   s = p+q(z^2-\xoo)/4\nu= \tp +\tq z^2,\quad \tp=p-\xoo/4\nu, \quad \tq=q/4\nu,
$$
with which we can further simplify the convergence factor,
\begin{equation}\label{eq:rhopqmod}
   \rho(z,p,q,L)=\ds  \frac{\tp+\tq\,z^2-z}{\tp+\tq\,z^2+z} \
   e^{- \frac{Lz}{2\nu}}.
\end{equation}
Note that we will still write the arguments in terms of $p$ and $q$,
which are now simply functions of $\tilde{p}$ and $\tilde{q}$, and the
min-max problem is still
\begin{equation}\label{BestApproxtildepq}
  \ds \inf_{(p,q)\in\C^2}\ \sup_{z\in
     \td}|\rho(z,p,q,L)| =\sup_{z\in
     \td}|\rho (z,p^*,q^*,L)|
     =: \delta_1^*(L) .
\end{equation}

\subsection{The Nonoverlapping Case}

\paragraph{Proof of Theorem \ref{th:pqnoverlap} (Ventcel Conditions
  Without Overlap):} by the abstract Theorem \ref{th:gennonoverlappq},
the best approximation problem has a unique solution
$({p}_1^*(0),{q}_1^*(0))$. We search now for a strict local minimum
for the function $F_0(p,q)$.  We first analyze the variations of $R$
on the boundaries, and identify three local maxima. Then we show that
there exists $(\bar{p}_1^*,\bar{q}_1^*)$ such that these three values
coincide, and we compute their asymptotic behavior, showing that they
satisfy the assumptions. We finally show that
$(\bar{p}_1^*,\bar{q}_1^*)$ constitutes a strict local minimum for the
function $F_0$ on $\R_+\times\R_+$, from which it follows that the
local minimizer $(\bar{p}_1^*,\bar{q}_1^*)=({p}_1^*(0),{q}_1^*(0))$,
the global minimizer.

{\bf\em Local Maxima of the Convergence Factor:} The following Lemma
gives the local maxima of the convergence factor for the two
asymptotic regimes of an explicit and implicit time integration we are
interested in:
\begin{lemma}\label{lem:localmaxpqnooverlap}
Suppose the parameters in the Ventcel transmission condition satisfy
\begin{equation}\label{eq:asumptionpq0}
  p\pr \kmm^\alpha, \quad q\pr \kmm^\beta,\quad
    0<\alpha < \frac12 <\beta < 1, \ \alpha+\beta \le 1.
\end{equation}
Then, we have for the two asymptotic regimes of interest
\begin{enumerate}
\item in the implicit case, when $\kmm= C_h\wmm$, the supremum of the
  convergence factor is given by
  $$
   \sup_{\td_{+}}|\rho_0(z,p,q)| =
     \begin{cases}
     \max(|\rho_0(\bz_{sw}(p,q),p,q)|,|\rho_0(\bz_1(p,q),p,q)|,
       |\rho_0(z_3,p,q)|)
      & \mbox{ if }  \frac{p}{q} < \wmm,\\
     \max(|\rho_0(\bz_{sw}(p,q),p,q)|,|\rho_0(\bz_n(p,q),p,q)|,
       |\rho_0(z_3,p,q)|)
      & \mbox{ if }  \frac{p}{q} > \wmm,\\
  \end{cases}
  $$
  where $\bz_n \in \ccn$ is defined in \eqref{eq:pointnord}, and the
  asymptotic behavior is
  \begin{equation}\label{eq:expimplicit}
    \begin{array}{rlrl}
      \ds|\rho_0(\bz_{sw},p,q)|&\sim \ds 1-2\frac{x_{sw}}{p},\quad
      &|\rho_0(z_3,p,q)|&\sim \ds 1-\frac{4}{q\kmm},\\
      |\rho_0(\bz_1,p,q)|
      &\sim\ds
      1-2\sqrt{\frac{pq }{2\nu }},\quad
      &\ds|\rho_0(\bz_n,p,q)|&\sim\ds 1- \frac{p}{\sqrt{2\nu\wmm}}
      P(\frac{q\wmm}{p}),
    \end{array}
  \end{equation}
  where $P(Q)$ is defined in \eqref{eq:PdeQ}.

  \item in the explicit case, when $\wmm=\frac{1}{\pi C_h}\kmm^2$,
  the supremum of the convergence factor is given by
  $$
    \sup_{\td_{+}}|\rho_0(z,p,q)| =
     \max(|\rho_0(\bz_{sw}(p,q),p,q)|,|\rho_0(\bz_1(p,q),p,q)|,
      |\rho_0(\bz'_n(p,q),p,q)|),
  $$
  where $\bz_n(p,q)$ is defined in \eqref{eq:bzn}, and
  \begin{equation}\label{eq:bzprimen}
    \bz'_n(p,q)=\begin{cases}
      z_3 &\mbox{if $d> d_0$},\\
      \bz_n(p,q) &\mbox{if $d < d_0$},\\
     \end{cases}
  \end{equation}
  and we have asymptotically
  \begin{equation}\label{eq:expexplicit}
    |\rho(\bz_{sw},p,q)|\sim 1-2\frac{x_{sw}}{p},\quad
    |\rho(\bz_{1},p,q)|\sim 1-2\sqrt{\frac{pq}{2\nu}},\quad
    |\rho_0(\bz'_{n},p,q)|\sim
     1- \frac{4C}{ q \kmm } \sqrt{\frac{d}{2}},
  \end{equation}
  with $C $ defined in \eqref{BestPRhoNoOverlap}.
\end{enumerate}
\end{lemma}

\begin{proof}
  The proof of this lemma is rather long and technical, but follows
  along the same lines as in the Robin case: we first compute the
  derivatives of $R_0(\w,k,p,q)$ in $\w$ and $k$, using the
  formulation \eqref{eq:rhopqmod}, to obtain
  $$
    \begin{array}{rcl}
      \partial_z \rho_0&=&2\frac{(\tq z^2-\tp )}{(\tp+\tq z^2+z)^2},\\
      \partial_{{\w},{k}} R_0(\w,k,p,q)
      &=&4 \Re(\partial_z \rho_0\, \bar{\rho_0}\, \partial_{{\w},{k}} z)\\
      &=&4 \frac{\Re((\tq z^2-p)(\overline{(\tp+\tq z^2)^2-z^2})
      \,\partial_{{\w},{k}} z)}{|\tp+\tq z^2+z|^4} \\
      &=&4 \frac{\Re(N(z,\bar{z})\,\partial_{{\w},{k}} z)
       }{|\tp+\tq z^2+z|^4}, \\
       N(z,\bar{z})&=&(\tq z^2-\tp)((\tp+\tq \bar{z}^2)^2-\bar{z}^2).
    \end{array}
  $$
We now expand the numerator $N(z,\bar{z})$, using $X:=\xoo+4\nu^2k^2$
and $Y:=4\nu(\w+ck)$, so that
$$
z^2= X+iY,\quad
z=x+iy,\quad
x^2-y^2=X,\quad 2xy=Y.
$$
Using this notation, we obtain
$$
\begin{array}{rcl}
\Re N(z,\bar{z})&=&
(\tq X-p)(\tq^2X^2+(2\tp\tq-1)X+\tp^2)+
\tq (\tq^2X+3\tp\tq-1)Y^2
,\\
\Im N(z,\bar{z})&=&Y(-\tq^3X^2+2\tp\tq^2X+
\tp(3\tp\tq -1)-\tq^3Y^2).
\end{array}
$$
With the assumption on the coefficients $\tp$ and $\tq$, $\tp\tq \ll
1$, we have
\begin{equation}\label{eq:formulasN}
\begin{array}{rcl}
\Re N(z,\bar{z})&\sim&
\tq^3 X(X^2+Y^2)-\tq(X^2+Y^2)+\tp X-\tp^3,\\
\Im N(z,\bar{z})&\sim&Y(-\tq^3(X^2+Y^2)+2\tp\tq^2X-\tp).
\end{array}
\end{equation}
We present now the remining three major steps in the proof:
\begin{enumerate}
  \item We begin by studying, for fixed $k$, the variations of
    $\w\mapsto R_0(\w,k,p,q)$.  Since $\partial_\w
    z=2\nu(y+ix)/|z|^2$,
$$
\begin{array}{rcl}
\partial_{{\w}} R_0(\w,k,p,q)&=&
8\nu \frac{
\Re(N(z,\bar{z})\,(y+ix))
}{|\tp+\tq z^2+z|^4|z|^2}=
8\nu \frac{
\Phi_\omega
}{|\tp+\tq z^2+z|^4|z|^2} \\
\Phi_\omega&=& y\Re N-x\Im N\\
&\sim& y(\tq^3 X(X^2+Y^2)-\tq(X^2+Y^2)+\tp X-\tp^3
-2x^2(-\tq^3(X^2+Y^2)+2\tp\tq^2X-\tp).\\
\end{array}
$$
\begin{enumerate}
\item We study first the left boundary $\ccw$ with $k=\km$, where
  $X=\go(1)$ is fixed. We define $\xi=2x^2-X$, and replace
  $2x^2=\xi+X$, $X^2+Y^2=\xi^2$ in the previous expression. This
  yields a third order polynomial in the $\xi$ variable,
  \begin{equation}\label{eq:defQ3}
    \Phi_\omega \sim yQ_3(\xi) := y\left(
    \tq^3\xi^3+\tq(2\tq^2X-1)\xi^2+\tp(1-2\tq^2X)\xi+
    \tp(2X-2\tq^2X^2-\tp^2)\right).
  \end{equation}
  The principal part of $Q_3$ is
  \begin{equation}\label{eq:asympQ3}
    Q_3(\xi)\sim \tq^3\xi^3-\tq \xi^2+\tp \xi  -\tp^3.
  \end{equation}
  Since $y$ is always positive or vanishes for $\w=-c\km$ if
  $|c|\km\in (\wm,\wmm)$ (see Figure \ref{fig:squarerootdomainscp}),
  the sign of $\partial_\w R_0(z,p,q) $ is the sign of $Q_3(\xi)$.
  $Q_{3}$ has asymptotically three positive roots
\[
1 \ll \xi_0\sim \tp^2  \ll
\xi_1= \frac{\tp}{ \tq}  \ \ll
\xi_2\sim \frac{1}{ \tq^2}.
\]
With the assumptions on $\tp$ and $\tq$, the roots are
separated. Therefore, by continuity, $Q_{3}$ has three roots
$\xi'_0,\xi'_1,\xi'_2$ equivalent to $\xi_0,\xi_1,\xi_2$, and
$\partial_\w R_0(\w,k,p,q)$ has, in addition to $-c\km $, three zeros
$ \bw_j \sim \xi_j/4\nu$, $j=0,1,2$.  $\bw_0 $ and $\bw_2$ correspond
to minima of $R_0$. Note that $z(\bw_j(k),k)=z(\bw_j(-k),-k)$, so that
we can consider the part corresponding to $k=\soc \km $ only: there
exists a unique maximum at $\bz_1(p,q)=z(\bw_1(\soc \km),\soc \km)$,
and two minima at $z(\bw_0(\soc \km),\soc \km)$ and $z(\bw_2(\soc
\km),\soc \km)$, and we have the ordering
\begin{equation}
  \wm \ll \bw_0 \sim \frac{\tp^2}{4\nu} \ll \bw_1\sim \frac{\tp}{4\nu\tq}  \
     \ll \bw_2\sim \frac{1}{4\nu\tq^2}.
\end{equation}
If $\wmm \pr \kmm $, then $\bw_2 \gg \wmm$, and
 $$
     \sup_{\ccw} |\rho_0(z,p,q)|=
   \begin{cases}
        \max(|\rho_0(z_1,p,q)|, |\rho_0(\bz_1(p,q),p,q)|)& \mbox{if $\bw_1 \sim\frac{\tp}{4\nu \tq}< \wmm$},\\
        \max(|\rho_0(z_1,p,q)|, |\rho_0(z_4,p,q)|)& \mbox{if $\bw_1 \sim\frac{\tp}{4\nu \tq} > \wmm$},\\
     \end{cases}
$$
with
$$
    \ds |\rho_0(z_1,p,q)|\sim 1-2\frac{x_1}{\tp}, \quad |\rho_0(z_4,p,q)|\sim 1-\frac{\tp+4\nu\tq\wmm}{\sqrt{2\nu\wmm}},\quad
         |\rho_0(\bz_1,p,q)| \sim
        1-2\sqrt{2\tp\tq }.
$$
If $\wmm \pr \kmm^2 $, then $\bw_2 \ll \wmm$, and
$$
\sup_{\ccw} |\rho_0(z,p,q)|=
\max(|\rho_0(z_1,p,q)|,|\rho_0(\bz_1,p,q)|,|\rho_0(z_4,p,q)|),
$$
with
$$
    \ds |\rho_0(z_1,p,q)|\sim 1-2\frac{x_1}{\tp}, \quad |\rho_0(z_4,p,q)|\sim 1-2 \sqrt{2\nu\wmm}\tq,\quad
         |\rho_0(\bz_1,p,q)| \sim
        1-2\sqrt{2\tp\tq }.
$$

\item We now examine the behavior of $Q_3$ for $|k|=\kmm$. In that
  case, $X=\go(\kmm^2)$, and the asymptotics of the coefficients in
  $\Phi_\w$ are different. We use the fact that $\tq^2X \gg 1$, and $\frac{\tq
    X}{\tp} \gg 1$, to obtain
  \begin{equation}\label{eq:coeffgoinf}
    \Re N(z,\bar{z})\sim\tq^3X(X^2+Y^2),\quad
      \Im N(z,\bar{z})\sim -\tq^3Y(X^2+Y^2),
  \end{equation}
  so that
  $$
    \Phi_\w=\tq^3y(X^2+Y^2)(yX+xY) > 0,
  $$
  and we obtain for the convergence factor
  $$
    \sup_{\cce} |\rho_0(z,p,q)|=
        |\rho_0(z_3,p,q)|\sim 1- 2\frac{x_3}{\tq|z_3|^2}.
  $$
\end{enumerate}

\item Let us compute now the variations in $k$:
$$
\begin{array}{rcl}
\partial_{{k}} R_0(\w,k,p,q)&=&
4\frac{\Re(N(z,\bar{z})\,(\partial_kx+i\partial_k y))
}{|\tp+\tq z^2+z|^4 }=  8\nu\frac{\Phi_k
}{|\tp+\tq z^2+z|^4 |z|^2}\ ,\\
\Phi_k
&=&\frac{|z|^2}{2\nu} (\partial_kx\Re N(z,\bar{z})-\partial_ky\Im N(z,\bar{z}))\\
&=& (2\nu k x+cy)\Re N(z,\bar{z})-(-2\nu k y+cx)\Im N(z,\bar{z}).
\end{array}
$$
\begin{enumerate}
\item We begin with the southwest curve $\ccsw$, defined by
$\w=\wm$. Then $k$, $X$ and $Y$ are $ \go(1)$, and the asymptotics for
the coefficients are given by $$
\begin{array}{c}
\Re N(z,\bar{z}) \sim -\tp^3,
\ \Im N(z,\bar{z}) \sim -\tp,\\
\Phi_k \sim  -\frac{|z|^2}{2\nu}\tp^3\partial_kx \mbox{ if $\partial_kx \ne 0$}.
\end{array}
$$
By Corollary \ref{cor:tangentes}, if $|\tk_1(\wm)|\le\km$, $\partial_k
x$ does not change sign in the interval, and $ |\rho_0|$ is a
decreasing function of $x$. If $|\tk_1(\wm)| \in(\km,\wm/\ac)$,
$\partial_k x$ changes sign at $k=\tk_1$, and therefore $\partial_k
R_0(\w,k,p,q)$ changes sign for a point $\bk_3$ in the neighbourhood
of $\tk_1(\wm)$, which produces a maximum for $|\rho_0|$ at
$\bz_3=z(\wm,\bk_3)$.  We define
$$
\bz_{sw}=
\begin{cases}
z_1 &\mbox{if } |c\km| < \wm \mbox{ or if } |c\km| >  \wm \mbox{ and } |\bk_3| \not\in [\km,\frac{\wm}{\ac}], \\
\bz_3 \sim\tz_1(\wm) &\mbox{if } |c\km| >  \wm \mbox{ and } |\bk_3 | \in [\km,\frac{\wm}{\ac}],
\end{cases}
$$
and then obtain for the convergence factor
$$
\sup_{\ccsw} |\rho_0(z,p,q)|=|\rho_0(\bz_{sw},p,q)|\sim 1-2\frac{x_{sw}}{p}.
$$

\item We study next the northern curve $\ccn$, {\textit i.e.} $\w=\wmm$,
  $\soc k\in(\km,\kmm)$.
\begin{itemize}
    \item For $\wmm \pr\km$, we define $Y_0=4\nu\wmm$, and perform the
      asymptotic analysis in terms of $\yo$. We analyze the sign of
      $\Phi_k$ in the five asymptotic cases $k=\go(1)$,
      $k\pr\yo^\theta$ with $0<\theta<\frac12$, $k\pr\yo^\frac12$,
      $k\pr\yo^\theta$ with $\frac12<\theta<1$, and $k\pr\yo$.

  \begin{itemize}
    \item[$\checkmark$] If $k=\go(1)$, then $X=\go(1)$ and $Y\sim
      \yo$. The asymptotics for the coefficients are given by
    $$
\begin{array}{c}
\Re N(z,\bar{z})\sim
-(\tp^3+\tq \yo^2),\quad
\Im N(z,\bar{z})\sim-   \yo(\tp+\tq^3 \yo^2),
\quad x\sim y \sim\sqrt{\frac{\yo}{2}},\\
\Phi_k
\sim
x\left(
-(\tp^3+\tq \yo^2)(2\nu k +c)+\yo(\tp+\tq^3 \yo^2)(-2\nu k +c)
\right)
.
\end{array}
$$
With the assumptions on the coefficients, $\tp^2\ll \yo$ and
$\tq^2\yo\ll1$, so that
$$\Phi_k
\sim
x\yo\left(
-2\nu k(\tp +\tq \yo ) +c(\tp -\tq \yo )
\right)
.
$$
The quantity on the left changes sign for one value of $k$, therefore
$\Phi_k $ changes sign for
$$
  \bk_4(p,q)\sim \frac{c}{2\nu}\frac{\tp-\tq \yo}{\tp+\tq \yo},
  \quad
  \bz_4(p,q)=z(\wmm,\bk_4(p,q)).
$$
The point $\bz_4$ corresponds to a maximum, and is on $\ccn$ if and
only if the sign of $\bk_4$ is the sign of $c$, and its modulus is
larger than $\km$. If $\alpha +\beta < 1$, $ \bk_4(p,q)\sim
-\frac{c}{2\nu}$ and has the wrong sign. Therefore $\bz_4$ belongs to
$\ccn$ if and only if $\alpha +\beta =1$, and $\frac{\tq\yo}{\tp} <
1$.  At that point, $\tp+\tq z^2 \sim\tp+i\yo\tq \pr\yo^\alpha \ll
\bz_4\sim\sqrt{\frac{\yo}{2}}$, and therefore
$$
  \rho_0(\bz_4(p,q),p,q)\sim -(1-2\frac{\tp+\tq \tz_4^2}{\bz_4}),\quad
  |\rho_0(\bz_4(p,q),p,q)| \sim 1-\sqrt{\frac{2}{\yo}}( \tp +\tq \yo).
$$
\item[$\checkmark$] If $k\pr\yo^\theta$ with $0<\theta<\frac12$], then
$$
\begin{array}{rcl}
\Phi_k
&\sim&
2\nu k x\left(\Re N(z,\bar{z})+\Im N(z,\bar{z})\right)\\
&\sim&
-2\nu k x\left(\tq^3\yo^3+\tq\yo^2+\tp\yo)\right)
.
\end{array}
$$
This quantity has a constant sign equal to the sign of $k$, or
equivalently to the sign of $\partial_k x$. Therefore in this area,
$|\rho_0|$ is an increasing function of $x$.

\item[$\checkmark$] If $k\pr\yo^\frac12$], then $X\pr\yo$,
      $Y\sim\yo$, and inserting $t=X/\yo$, we have
$$
\begin{array}{rcl}
\Re N(z,\bar{z})&\sim&
 \tp X-\tq (X^2+\yo^2),\quad
\Im N(z,\bar{z})\sim-   \yo(\tp+\tq^3 (X^2+\yo^2)),
\quad x\pr y \pr\sqrt{\frac{\yo}{2}},\\
\Phi_k
&\sim&
2\nu k\left(x\Re N(z,\bar{z})+y\Im N(z,\bar{z})\right)
\sim
2\nu k\left(x(\tp X-\tq (X^2+\yo^2))
-y \yo(\tp+\tq^3 (X^2+\yo^2))
\right)\\
&\sim&
2\nu kx\yo\left( \tp t-\tq\yo  (t^2+1 )
-(\sqrt{t^2+1}-t)(\tp+\tq^3\yo^2 (t^2+1))
\right).
\end{array}
$$
Since $\tq^3\yo^2\ll \tq\yo$, asymptotically the only remaining terms are
$$
  \Phi_k\sim 2\nu kx\yo( \tp(2 t-\sqrt{t^2+1})-\tq\yo  (t^2+1 )).
$$
If $\alpha + \beta < 1$, $ \Phi_k\sim- 2\nu kx \tq\yo^2 (t^2+1 ) ) $
and does not vanish; $|\rho_0|$ is still a decreasing function of $x$
in this zone. If $\alpha + \beta =1$, we define the function
$g(t)=\frac{2t-\sqrt{t^2+1}}{t^2+1}$, drawn in Figure
\ref{fig:fonctiong}, and rewrite $\Phi_k$ as
\begin{equation}\label{eq:phik12}
\Phi_k\sim 2\nu kx\yo\tp (t^2+1 )( g(t)-\frac{\tq\yo}{\tp} ).
\end{equation}
 
\begin{figure}
  \centering
  \psfrag{Q}[][]{\footnotesize$Q$}
  \psfrag{t2}[][]{\footnotesize$t_2(Q)$}
  \psfrag{t0}[][]{\footnotesize$t_0$}
  \includegraphics[width=0.45\textwidth]{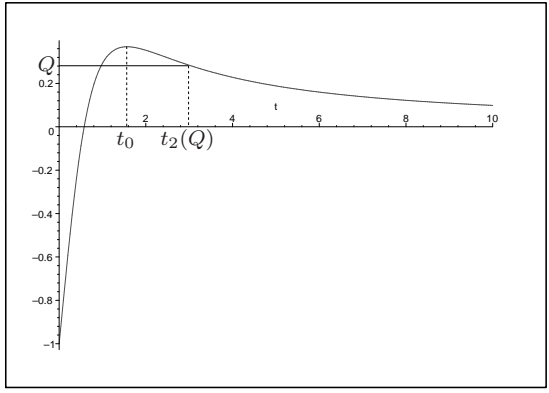}
\caption{Graph of the fonction $g$}
\label{fig:fonctiong}
\end{figure}
The function $g$ has a maximum at $t_0=\sqrt{54+6\sqrt{33}}/6\approx
1.5676$, with $g_0:=g(t_0)\approx 0.3690$. Therefore, if
$\frac{\yo\tq}{\tp} > g_0$, $k \Phi_k $ is negative for all $t$, and
$|\rho_0|$ is a decreasing function of $x$. Otherwise, the right hand
side in \eqref{eq:phik12} changes sign twice: the first time at
$t_1(\frac{Y\tq}{\tp}) <t_0$ corresponds to a local minimum, and the
second time at $t_2(\frac{Y\tq}{\tp}) > t_0$ corresponds to a local
maximum,
$$
\bk_5(p,q)\sim  \frac{\soc}{2\nu}
\sqrt{ \yo t_2(\frac{\yo\tq}{\tp}) },
\quad
\bz_5(p,q)=z(\wmm,\bk_5(\tp,\tq)) .
$$
\item[$\checkmark$] If $k\pr\yo^\theta$ with $\frac12<\theta<1$,
      then $X \gg \yo$, $Y\sim \yo$, and
$$
\begin{array}{rcl}
\Re N(z,\bar{z})&\sim&
 X(\tq^3X^2-\tq X +\tp),\quad
\Im N(z,\bar{z})\sim-   \yo(\tp+\tq^3 X^2),
\quad x\sim \sqrt{X},   y \sim\frac{\yo}{2\sqrt{X}},\\
\Phi_k
&\sim&
 \nu k X^{-\frac12}(2X\Re N(z,\bar{z}) +\yo\Im N(z,\bar{z}))\\
\Phi_k
&\sim&
2\nu k X^\frac32(\tq^3X^2-\tq X+\tp).
\end{array}
$$
The right hand side, as a function of $X$, has only one root for
$\frac{1}{2}<\theta<1$, $\frac{1}{\tq^2}$, corresponding to a local
minimum.

\item[$\checkmark$] If $k\pr\yo$, then $X\pr\yo^2$, $Y\pr \yo$, and
$$
\begin{array}{rcl}
\Re N(z,\bar{z})&\sim&
 \tq^3X^3 ,\quad
\Im N(z,\bar{z})\sim-  \tq^3Y X^2),
\quad x\sim \sqrt{X},   y \sim\frac{Y}{2\sqrt{X}},\\
\Phi_k
&\sim&
 \nu k X^{-\frac12}(2X\Re N(z,\bar{z}) +Y\Im N(z,\bar{z}))\\
&\sim&
2\nu k \tq^3X^\frac32(2X^2-Y^2)
\sim 4\nu k \tq^3X^\frac72 .
\end{array}
$$
\end{itemize}
To summarize we have :
\begin{itemize}
  \item if $\alpha + \beta < 1$, $ k \mapsto
    |\rho_0(\wmm,k,p,q)|$ has no local maximum on the curve $\ccn$.
  \item if $\alpha + \beta = 1$, $ k \mapsto
    |\rho_0(\wmm,k,p,q)|$ has two local maxima on the curve $\ccn$,
    $\bz_4(p,q)$ and $\bz_5(p,q)$.
\end{itemize}
To compare them, we define $Q=\frac{\tq \yo}{\tp}$, and get
$$
  \bk_5(p,q)\sim  \frac{\soc}{2\nu} \sqrt{ \yo t_2(Q) }, \quad
  \bz_5(p,q)=z(\wmm,\bk_5(p,q)),
\quad
|\rho_0(\bz_5,p,q)|\sim 1-2(\frac{\tp}{|\bz_5|^2}+\tq) \Re \bz_5.
$$
The convergence factors $|\rho_0(\bz_4,p,q)|$ and
$|\rho_0(\bz_5,p,q)|$ are both $1-\vpr(\wmm^{\frac14})$. In order to
compare the two, we compute
$$
|\rho_0(\bz_4,p,q)|\sim
 1-\tp\sqrt{\frac{2}{Y}}( 1 +Q),\quad
|\rho_0(\bz_5,p,q)|\sim 1-\tp\sqrt{\frac{2}{Y}}\sqrt{1+\sqrt{t_2(Q)^2+1}}( \frac{1}{\sqrt{t_2(Q)^2+1}} +Q).
$$
It is easier to compare
$$
h_2(t)=1 +g(t) \mbox{ and }h_1(t)=\sqrt{1+\sqrt{t^2+1}}( \frac{1}{\sqrt{t^2+1}} +g(t)),
$$
for $t \ge t_0$. A direct computation shows that
$$
\begin{cases}
\mbox{for }t < \bar{t}\approx 2.5484  & h_1(t) > h_2(t),\\
\mbox{for }t > \bar{t}\approx 2.5484  & h_1(t) < h_2(t),\\
\end{cases}
$$
which implies
$$
\begin{cases}
\mbox{for }\frac{\tq \yo}{\tp}> g_1\approx  0.3148  & |\rho_0(\bz_5,p,q)|< |\rho_0(\bz_4,p,q)|,\\
\mbox{for }\frac{\tq \yo}{\tp}< g_1\approx  0.3148  & |\rho_0(\bz_5,p,q)|> |\rho_0(\bz_4,p,q)|.
\end{cases}
$$
We can now conclude the northern study for the case where $\wmm
\pr\km$. We define
\begin{equation}\label{eq:pointnord}
\bz_n(\tp,\tq)=
\left\{
\begin{array}{lll}
\bz_5(\tp,\tq)&\mbox{if }\frac{\tq \yo}{\tp}< g_1\approx .1735,&\\
\bz_4(\tp,\tq) &\mbox{if }g_1<\frac{\tq \yo}{\tp}<1 &\mbox{and }\km\le|\bk_4|,\\
z_4 &\mbox{if }g_1<\frac{\tq \yo}{\tp}<1 &\mbox{and }\km \ge |\bk_4|,\\
z_4 &\mbox{if } \frac{\tq \yo}{\tp}>1. &
\end{array}
\right.
\end{equation}
Then we obtain for the convergence factor
$$
\sup_{\ccn} |\rho_0(z,p,q)|=
\max(|\rho_0(z_3,p,q)|,|\rho(z_n(\tp,\tq),p,q)|).
$$
with the asymptotic behavior ($P(Q)$ is defined in \eqref{eq:PdeQ})
\begin{equation}\label{eq:rhonord}
|\rho_0(\bz_n(p,q),p,q)|
\sim 1-\sqrt{\frac{2}{\yo}}\tp P(Q),
\quad
|\rho_0(z_3,p,q)| \sim1-\frac{1}{\nu \kmm\tq}.
\end{equation}

  \item If $\wmm\pr\kmm^2$, then $\yo=\go(\kmm^2)$, $X \ll Y$, and we obtain
  that
  \begin{itemize}
   \item[$\checkmark$] for $k\ll \kmm $, the dominant part of $\Phi_k$
   is given by
$$
\begin{array}{rcl}
\Phi_k &\sim& x\tq Y^2((2\nu k +c)(\tq^2X-1)+(-2\nu k+c)\tq^2Y)\\
&\sim& x\tq Y^2(2\nu k (\tq^2(X-Y)-1)+c(\tq^2(X+Y)-1))\\
&\sim& x\tq^3 Y^3(-2\nu k+c).
\end{array}
$$
Remember that $\tk_2(\wmm)$ is the point where $\partial_k y$
vanishes. If $|\tk_2(\wmm)|\le \km$, $\partial_k y$ does not vanish on
the curve $\ccn$, and $|\rho_0|$ is a decreasing function of $x$. If
$|\tk_2(\wmm)|> \km$, $\partial_k y$ does vanish on $\ccn$, at
$$
\bk_4(p,q)\sim \tk_2(\wmm)\sim \frac{c}{2\nu},\quad
\bz_4(p,q)=z(\bk_4(p,q),p,q)\sim \sqrt{\frac{\yo}{2}}(1+i), \quad
\rho_0(\bz_4(p,q),p,q) \sim  1-2\frac{1}{\tq \bz_4},
$$
which implies for the modulus of the convergence factor
$$
|\rho_0(\bz_4(p,q),p,q)|
\sim 1-2\Re \frac{1}{\tq \bz_4}
\sim 1-\frac{1}{\tq}\sqrt{\frac{2}{\yo}}.
$$
\item[$\checkmark$] for $k\pr \kmm  $
$$
\begin{array}{rcl}
\Re N(z,\bar{z})&\sim&
 \tq^3X(X^2+Y^2) ,\quad
\Im N(z,\bar{z})\sim-  \tq^3Y(X^2+Y^2),
\quad x\pr  y \pr\sqrt{\yo}, ,\\
\Phi_k
&\sim&
 2\nu k \tq^3(X^2+Y^2)(xX-yY)\sim  2\nu k x\tq^3(X^2+Y^2)(2X -\sqrt{X^2+Y^2}).
\end{array}
$$
The right hand side changes sign for $X=Y/\sqrt{3}$ corresponding to a
minimum.  Since $x$ is an increasing function of $X$,
 $$
 z(k)\in \ccn \iff \yo/\sqrt3 \le 4\nu^2\kmm^2.
 $$
 Note as in the first part, $\wmm=\frac{\nu}{d}\kmm^2$, and thus
$$
 z(k)\in \ccn \iff  \frac{1}{d\sqrt3} \le  1
 \iff  d\ge \frac{1}{ \sqrt3}.
 $$
\end{itemize}
 We define
\begin{equation}\label{eq:bzn}
  \bz_n(p,q)=\left\{\begin{array}{ll}
    \bz_4(p,q) &\mbox{if } \km\le|\bk_4|\sim\frac{\ac}{2\nu},\\
    z_4 &\mbox{if }\km \ge |\bk_4|,
  \end{array}\right.
\end{equation}
and obtain
$$
  |\rho_0(\bz_n(\tp,\tq),p,q)| \sim 1- \frac{1}{\tq}\sqrt{\frac{2}{\yo}}.
$$
The maximum of $|\rho_0|$ on $\ccn$ is therefore reached at $\bz_n$ or
$z_3$, with
$$
  |\rho_0(z_3,p,q)|\sim1- \frac{1}{\nu\tq \kmm }
   \sqrt{\frac{d}{2}\left( \frac{d+\sqrt{d^2+1}}{d^2+1}\right)}.
$$
A short computation shows that $|\rho_0(z_3,p,q)|$ and
$|\rho_0(\bz_n,p,q)|$ are asymptotically of the same order, and that
$$
  \sup_{\ccn} |\rho_0(z,p,q)|=\left\{\begin{array}{ll}
    |\rho_0(z_3,p,q)|  &\mbox{if } d> d_0\\
    |\rho(\bz_n,p,q)|  &\mbox{if } d < d_0
  \end{array}\right.
  \sim\ 1- \frac{1}{\nu\tq \kmm } C\sqrt{\frac{d}{2}},
$$
in the notation of Theorem \ref{th:toutno}.
\end{itemize}
\end{enumerate}

\item We can now finish with the southern part on the east, {\textit
  i.e.} $\w=-\wmm$, $\soc k\in(\wmm/\ac,\kmm)$. For this part to
  exist, $\wmm/\ac$ has to be smaller than $\kmm$, thus
  $\wmm=\go(\kmm)$, which implies that $X=\go(\kmm^2) \gg Y$, and
  $$
   \Phi_k \sim \tq^3 X^2(X-iY)(\partial_kx+i\partial_k y) \sim
    \tq^3 X^2\soc\frac{2\nu}{|z|^2}(X^2+Y(\ac\sqrt{X}-\frac{Y}{2})
    \sim \tq^3 X^4\frac{2\nu\soc}{|z|^2}.
  $$
  Therefore $|\rho_0|$ is an increasing function of $x$, and
  $$
    \sup_{\ccse}|\rho_0(z,p,q)| =|\rho_0(z_3,p,q)|.
  $$
\end{enumerate}
We can now simply collect all the previous results, and returning to
the variables $p$ and $q$ concludes the proof of this long lemma.

\end{proof}

{\bf\em Determination of the Global Minimizer by Equioscillation:} The
following lemma gives asymptotically the local minimizers for both the
implicit and explicit time integration schemes:
\begin{lemma}\label{lem:equioscpqnooverlap}
In the implicit case, when $\kmm= C_h\wmm$, there exist
$\bar{p}^*_1\pr \kmm^\frac14$, $\bar{q}^*_1\pr \kmm^{-\frac34}$
such that
$$
\begin{cases}
|\rho_0(\bz_{sw}(p,q),p,q)|=
|\rho_0(\bz_1(p,q),p,q)|=
|\rho_0(z_3,p,q)|)
& \mbox{ if }  \frac{p}{q} < \wmm,\\
|\rho_0(\bz_{sw}(p,q),p,q)|=
|\rho_0(\bz_n(p,q),p,q)|=
|\rho_0(z_3,p,q)|)
& \mbox{ if }  \frac{p}{q} > \wmm.\\
\end{cases}
$$
Defining $Q_0=\frac{2}{C_hx_{sw}}$, the coefficients are given
asymptotically by
$$
\bar{q}^*_1\sim \frac{2p}{x_{sw}\kmm},
\quad
\bar{p}^*_1\sim
\begin{cases}
\sqrt[4]{x_{sw}^3 \nu\kmm}
&\mbox{if $Q_0 > 1$},\\
\sqrt[4]{\frac{8\nu x_{sw}\wmm}{P(Q_0)^2}}.
&\mbox{if $Q_0 < 1$}.
\end{cases}
$$
In the explicit case, when $\wmm=\frac{1}{\pi C_h}\kmm^2$, there exist
$\bar{p}^*_1\pr \kmm^\frac14$, $\bar{q}^*_1\pr \kmm^{-\frac34}$
such that
 $$
|\rho_0(\bz_{sw}(p,q),p,q)|=
|\rho_0(\bz_1(p,q),p,q)|=
|\rho_0(\bz'n,p,q)|).
$$
The coefficients are given by
$$
\bar{q}^*_1 \sim \frac{2C p}{ x_{sw} \kmm },
\quad
\bar{p}^*_1\sim
 \sqrt[4]{\frac{\nu x_{sw}^3\kmm}{C}}.
$$

\end{lemma}
\begin{proof}
In each asymptotic regime for $\kmm$ and $\wmm$, we proceed in two steps:
\begin{itemize}
\item In the implicit case, $\wmm=\frac{1}{ C_h}\kmm$:
\begin{enumerate}
  \item For $p$ such that $p \pr \kmm^{\alpha}$, $\alpha <
    \frac12$, consider the equation
$$
  |\rho_0(\bz_{sw},p,q)|-|\rho_0(z_3,p,q)|=0,
$$
with the unknown $q$. By the expansions \eqref{eq:expexplicit}, we see
that for any $q\pr\kmm^{-\beta}$, $\frac12 < \beta < 1$,
$$
  |\rho_0(\bz_{sw},p,q)|-|\rho_0(z_3,p,q)|\sim
  \frac{4}{q\kmm}-2\frac{x_{sw}}{p},
$$
which can take positive or negative values according to the sign of
the right hand side. Therefore it vanishes for $q=\hat{q}(p)$, with
\begin{equation}\label{eq:qfunctionofpnooverlap}
  \hat{q}(p)\sim\frac{2p}{x_{sw} \kmm}\,.
\end{equation}
We verify that $\hat{q}(p)\pr \kmm^{-\beta}$, $\frac12 < \beta < 1$.

\item Consider now for large $\kmm$ and $Q_0 > 1$ the equation in the
  $p$-variable,
$$
  |\rho_0(\bz_{sw},p,\hat{q}(p))|-|\rho_0(\bz_1,p,\hat{q}(p))|=0.
$$
By the asymptotic expansions above, for $q=\hat{q}(p)$,
$$
  |\rho_0(\bz_{sw},p,q)|-|\rho_0(\bz_1,p,q)|\sim
  2\left(\sqrt{\frac{pq }{2\nu }}-\frac{x_{sw}}{p}\right)\sim
  2\left(p\sqrt{\frac{1}{x_{sw}\kmm }}-\frac{x_{sw}}{p}\right).
$$
This quantity takes positive or negative values, and vanishes for
a $\bar{p}^*_1$ with
$$
  \bar{p}^*_1\sim \sqrt[4]{x_{sw}^3 \nu\kmm}.
$$
Consider alternatively for $Q_0 < 1$  the equation in the $p$-variable,
$$
  |\rho_0(\bz_{sw},p,\hat{q}(p))|-|\rho_0(\bz_n,p,\hat{q}(p))|=0.
$$
By the asymptotic expansions above, for $q=\hat{q}(p)$,
$$
  |\rho_0(\bz_{sw},p,q)|-|\rho_0(\bz_n,p,q)|\sim \frac{p}{\sqrt{2\nu\wmm}}
    P(Q_0)-2\frac{x_{sw}}{p}.
$$
Again, this quantity vanishes for a $\bar{p}^*_1$ with
$$
  \bar{p}^*_1\sim \sqrt[4]{\frac{8\nu x_{sw}^2\wmm}{P(Q_0)^2}}.
$$
\end{enumerate}
\item In the explicit case, $\wmm=\frac{1}{\pi C_h}\kmm^2$:

\begin{enumerate}
\item We first solve, for fixed $p$, the equation in $q$,
\[
|\rho(\bz_{sw}(p,q),p,q)|-|\rho(\bz'_n(p,q),p,q)|=0.
\]
By the expansions in \eqref{eq:expexplicit},
$$
  |\rho(\bz_{sw}(p,q),p,q)|-|\rho(\bz'_n(p,q),p,q)|\sim
    \frac{4C}{ q \kmm }\sqrt{\frac{d}{2}}-2\frac{x_{sw}}{p},
$$
and $|\rho(\bz_{sw}(p,q),p,q)|-|\rho(\bz'_n(p,q),p,q)|$ vanishes for
$$
  q=\hat{q}(p)\sim \frac{2C p}{ x_{sw} \kmm }\sqrt{\frac{d}{2}}.
$$
\item We solve now for $q=\hat{q}(p)$, the equation
$$
  |\rho(\bz_{sw}(p,q),p,q)|-|\rho(\bz_{1}(p,q),p,q)|=0,
$$
whose asymptotic behavior is
$$
  |\rho(\bz_{sw}(p,q),p,q)|-|\rho(\bz_{1}(p,q),p,q)|\sim
    2\sqrt{\frac{pq}{2\nu}}-2\frac{\bx_{sw}}{p}\sim
    2p\sqrt{\frac{C}{\nu x_{sw}\kmm}\sqrt{\frac{d}{2}}}-2\frac{x_{sw}}{p}.
$$
By the same arguments as before,
$|\rho(\bz_{sw}(p,q),p,q)|-|\rho(\bz_{1}(p,q),p,q)|$ vanishes for
$$
\bar{p}^*_1\sim\sqrt[4]{\frac{\nu x_{sw}^3\kmm}{C}\sqrt{\frac{2}{d}}}.
$$
\end{enumerate}
\end{itemize}
We have now proved that there exist in all cases coefficients $p$ and $q$
satifying the relations in the lemma.  They satisfy $\bar{p}^*_1\pr
\kmm^\frac14$, $\bar{q}^*_1\pr \kmm^{-\frac34}$, and are therefore conforming
to the previous study with $\alpha +\beta =1$.
\end{proof}

It remains to show that this is indeed a strict local minimum for the
function $F_0$. By the same argument as in the Robin case, we can
prove that for $\delta p$ and $\delta q$ sufficiently small and
$p=\bar{p}_1^*+\delta p$, $q=\bar{q}_1^*+\delta q$,

$$
F_0(p,q)-F_0(\bar{p}_1^*,\bar{q}_1^*)=
\max_{\mu}((\delta p\,
\partial_{\tp}+\delta q\,
\partial_{\tq})|\rho_0(\bz_\mu,\bar{p}_1^*,\bar{q}_1^*)|)+
\po(\delta p,\delta q),
$$
where the points $\bz_\mu$ are those involved in the maximum: if
$\wmm\pr\kmm $, $\bz_{sw}$ and $z_3$ in any case, and either $\bz_n$
or $\bz_1$, and if $\wmm\pr\kmm^2 $, $\bz_{sw}, \,\bz_n'$ and $\bz_1$.

Therefore, $(\bar{p}^*_1,\bar{q}^*_1)$ is a strict local minimum of
$F_0(p,q)$ if and only if for any $(\delta p ,\delta q)$, there exists
a $\bz_\mu$ such that $(\delta p\, \partial_{\tp}+\delta q\,
\partial_{\tq})R_0(\bz_\mu,\bar{p}_1^*,\bar{q}_1^*) > 0$.  To analyze
this quantity, we rewrite the convergence factor in the form
$$
R_0=\frac{\phi-\psi}{\phi+\psi}, \mbox{ with }
\begin{cases}
\phi=\tq^2|Z|^2+2\tp\tq X+\tp^2+|z|^2,\\
\psi=2x(\tp+\tq|z|^2).
\end{cases}
$$
This allows us to write the derivatives in the more elegant form
$$
R_0'=\frac{\psi\phi'-\psi'\phi}{(\phi+\psi)^2},
$$
and at an extremum, $R_0={\delta_1^*}^2$ implies that
$\psi/\phi=\zeta:=\frac{1-(\delta_1^*)^2}{1+(\delta_1^*)^2}$, and
$$
R_0'=\frac{\zeta\phi'-\psi'}{(1+\zeta)^2\phi}.
$$
We therefore obtain
\begin{eqnarray*}
(\delta p\,
\partial_{\tp}+\delta q\,
\partial_{\tq})R_0(\bz_\mu,\bar{p}_1^*,\bar{q}_1^*) & = &
  \frac{\zeta\partial_{\tp}\phi
    -\partial_{\tp}\psi}{(1+\zeta)^2\phi}\delta p
  + \frac{\zeta\partial_{\tq}\phi
    -\partial_{\tq}\psi}{(1+\zeta)^2\phi}\delta q\\
  & = & \frac{2(\zeta(\tp+\tq X)- x)}{(1+\zeta)^2\phi}\delta p
  + \frac{2(\zeta(\tp X+\tq |Z|^2)- x|z|^2)}{(1+\zeta)^2\phi}\delta q\\
  & =: & \frac{2\Phi(\bz_\mu,\delta p,\delta q)}{(1+\zeta)^2\phi}.
\end{eqnarray*}
We now study the asymptotic behavior of $\Phi$ for the two cases of interest:
\begin{itemize}
\item If $\wmm\pr\kmm $, then
$$
\begin{array}{rcl}
\Phi(\bz_1,\delta p,\delta q)&\sim&
-\bx_1(\delta p  + \frac{\tp}{\tq} \delta q),\\
\Phi(\bz_n,\delta p,\delta q)&\sim&
-\bx_n(\delta p  +M Y_0\, \delta q ),\\
\Phi(\bz_{sw},\delta p,\delta q)&\sim&
x_{sw}(\delta p  +(x_{sw}^2-3y_{sw}^2)\delta q) ,\\
\Phi(z_3,\delta p,\delta q)&\sim&
2\nu\kmm( \delta p  +(2\nu\kmm)^2\delta q ).
\end{array}
$$
where $M$ is given by
$$
M=
\begin{cases}
1 &\mbox{ if } Q_0=\frac{2}{C_hx_{sw}} > g_1,\\
\sqrt{1+\left(t_2(Q_0)\right)^2}
 &\mbox{ if } Q_0< g_1.
 \end{cases}
$$
Therefore, $(\bar{p}^*_1,\bar{q}^*_1)$ is a strict local minimum of
$F_0(p,q)$ if and only if the union of the following set equals
$\R^2$:
$$
\begin{array}{c}
{\cal E}_1=\{(\delta p,\delta q), - (\delta p  + M\kmm \delta q)>0\},\\
{\cal E}_2=\{(\delta p,\delta q),\delta p  +(2\nu\kmm)^2\delta q >0\},\\
{\cal E}_3=\{(\delta p,\delta q),\delta p  +(x_{sw}^2-3y_{sw}^2)\delta q>0 \}.
 \end{array}
$$

The domains are shown in Figure \ref{fig:minstrictpq}: for large
$\kmm$, the slopes of $D_1$, $\delta p + M\kmm \delta q=0$ and $D_2$,
$\delta p +(2\nu\kmm)^2\delta q =0$ are such that ${\cal E}_1
\cup{\cal E}_2$ is $\R^2$ excluding a small angle $\breve{\cal
  E}=\{\delta q < 0, \, -M\kmm \delta p < \delta q < (2\nu\kmm)^2
\delta q\}$. If $ x_{sw}^2-3y_{sw}^2 < 0$, ${\cal E}_3$ contains the
whole quadrant $\delta p > 0, \ \delta q < 0$. If $ x_{sw}^2-3y_{sw}^2
> 0$, the slope of $D_3$, $\delta p +(x_{sw}^2-3y_{sw}^2)\delta q
=0$, is $\go(1)$, so that ${\cal E}_3$ contains $\breve{\cal E}$.
\begin{figure}
  \centering
  \psfrag{dp}{$\delta p$}
  \psfrag{dq}{$\delta q$}
  \psfrag{d1}{$D_1$}
  \psfrag{d2}{$D_2$}
  \psfrag{d3}{$D_3$}
  \begin{tabular}{c c  }
   \subfloat[$x_{sw}^2-3y_{sw}^2 < 0$]{
\includegraphics[width=0.45\textwidth]{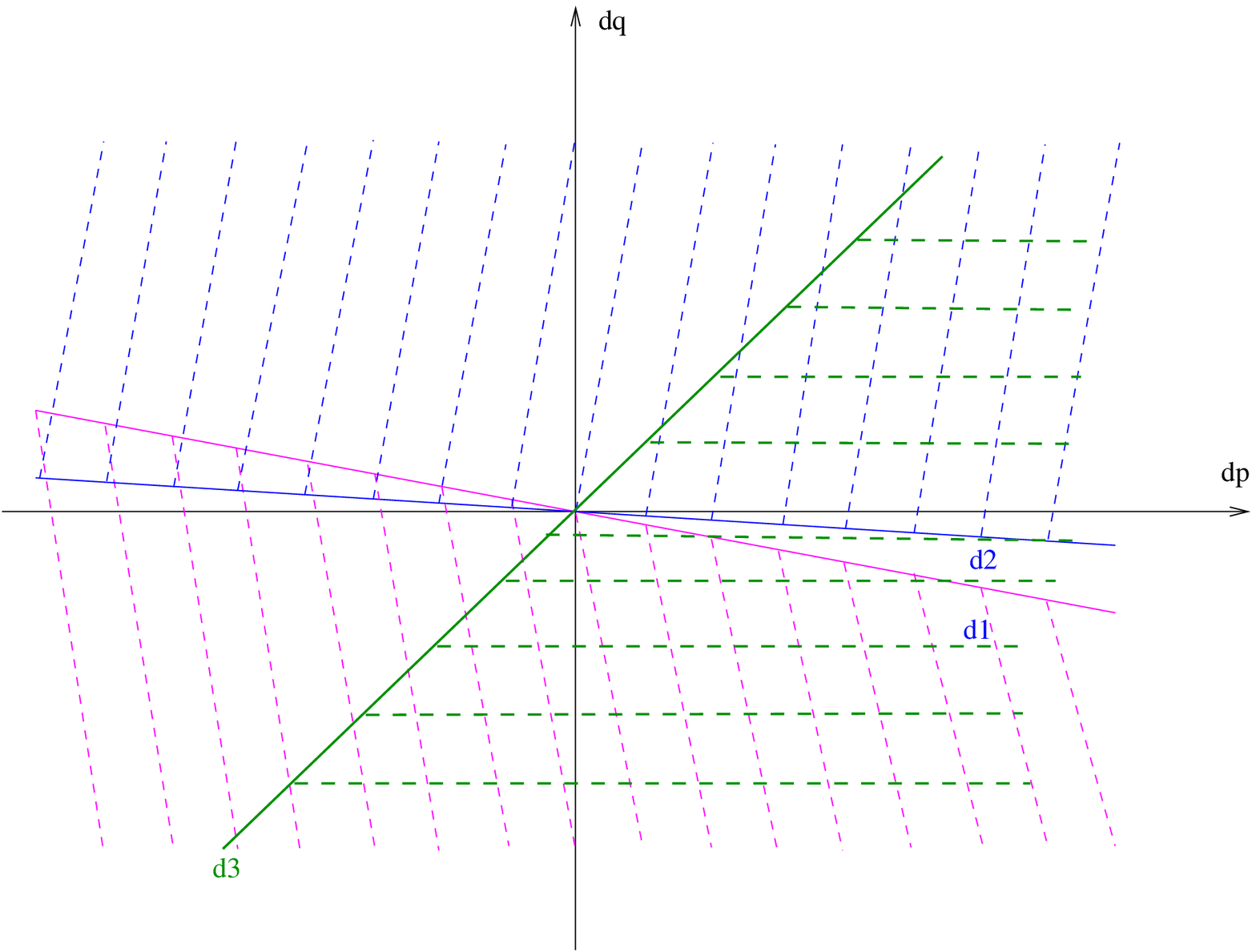}
 \label{fig:minstrictpq2}}
&
\subfloat[$x_{sw}^2-3y_{sw}^2 > 0$]{
\includegraphics[width=0.45\textwidth]{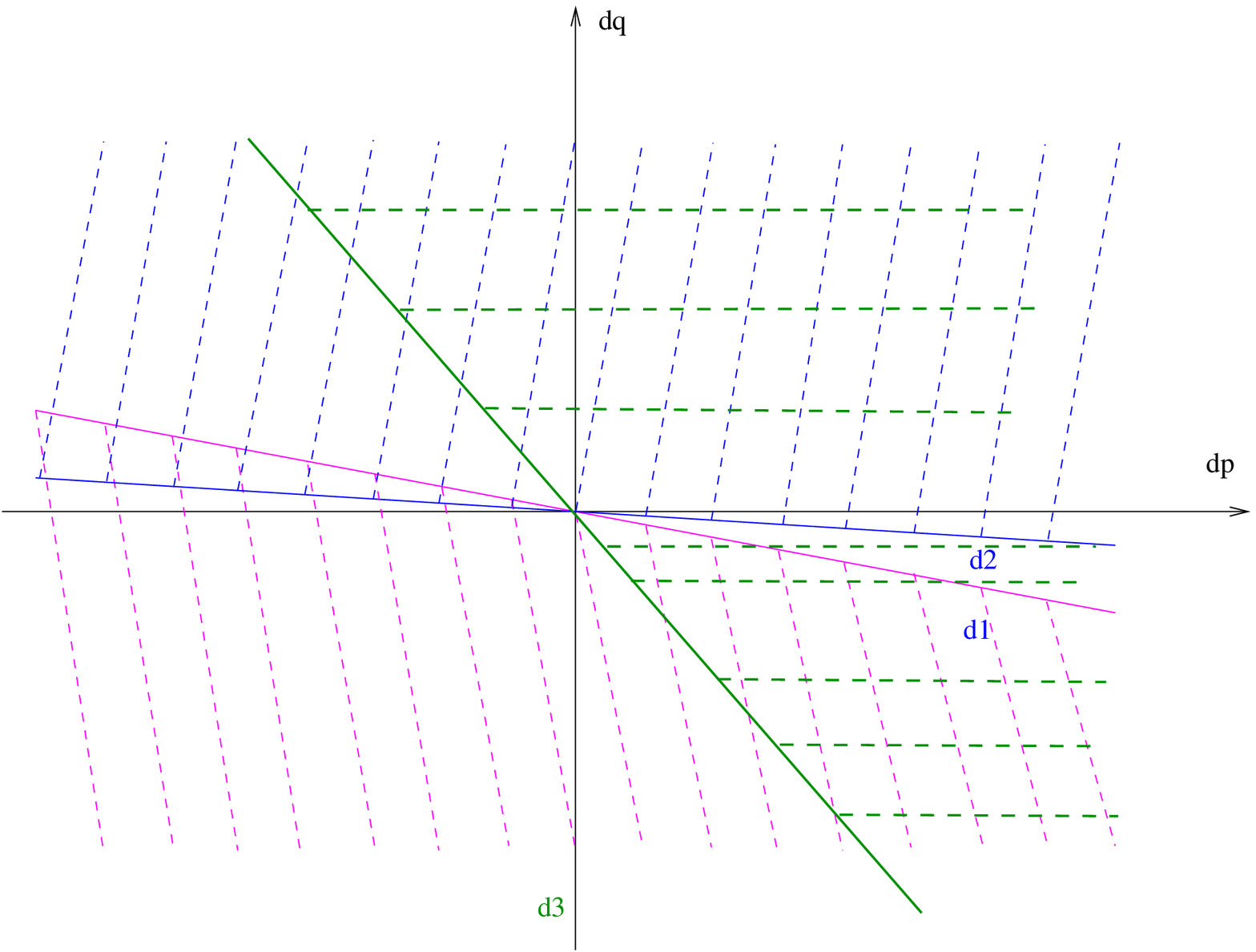}
 \label{fig:minstrictpq1}}
\end{tabular}
\caption{Description of the analysis for $\wmm\pr\kmm$}
\label{fig:minstrictpq}
\end{figure}

\item If $\wmm=\frac{\nu\kmm^2}{d}$, then the asymptotics for
  $\bz_{sw}$ and $\bz_1$ remain unchanged. The asymptotics for $\bz_n'$
  become
  $$
  \Phi(\bz_n',\delta p,\delta q) \sim
\begin{cases}
\sqrt{2\nu\wmm}(-\delta p  + 4\nu\wmm \delta q)
&\mbox{ if } d < d_0\\
\frac{2\nu C\kmm}{\sqrt{2d}}((2d-\sqrt{d^2+1})\delta p  + 4\frac{d^2+1}{d}(\nu\kmm)^2 \delta q)
&\mbox{ if } d > d_0.
\end{cases}
$$
If $ d < d_0$, the situation is the same as in Figure
\ref{fig:minstrictpq}. If $ d > d_0$, we obtain the conclusion as
indicated in Figure \ref{fig:minstrictpq}.
\begin{figure}
  \centering
  \psfrag{dp}{$\delta p$}
 \psfrag{dq}{$\delta q$}
  \psfrag{d1}{$D_1$}
   \psfrag{d2}{$D_2$}
    \psfrag{d3}{$D_3$}
 \begin{tabular}{c c  }
  \subfloat[$x_{sw}^2-3y_{sw}^2 < 0$]{
\includegraphics[width=0.45\textwidth]{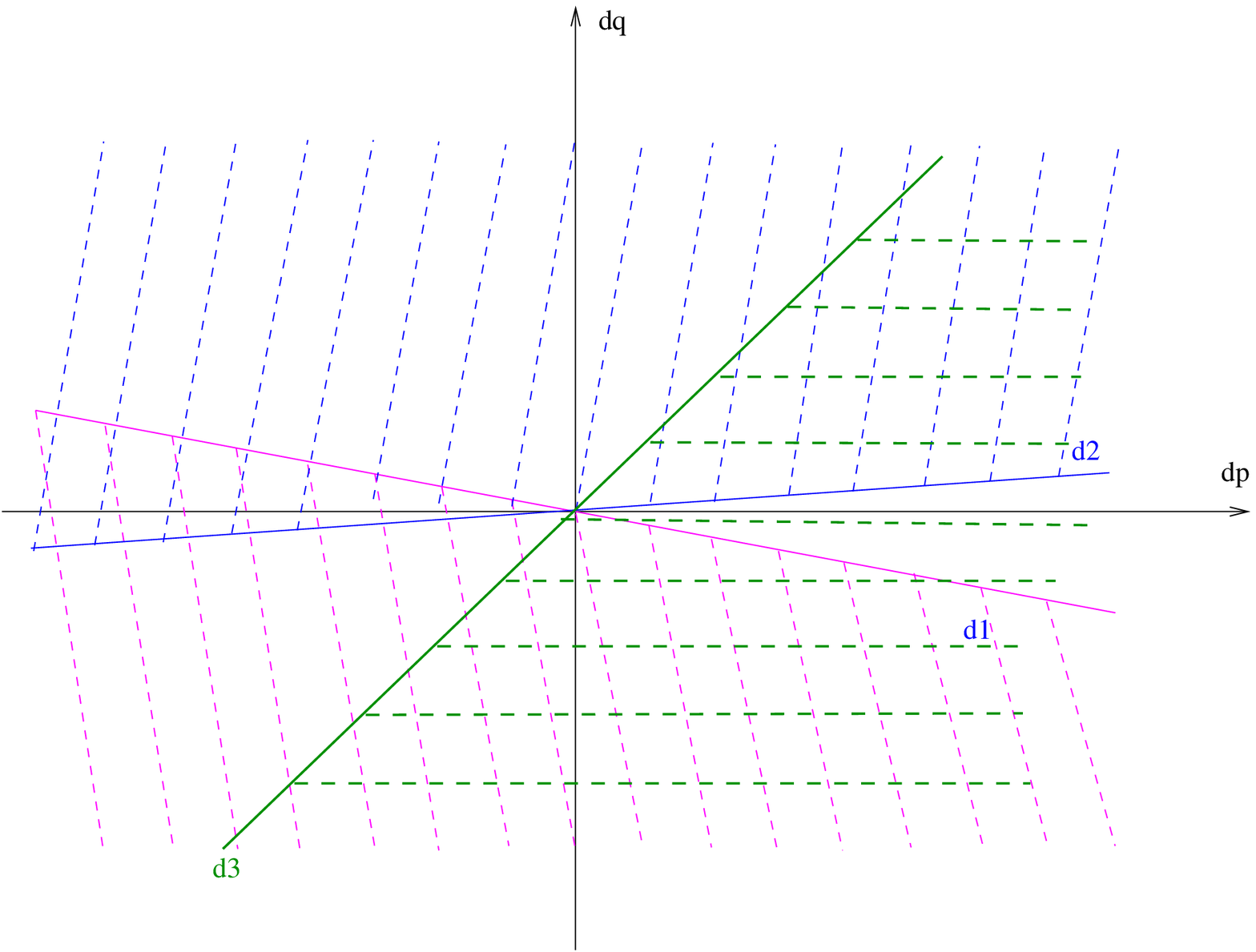}
 \label{fig:minstrictpq3}}
&
\subfloat[$x_{sw}^2-3y_{sw}^2 > 0$]{
\includegraphics[width=0.45\textwidth]{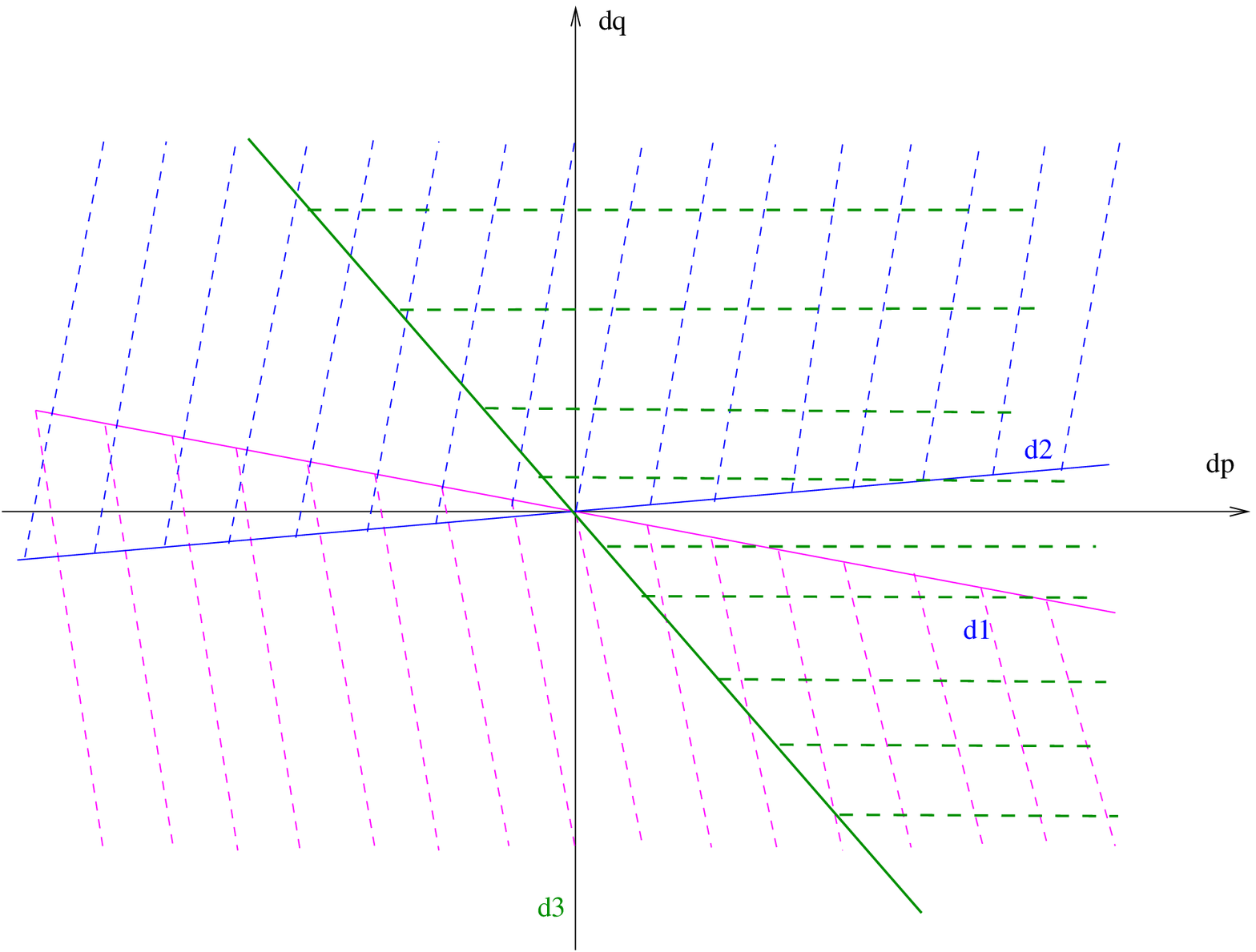}
 \label{fig:minstrictpq1b}}
\end{tabular}
\caption{Description of the analysis in the case $\wmm=\frac{\nu\kmm^2}{d}$ with $d > d_0$}
\label{fig:minstrictpqb}
\end{figure}
\end{itemize}

\subsection{The Overlapping Case}

We follow along the same lines as in the Robin case, starting with the
infinite case where only $L$ is involved.  Denoting by $\ell:=L/2\nu$
as before to simplify the notation, we obtain for the derivatives of
the convergence factor
 $$
\begin{array}{rcl}
R(\w,k,p,q,L)&=&R_0(\w,k,p,q)e^{-2\ell x},\\
\partial_{{\w},k} R(\w,k,p,q,L)&=&\partial_{{\w},k} R_0(\w,k,p,q)-
2\ell\partial_{\w}x R_0(\w,k,p,q)\\
&=&
\frac{
4\Re(N(z,\bar{z})\,(\partial_{\w,k}x + i \partial_{\w,k}y))
-2\ell \partial_{\w}x |(\tp+\tq z^2)^2-z^2|^2
}{|\tp+\tq z^2+z|^4}\\
&=&
4\frac{
 (\Re N(z,\bar{z})-\frac{\ell}{2}M)\, \partial_{\w,k}x -
  \Im N(z,\bar{z}) \partial_{\w,k}y
}{|\tp+\tq z^2+z|^4},\\
\end{array}
$$
with $M=|(\tp+\tq z^2)^2-z^2|^2$.

\paragraph{Proof of Theorem \ref{th:toutopq} (Ventcel Conditions with
  Overlap, Continuous):} we solve the min-max problem
on the infinite domain $\td_+^\infty$. By the abstract Theorem
\ref{th:genoverlappq}, for sufficiently small $L$, the problem has a
solution.  We need to prove that $F_L$ has a strict local minimum,
which will again be achieved by equioscillation. The proof consists of
two steps, shown in the following lemmas:

\begin{lemma}[Local Extrema]\label{lem:localmaxpqoverlapinf}
Suppose $p\pr \kmm^\alpha$, $q\pr \kmm^\beta$, $0<\alpha < \frac12
<\beta < 1$, $\alpha+\beta < 1$. Then,
$$
  \sup_{\td^\infty}|\rho(z,p,q,L)| = \max(
|\rho(\bz'_{sw}(p,q),p,q,L)|,
|\rho(\bz'_1(p,q),p,q,L)|,
|\rho(\bz^"_1(p,q),p,q,L)|),
$$
where  $\bz'_{sw}\sim z_{sw}$. The two other points belong to $\ccwi$, with
 $$
 \begin{array}{ll}
 \bw_1'\sim \frac{\tp}{4\nu\tq},
& |\rho(\bz_1',p,q,L)|\sim 1-2\sqrt{2\tp\tq},\\
 \bw_1^"\sim \frac{2}{\ell\tq}, &
|\rho(\bz_1^",p,q,L)|\sim 1- 2\sqrt{\frac{\ell}{\tq}}.
\end{array}
$$
\end{lemma}
\begin{proof}
We make the assumptions on the coefficients $p$ and $q$ in
\eqref{eq:asumptionpq0}. We start with the variations of $R$ on the
west boundary, \textit{i.e.} as a function of $\w$ for $k=\km$:
$$
\begin{array}{rcl}
\partial_{{\w}} R(\w,k,p,q,L)&=&
 8\nu
 \ds\frac{\Phi_\w^\ell}{|z|^2|\tp+\tq z^2+z|^4},\\[3mm]
\Phi_\w^\ell&=&
\ds\Phi_\w -\frac{\ell}{2}M y.
\end{array}
$$
We rewrite $M$ in terms of $\xi$ as in \eqref{eq:defQ3}, using $ Y\sim
\xi$,
$$
M= |(\tp+\tq X+i\tq Y)^2-X-iY|^2
\sim |(\tp^2-\tq^2 Y^2-X) +iY(2(\tp+\tq X)\tq-1)|^2
\sim tp^4+\tq^4 Y^4 + Y^2
\sim
\tq^4\xi^4+\xi^2+\tp^4,
$$
and we obtain
$$
  \Phi_\w^\ell\sim yQ_4:= y( -\frac{\ell}{2} \tq^4\xi^4+Q_3).
$$
The fourth-order polynomial $Q_4$ is a singular perturbation of $Q_3$
defined in \eqref{eq:defQ3}. The roots are therefore perturbations of
those already defined, with in addition $\xi^"_1$, whose principal part
solves
$$
 \tq^3 \xi^3
- \frac{\ell}{2}\tq^4\xi^4
=0.
$$
By the same argument as before, $Q_4$ has four roots,
\[
1 \ll \xi'_0\sim  \tp^2  \ \ll
\xi'_1\sim \frac{\tp}{ \tq}  \ \ll
\xi'_2\sim \frac{1}{ \tq^2}\ll
\xi^"_1\sim\frac{2}{\ell\tq},
\]
and $\partial_\w R(\w,k,p,q)$ has, in addition to $\w=-c\km$, four
zeros $\omega_0'$, $\omega_1'$, $\omega_2'$ and $\omega_1"$,
equivalent to the corresponding $\xi/4\nu$.  $\xi'_0 $ and $\xi'_2$
correspond to minima of $R$, while $\bz_1'=z(\bw_1',\soc \km)$ and
$\bz_1^"=z(\bw_1^",\soc \km)$ correspond to maxima. At the maxima
we have $Y\sim\xi$, $X=\go(1)$, and $z\sim \sqrt{\xi}(1+i)$, which implies
for the convergence factor
$$
  |\rho(\bz_1',p,q,L)|\sim 1-2\sqrt{2\tp\tq},\quad
    |\rho(\bz_1^",p,q,L)|\sim 1- 2\sqrt{\frac{\ell}{\tq}}.
$$
If $\ac\km > \wm$, the local extrema are $z_1$, $\bz_1'$ and
$\bz_1^"$.  If $\ac\km < \wm$, we must take $\ccsw$ into account. We
use the results derived in the nonoverlapping case to obtain
$$
\begin{array}{rcl}
\partial_{{k}} R(\w,k,p,q,L)&=&8\nu
\frac{ \Phi_k^\ell }{|\tp+\tq z^2+z|^4|z|^2 },\\
\Phi_k^\ell
&=&  \Phi_k -\frac{\ell}{2} \partial_{\w}x |(\tp+\tq z^2)^2-z^2|^2\\
&=&  (\Re N(z,\bar{z})-\frac{\ell}{2}M) \partial_{k}x
-\Im N(z,\bar{z}) \partial_{k}y.
\end{array}
$$
By the results in the previous section, since $M=\go(1)$,
$$
\Phi_k^\ell
\sim
-\frac{|z|^2}{2\nu}\partial_kx(\tp^3+4\frac{\ell}{2}M)\\
\sim  -\frac{|z|^2}{2\nu}\partial_kx,
$$
if $\partial_kx \ne 0$. By Corollary \ref{cor:tangentes}, if
$|\tk_1(\wm)|\le\km$, $\partial_k x$ does not change sign in the
interval, and thus $|\rho|$ is a decreasing function of $x$. If
$|\tk_1(\wm)| \in(\km,\wm/\ac)$, $\partial_k x$ changes sign at
$k=\tk_1$, and therefore $\partial_k |\rho|^2$ changes sign for a
point $\bk_3'$ in the neighbourhood of $\tk_1(\wm)$, which produces a
maximum at $\bz_3'=z(\wm,\bk_3')$. We thus define
$$
\bz_{sw}'=
\begin{cases}
z_1 &\mbox{if } |c\km| < \wm,\\
z_1 &\mbox{if } |c\km| >  \wm \mbox{ and } |\bk_3'| \not\in [\km,\frac{\wm}{\ac}], \\
\bz_3' \sim\tz_1(\wm) &\mbox{if } |c\km| >  \wm \mbox{ and } |\bk_3' | \in [\km,\frac{\wm}{\ac}],
\end{cases}
$$
and obtain for the convergence factor
$$
\sup_{\ccsw} |\rho(z,p,q,L)|=|\rho(\bz_{sw}',p,q,L)|\sim 1-2\frac{\bx_{sw}'}{\tp}\sim 1-2\frac{x_{sw}}{\tp}.
$$
We can therefore conclude that
\[
\sup_{z\in\td_+} |\rho(z,p,q,L)|=\max(|\rho(\bz_{sw}',p,q,L)|,
|\rho(\bz_1',p,q,L)|,|\rho(\bz_1^",p,q,L)|).
\]
\end{proof}

\begin{lemma}[Local Minimum for $F_L(p,q)$]\label{lem:equioscpqoverlapinfini}
There exist $\bar{p}^*_\infty\pr \kmm^\frac15$, $\bar{q}^*_\infty\pr
\kmm^{-\frac35}$ such that
$$
  |\rho(\bz_{sw}^",p,q,L)|=|\rho(\bz_1',p,q,L)|=|\rho(\bz_1^",p,q,L)|.
$$
The coefficients are given asymptotically by
$$
\bar{p}^*_\infty\sim \sqrt[5]{\frac{x_{sw}^4}{2\ell}},
\quad
\bar{q}^*_\infty\sim 4\nu \frac{x_{sw}^2}{2\tp^3}\sim 4\nu\sqrt[5]{\frac{\ell^3}{4x_{sw}^2}},
\quad
\delta_\sim 1-2\sqrt[5]{2\ell x_{sw}}.
$$
\end{lemma}
\begin{proof}
We skip the arguments which are similar to those of the previous
section, and show only the computation of the parameters.  Since
$$
|\rho(\bz_{sw}^",p,q,L)|-|\rho(\bz_1',p,q,L)|
\sim 2(\sqrt{2\tp\tq}-\frac{x_{sw}}{\tp}),\quad
|\rho(\bz_{sw}^",p,q,L)|-|\rho(\bz_1^",p,q,L)|
\sim  2\left( \sqrt{\frac{\ell}{\tq}}-\frac{x_{sw}}{\tp}\right),
$$
we must have asymptotically
$$
2\tp^3\tq\sim x_{sw}^2,\quad
\ell\frac{\tp^2}{\tq}\sim x_{sw}^2.
$$
which gives the formulas in the lemma. Notice that they have the
announced asymptotic behavior $\bar{p}^*_\infty=\gol{-\frac15}$, $
\bar{q}^*_\infty=\gol{ \frac35}$, validating the computations made
above.  We finally recover the results in the Lemma by returning to
the original variables $p$ and $q$.
\end{proof}

The proof that $\bar{p}^*_\infty$, $ \bar{q}^*_\infty$ is a strict
local minimum of $F_L$ is analogous to that in the nonoverlapping case
and therefore we omitted it. Then by the abstract Theorem
\ref{th:genoverlappq}, we found the global minimum, and the proof of
Theorem \ref{th:toutopq} is complete.

\paragraph{Proof of Theorem \ref{th:toutoboundedpq} (Ventcel Conditions with
  Overlap, Discrete):} the
existence and uniqueness for the min-max problem is again covered by
the abstract theorem. We thus only need to show the local maxima in
the convergence factor, and the strict local minimizer for $F_L(p,q)$,
which is done in the following two lemmas:

\begin{lemma}[Local Maxima of $R$ on $\td$]\label{lem:localmaxpqoverlapbounded}
Suppose $p\pr \kmm^\alpha$, $q\pr \kmm^\beta$, $0<\alpha < \frac12
<\beta < 1$, $\alpha+\beta < 1$.  Then, if $\wmm\pr\kmm^2$, we have
 $$
  \sup_{z\in \td}|\rho(z,p,q,L)| =
    \max(|\rho(\bz'_{sw}(p,q),p,q,L)|,|\rho(\bz'_1(p,q),p,q,L)|,
    |\rho(\bz^"_1(p,q),p,q,L)|).
$$
If $\wmm\pr\kmm$, then
\[
  \sup_{z\in\td} |\rho(z,p,q,L)|=\max(|\rho(\bz_{sw}'(p,q),p,q,L)|,
    |\rho(\bz_1'(p,q),p,q,L)|,|\rho(\bz_4'(p,q),p,q,L)|),
\]
where $\bz_4'(p,q) \in \ccn$ is such that
$$
|\rho(\bz_4'(p,q),p,q)\sim 1- 2 \sqrt{\frac{2}{\ell\tq}}.
$$
\end{lemma}
\begin{proof}
We have already computed the extrema on $\ccsw$ and $\ccwi$. For the
west boundary $\ccw$, we need to check if the computed values are
indeed inside the bounded domain.  With the assumptions on $p$ and
$q$, the first maximum on $\ccwi$ is at $\w'_1\sim \frac{\tp}{\tq} \ll
\wmm$. The second maximum is at $\w^"_1\sim \frac{1}{4\nu\ell\tq} \pr
\kmm^{1+\beta}$. It belongs to $\ccw$, if $\wmm\pr\kmm^2$. In the other
case, the minimum at $\xi'_2$ does not belong either to $\ccw$, and
$$
  \sup_{\ccw}|\rho(z,p,q,L)|=max(|\rho(z_1,p,q,L)|,|\rho(\bz'_1,p,q,L)|).
$$
We compute now the local extrema on the curve $\ccn$, treating again
the two cases of interest:
\begin{itemize}
\item If $\wmm\pr\kmm^2$, the term $-\ell M$ dominates in the
  derivative, so that
  $$
    \Phi_k^\ell \sim-\frac{\ell}{2}M\partial_k x,
  $$
  and $R$ is a decreasing function of $x$ on $\ccn$.

\item If $\wmm\pr\kmm$, then we have the cases
\begin{itemize}
  \item[$\checkmark$] If $k=\go(\kmm)$, $\frac{l}{2}M \pr \yo$, $Re
    N(z,\bar{z})\sim -\tp^3-\tq^2\yo^2 \gg \yo$. Therefore the
    computations from the nonoverlapping case are valid. According
    to \eqref{eq:pointnord}, since $\frac{\tq\yo}{\tp} \gg 1$, there
    is no maximum for $k =\go(\kmm)$.

  \item[$\checkmark$] If $k\pr\kmm^\theta$, $\frac12 < \theta < 1$,
    $M\sim X^2(\tq^2X-1)^2$, and
    $$
    \begin{array}{rcl}
     \Phi_k^\ell&\sim& 2\nu k\sqrt{X}(\Re N(z,\bar{z})-\frac{l}{2} X^2(\tq^2X-1)^2)\\
     &\sim& 2\nu kX^{\frac32}(-\frac{\ell}{2}\tq^4 X^3+\tq^3X^2-\tq X+\tp).
\end{array}
$$
The polynomial on the right hand side is a singular perturbation of
the polynomial in $\Phi_k$, $\tq^3X^2-\tq X+\tp$, and it has
asymptotically the following two roots:
$$
 \frac{1}{\tq^2} \ll \frac{2}{\ell\tq}.
$$
The first one corresponds to a minimum, the second one to a
maximum. Therefore the overlap creates a new local maximum, $\bk_4'
\sim \frac{\soc}{2\nu}\sqrt{\frac{2}{\ell\tq}}$. The convergence
factor is in this case
$$
|\rho(\bk_4',\wmm,p,q)\sim 1- 2 \sqrt{\frac{2}{\ell\tq}}
$$
\end{itemize}
\end{itemize}
Hence we found all the possible maxima, and
\[
\sup_{z\in\td_+} |\rho(z,p,q,L)|=\max(|\rho(\bz_{sw}',p,q,L)|,
|\rho(\bz_1',p,q,L)|,|\rho(\bz_4',p,q,L)|).
\]
\end{proof}

\begin{lemma}[local minimum for $F_L(p,q)$]\label{lem:equioscpqoverlap}
There exist $\bar{p}^*_L\pr \kmm^\frac15$, $\bar{q}^*_L\pr
\kmm^{-\frac35}$ such that the three values in Lemma
\ref{lem:localmaxpqoverlapbounded} coincide. The coefficients and
associated convergence factor are given asymptotically by
$$
\bar{p}^*_L\sim
\begin{cases}
\sqrt[5]{\frac{x_{sw}^4}{2\ell}},
&\mbox{ if $\wmm\pr\kmm^2$},\\
\sqrt[5]{\frac{x_{sw}^4}{4\ell}},
&\mbox{ if $\wmm\pr\kmm$}
\end{cases}
,\quad
\bar{q}^*_L\sim 4\nu \frac{x_{sw}^2}{2\tp^3},\quad
\ds\sup_{z\in\td}|\rho(z,\bar{p}^*_L,\bar{q}^*_L,L)|\sim 1-2\sqrt[5]{4\ell x_{sw}}.
$$
\end{lemma}
\begin{proof}
We skip the arguments which are similar to those previously, and retain
only the conclusion.  The case $\wmm\pr\kmm^2$ is like in the previous
analysis. In the other case, we prove as before that there exist
$\bar{p}^*_L$ and $\bar{q}^*_L$ which solve the two equations
$$
  |\rho(\bz_{sw}^",p,q,L)|-|\rho(\bz_1',p,q,L)|=0,\quad
    |\rho(\bz_{sw}^",p,q,L)|-|\rho(\bz_4',p,q,L)|=0.
$$
The first one is the same as in the infinite case, providing the relation
$$
  2\tp^3\tq\sim x_{sw}^2,
$$
and the second one becomes
$$
  |\rho(\bz_{sw}^",p,q,L)|-|\rho(\bz_1^",p,q,L)|
    \sim  2\left( \sqrt{\frac{2\ell}{\tq}}-\frac{x_{sw}}{\tp}\right),
$$
which provides the relation
$$
  2\ell\frac{\tp^2}{\tq}\sim x_{sw}^2,
$$
and the solution
$$
\ds\tp_L\sim 2^{-\frac15}\tp_\infty ,\quad
\tq\sim 2^{\frac35}\tq_\infty.
\sup_{z\in\td_+} |\rho(z,p,q,L)|\sim 1-2\sqrt[5]{4\ell x_{sw}}.
$$
\end{proof}

We can conclude now the proof of Theorem \ref{th:toutoboundedpq} as in
the other cases.

\section{Numerical experiments}

We now present a substantial set of numerical experiments in order to
illustrate the performance of the optimized Schwarz waveform
relaxation algorithm, both for cases where our analysis is valid, and
for more general decompositions. We work on the domain
$\Omega=(0,1.2)\times(0,1.2)$ and chose for the coefficients in
(\ref{eq:ModelProblem}) $\nu=1$, ${\bf a}=(1, 1)$ and $b = 0$, and the
time interval length $T=1$. We discretized the problem using Q1 finite
elements and simulate directly the error equations, $f=0$, and start
with a random initial error, to make sure all frequencies are present,
see \cite{Gander:2008:SMO} for a discussion of the importance of
this. We use as the stopping criterion the relative residual reduction
to $10^{-6}$.  We start with the case of an implicit time integration
method (Backward Euler), where one can choose $\Delta
t=\frac{h}{4}$. We show in Table \ref{Table1} the number of iterations
needed by the various Schwarz waveform relaxation algorithms for the
case of non-overlapping decompositions.
\begin{table}
   \centering
\begin{tabular}{|c|c|ccccc|ccccc|}
\hline
\multicolumn{2}{|c|}{}& \multicolumn{5}{c|}{Iterative}&
\multicolumn{5}{c|}{GMRES}\\
\hline
\multicolumn{2}{|c|}{h} & 0.04 & 0.02 & 0.01 & 0.005 & 0.0025 & 0.04 &
0.02 & 0.01 & 0.005 & 0.0025 \\
\hline
\hline
\multirow{4}{*}{Robin} &2x1  & 49  & 71  & 97 & 144  & 198  & 23 & 29 &
36 & 45 & 55 \\
&2x2  & 53 & 74 & 101 & 145 & 202 & 30 & 38 & 48 & 59 & 73 \\
&4x1  & 52 & 72 & 101 & 140 & 204 & 30 & 40 & 50 & 63 & 78\\
&4x4  & 81 & 116 & 160 & 219 & 303 & 47 & 64 & 84 & 107 & 133\\
\hline
\multirow{4}{*}{Ventcell} &2x1  & 13 & 15 & 18 & 21 & 24 & 10 & 12 & 14
& 16 & 18 \\
&2x2  & 23 & 29 & 39 & 48 & 63 & 16 & 19 & 22 & 25 & 29 \\
&4x1  & 18 & 21 & 25 & 29 & 35 & 14 & 17 & 20 & 24 & 27 \\
&4x4  & 30 & 37 & 44 & 54 & 65 & 22 & 28 & 34 & 40 & 46\\
\hline
\end{tabular}
    \caption{Number of iterations for an implicit time discretization
      setting $\Delta t=\frac{h}{4}$, algorithms without overlap}
    \label{Table1}
\end{table}
We first note that the algorithms work also very well for
decompositions into more than two subdomains, and the optimized
parameters we derived are also very effective in that case. For
example for a decomposition into $4\times4$ subdomains and a high
mesh resolution, the Ventcell conditions need about 5 times less
iterations than the Robin conditions for convergence, and the cost per
iteration is virtually the same. 

In Table \ref{Table2}, we show the corresponding results for the
overlapping algorithms, using an overlap of $2h$. 
\begin{table}
\centering
\begin{tabular}{|c|c|ccccc|ccccc|}
\hline
\multicolumn{2}{|c|}{}& \multicolumn{5}{c|}{Iterative}&
\multicolumn{5}{c|}{GMRES}\\
\hline
\multicolumn{2}{|c|}{h} & 0.04 & 0.02 & 0.01 & 0.005 & 0.0025 & 0.04 &
0.02 & 0.01 & 0.005 & 0.0025 \\
\hline
\hline
\multirow{4}{*}{Robin} &2x1  & 12 & 14 & 16 & 19 & 23 & 8 & 10 & 12 & 14
& 17 \\
&2x2  & 14 & 17 & 21 & 27 & 33 & 11 & 14 & 17 & 20 & 24 \\
&4x1  & 14 & 15 & 18 & 23 & 29 & 11 & 13 & 16 & 20 & 24 \\
&4x4  & 19 & 24 & 32 & 41 & 52 & 14 & 20 & 26 & 32 & 40 \\
\hline
\multirow{4}{*}{Ventcell} &2x1  & 9 & 10 & 11 & 12 & 13 & 6 & 7 & 8 & 9
& 10 \\
&2x2  & 12 & 14 & 17 & 20 & 23 & 8 & 10 & 11 & 13 & 16 \\
&4x1  & 12 & 11 & 11 & 14 & 16 & 10 & 9 & 9 & 11 & 13 \\
&4x4  & 16 & 17 & 19 & 24 & 29 & 13 & 13 & 14 & 18 & 22 \\
\hline
\multirow{4}{*}{Classical} &2x1  & 54 & 106 & 189 & 360 & 733 & 27 & 40
& 58 & 83 & 117 \\
&2x2  & 84 & 159 & 303 & 570 & 1058 & 37 & 56 & 82 & 118 & 166  \\
&4x1  & 73 & 145 & 282 & 553 & 969 & 38 & 60 & 89 & 127 & 179\\
&4x4  & 127 & 258 & 487 & 912 & 1706 & 54 & 94 & 143 & 209 & 296 \\
\hline
\end{tabular}
    \caption{Number of iterations for an implicit time discretization
      setting $\Delta t=\frac{h}{4}$, algorithms with overlap $2h$}
\label{Table2}
\end{table}
We see that overlap greatly enhances the convergence of the
algorithms, as predicted by our analysis. At a high mesh resolution,
the number of iterations on the $4\times 4$ example can be reduced by
a factor of 6 using overlap in the case of Robin conditions, and by a
further factor of 2 when optimized Ventcell conditions are used.

We illustrate our asymptotic results now in Figure \ref{Figure1} by
plotting in dashed lines the iteration numbers from Table \ref{Table1}
and \ref{Table2} in log-log scale, and we add the theoretically
predicted growth of the iteration numbers.
\begin{figure}
  \centering
  \psfrag{iteration number}[][]{\footnotesize interation number}
  \psfrag{h}[][]{\footnotesize $h$}
  \includegraphics[width=0.48\textwidth]{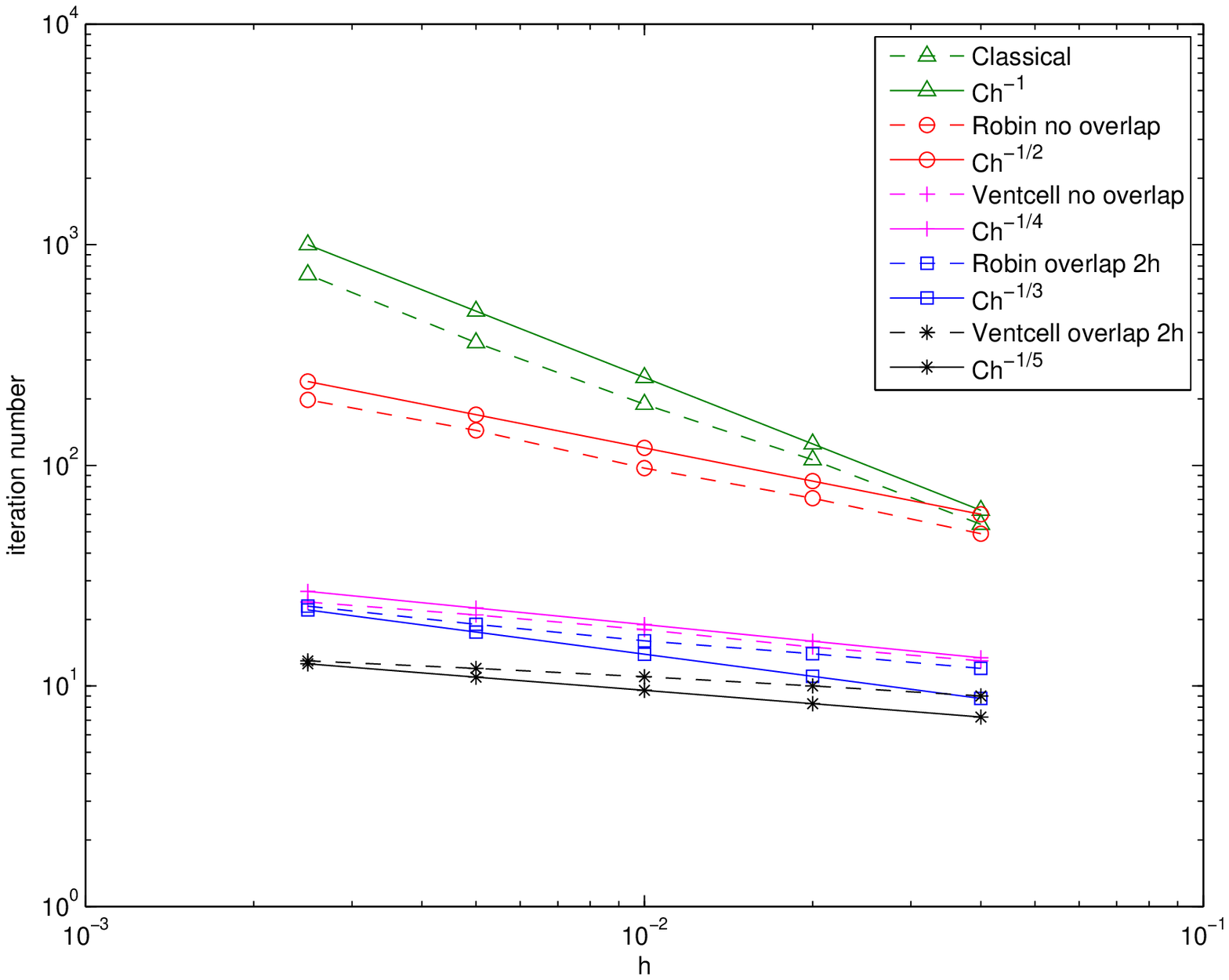}
  \includegraphics[width=0.48\textwidth]{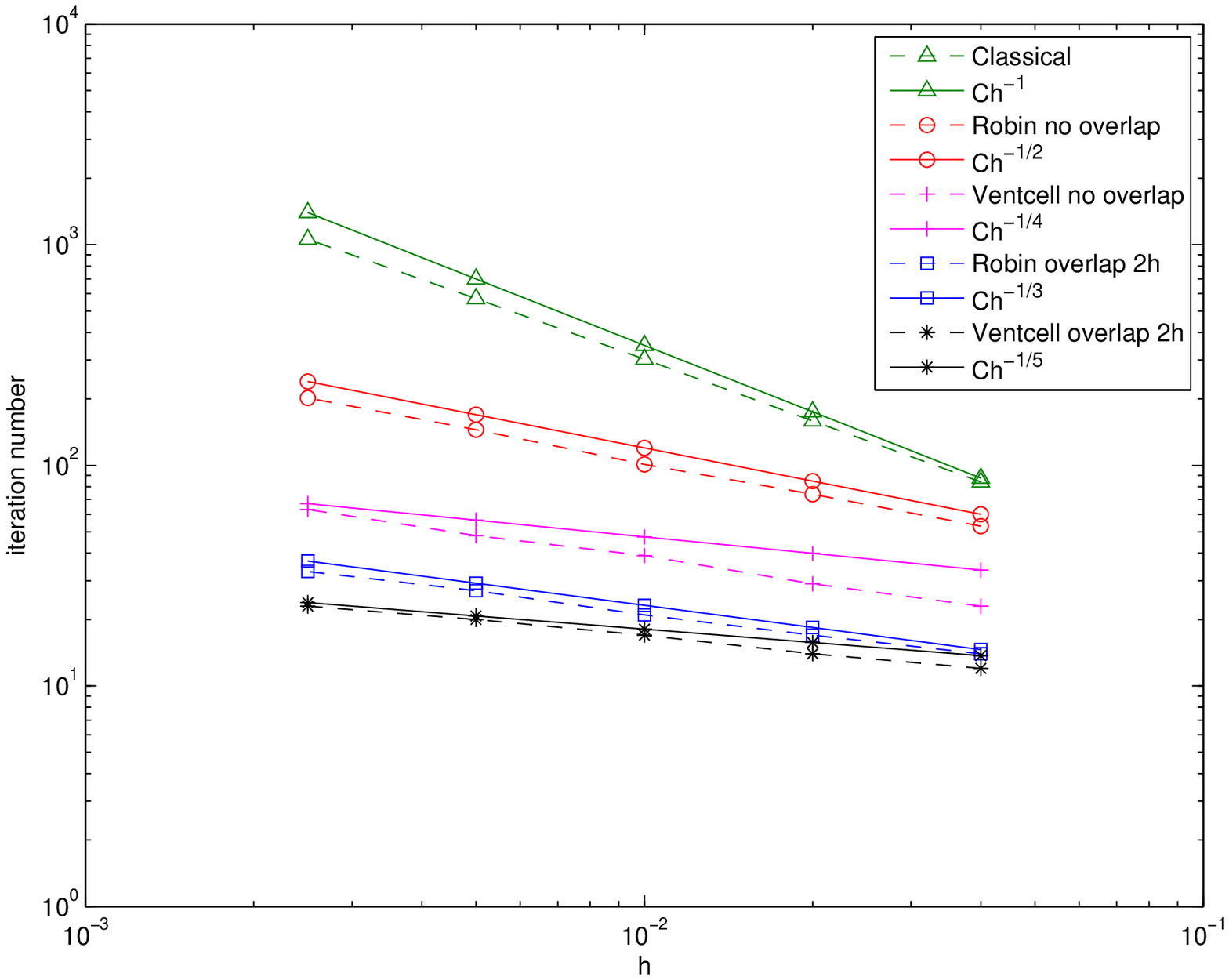}
  \includegraphics[width=0.48\textwidth]{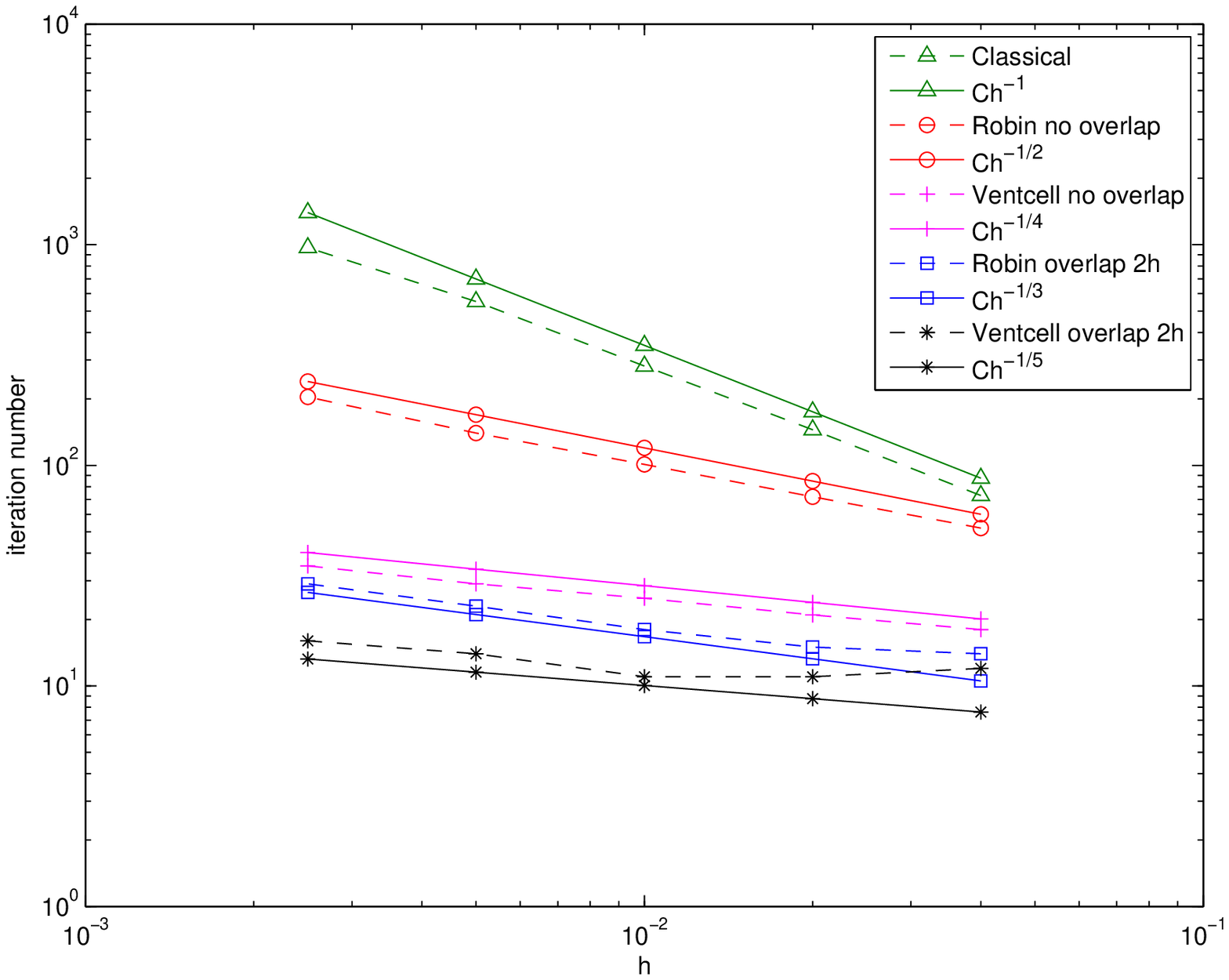}
  \includegraphics[width=0.48\textwidth]{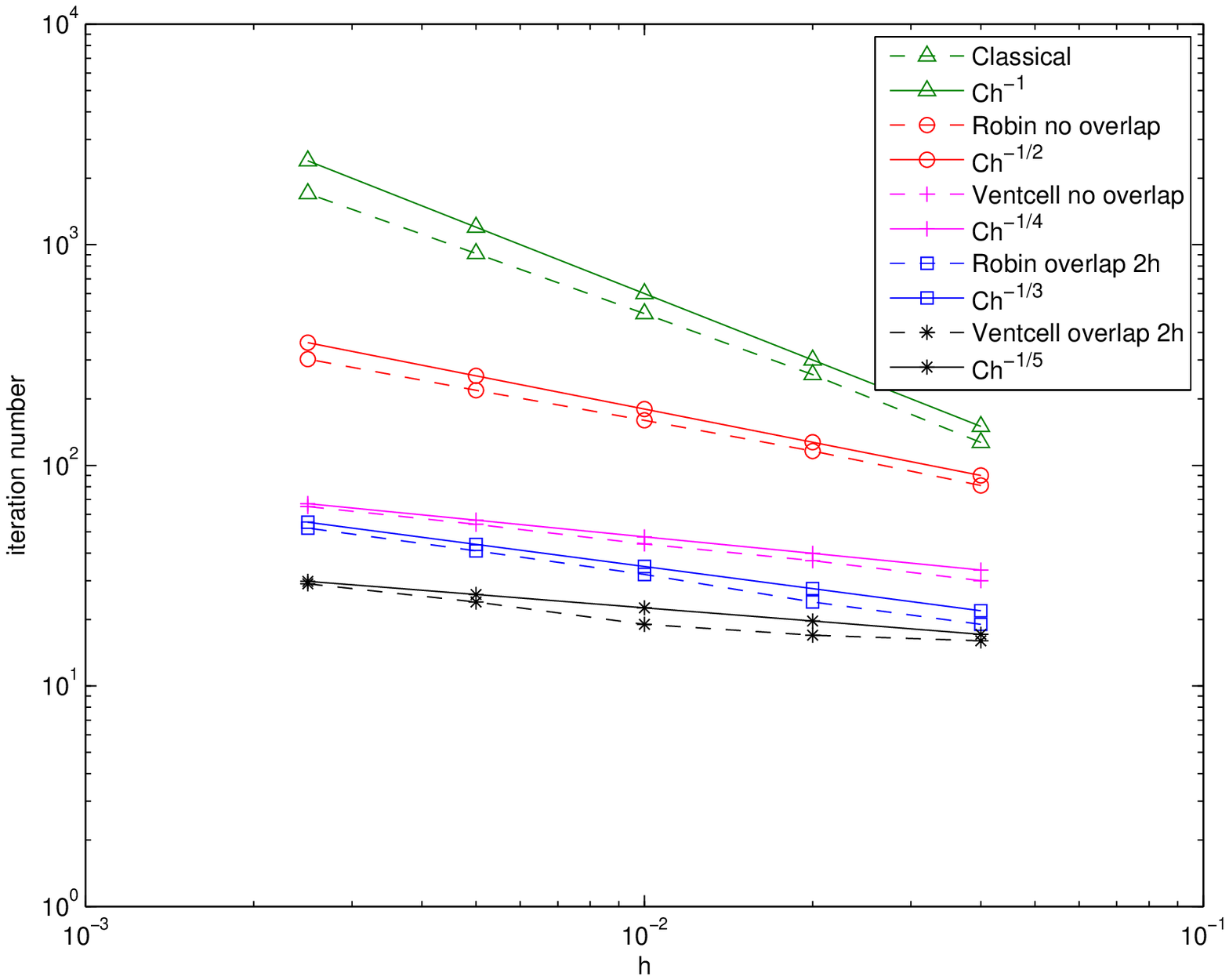}
  \caption{Plots of the iteration numbers from Table \ref{Table1} and
    \ref{Table2} when the methods are used iteratively, and
    theoretically predicted rates. Top left $2\times 1$ subdomains,
    Top right $2\times 2$ subdomain, bottom left $4\times 1$
    subdomains and bottom right $4\times 4$ subdomains}
  \label{Figure1}
\end{figure}
We see that our asymptotic analysis for the two subdomain case also
predicts quite well the behavior of the algorithms in the case of many
subdomains. 

Next, we investigate the setting of an explicit method (Forward Euler
with mass lumping), where $\Delta t=h^2/4$. We show in Table
\ref{Table3} and \ref{Table4} the number of iterations needed to
reduce the relative residual again by a factor of $10^{-6}$.
\begin{table}
\centering
\begin{tabular}{|c|c|cccc|cccc|}
\hline
\multicolumn{2}{|c|}{}& \multicolumn{4}{c|}{Iterative}&
\multicolumn{4}{c|}{GMRES}\\
\hline
\multicolumn{2}{|c|}{h} & 0.04 & 0.02 & 0.01 & 0.005 & 0.04 & 0.02 &
0.01 & 0.005 \\
\hline
\hline
\multirow{4}{*}{Robin} &2x1  & 57 & 85 & 117 & 176 & 24 & 31 & 36 & 44  \\
&2x2  & 59 & 87 & 117 & 174 & 25 & 32 & 39 & 48\\
&4x1  & 63 & 86 & 121 & 170 & 26 & 30 & 37 & 44 \\
&4x4  & 62 & 84 & 123 & 166 & 26 & 31 & 40 & 48\\
\hline
\multirow{4}{*}{Ventcell} &2x1  & 20 & 22& 25 & 28 & 12 & 13 & 15 & 16\\
&2x2  & 22 & 25 & 26 & 30 & 13 & 14 & 16 & 18 \\
&4x1  & 21 & 22 & 25 & 29 & 12 & 14 & 15 & 16\\
&4x4  & 23 & 27 & 26 & 34 & 15 & 16 & 18 & 19\\
\hline
\end{tabular}
    \caption{Number of iterations for an explicit time discretization
      setting $\Delta t = \frac{h^2}{4}$, without overlap}
    \label{Table3}
\end{table}

\begin{table}
\centering
\begin{tabular}{|c|c|cccc|cccc|}
\hline
\multicolumn{2}{|c|}{}& \multicolumn{4}{c|}{Iterative}&
\multicolumn{4}{c|}{GMRES}\\
\hline
\multicolumn{2}{|c|}{h} & 0.04 & 0.02 & 0.01 & 0.005 & 0.04 & 0.02 &
0.01 & 0.005 \\
\hline
\hline
\multirow{4}{*}{Robin} &2x1  & 13 & 16 & 20 & 24 & 8 & 9 & 10 & 10  \\
&2x2  & 13 & 16 & 19 & 23 & 9 & 10 & 11 & 12\\
&4x1  & 14 & 18 & 20 & 24 & 9 & 10 & 12 & 12 \\
&4x4  & 14 & 18 & 20 & 23 & 10 & 13 & 15 & 16 \\
\hline
\multirow{4}{*}{Ventcell} &2x1  & 9 & 10 & 11 & 13 & 6 & 8 & 9 & 10\\
&2x2  & 9 & 10 & 11 & 13 & 7 & 8 & 9 & 10\\
&4x1  & 9 & 10 & 11 & 14 & 7 & 8 & 9 & 10\\
&4x4  & 11 & 11 & 12 & 14 & 8 & 9 & 10 & 11\\
\hline
\multirow{4}{*}{Classical} &2x1  & 25 & 46 & 88 & 169 & 17 & 27 & 43 & 66\\
&2x2  & 33 & 63 & 122 & 235 & 21 & 34 & 54 & 83\\
&4x1  & 25 & 48 & 91 & 176 & 17 & 27 & 43 & 66\\
&4x4  & 36 & 70 & 136 & 263 & 22 & 36 & 58 & 89\\
\hline
\end{tabular}
    \caption{Number of iterations for an explicit time discretization
      setting $\Delta t = \frac{h^2}{4}$, with overlap $2h$}
    \label{Table4}
\end{table}


We illustrate our asymptotic results now in Figure \ref{Figure2} by
plotting in dashed lines the iteration numbers from Table \ref{Table3}
and \ref{Table4} in log-log scale, and we add the theoretically
predicted growth of the iteration numbers.
\begin{figure}
  \centering
  \psfrag{iteration number}[][]{\footnotesize interation number}
  \psfrag{h}[][]{\footnotesize $h$}
  \includegraphics[width=0.48\textwidth]{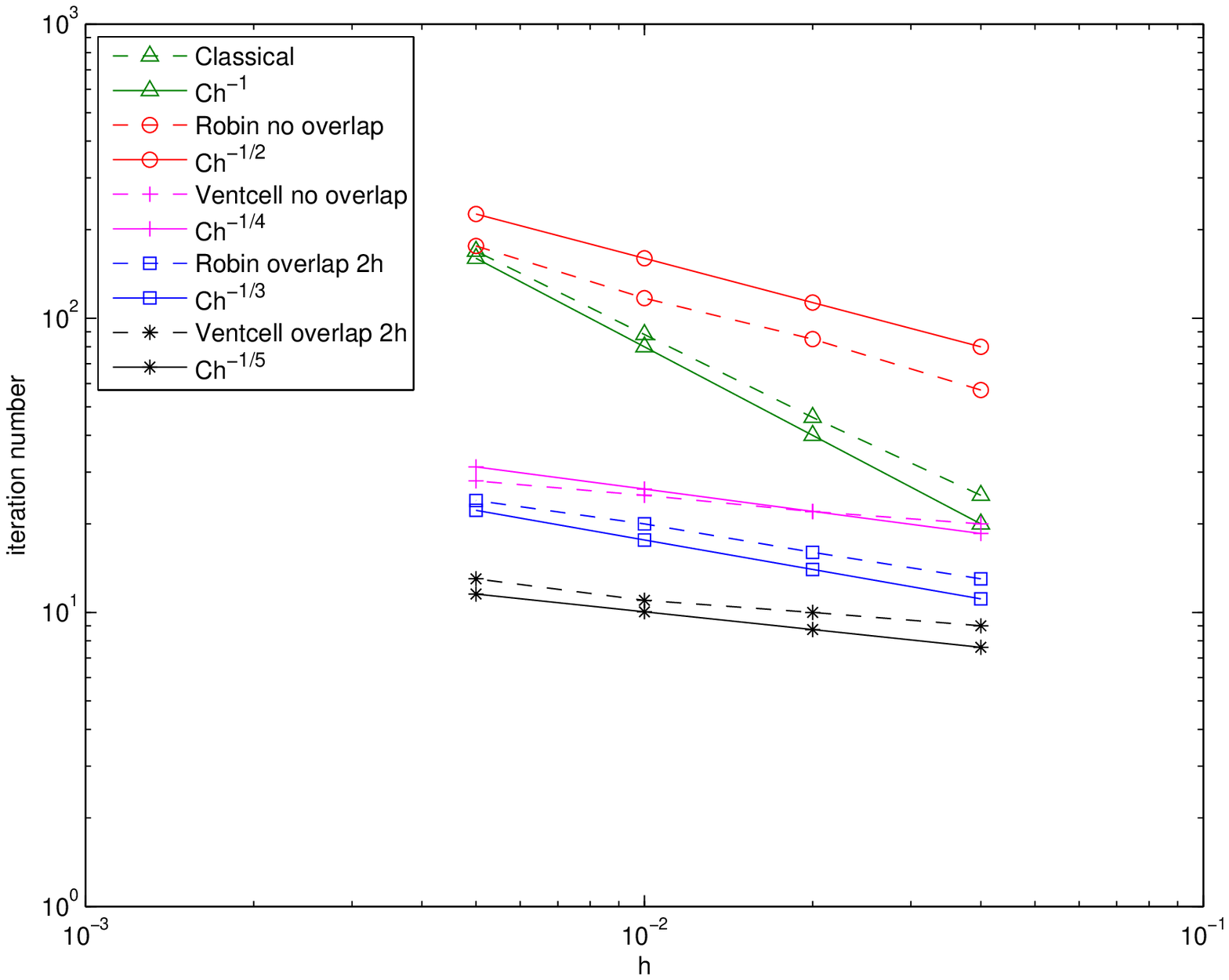}
  \includegraphics[width=0.48\textwidth]{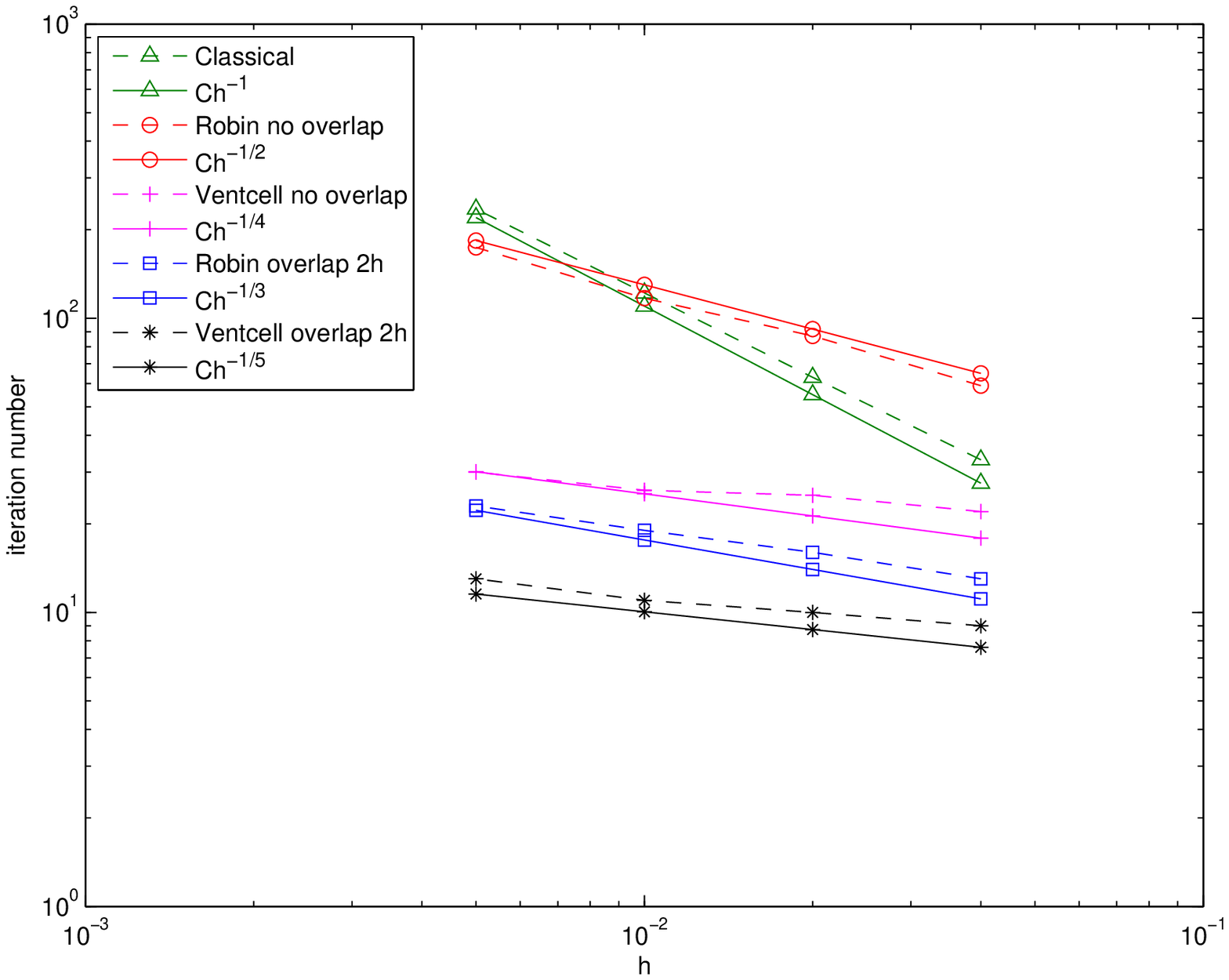}
  \includegraphics[width=0.48\textwidth]{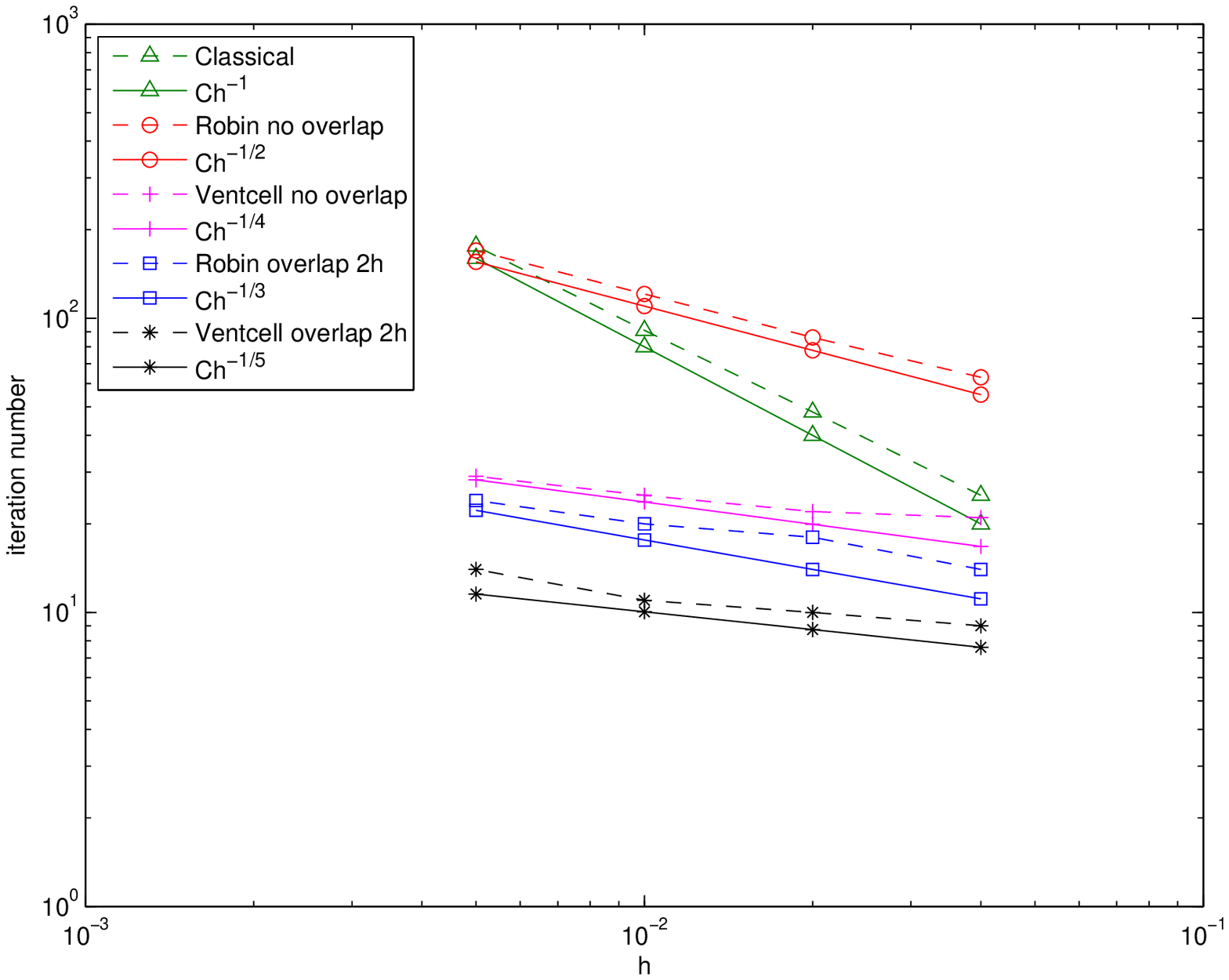}
  \includegraphics[width=0.48\textwidth]{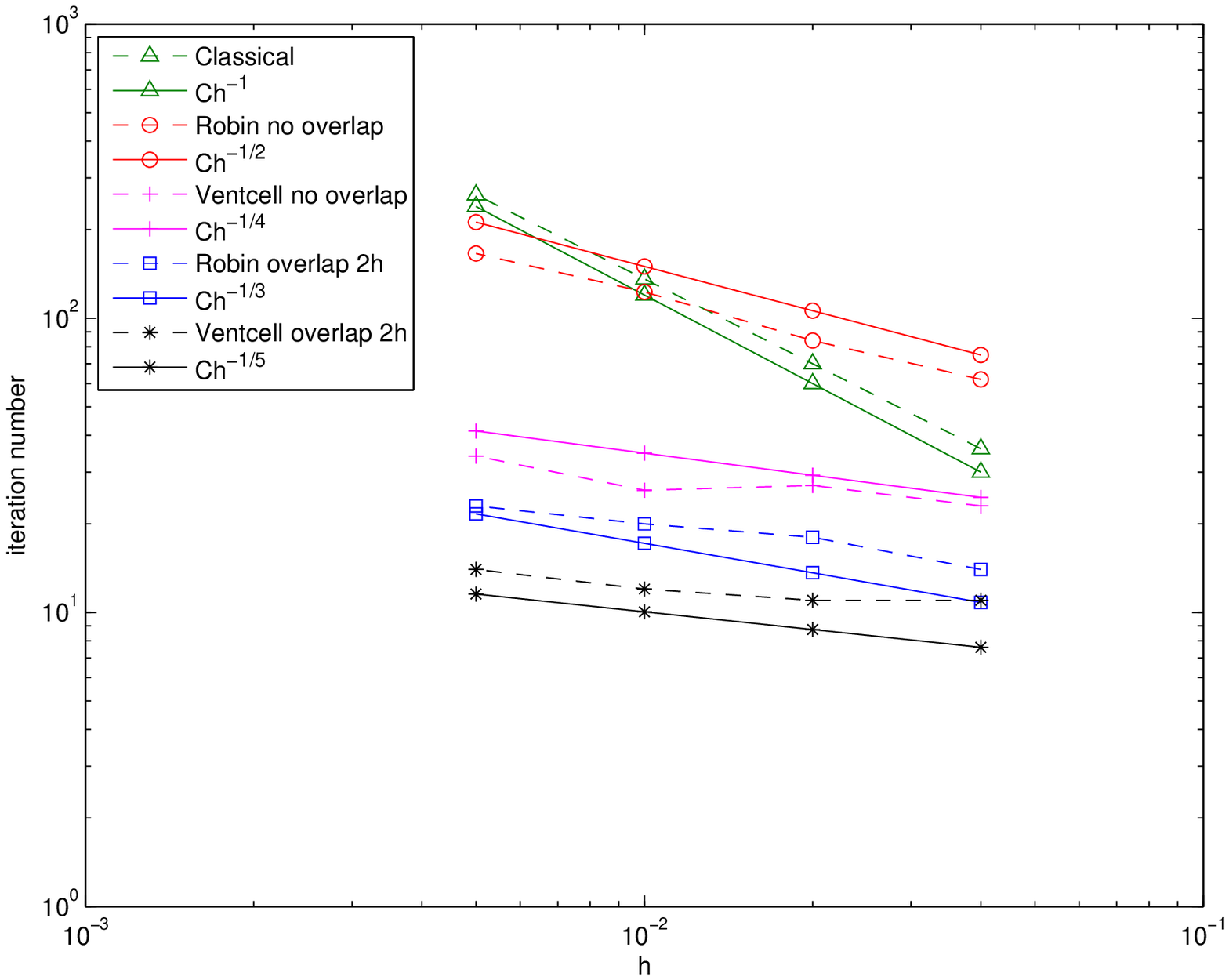}
  \caption{Plots of the iteration numbers from Table \ref{Table3} and
    \ref{Table4} when the explicitly discretized methods are used
    iteratively, and theoretically predicted rates. Top left $2\times
    1$ subdomains, Top right $2\times 2$ subdomain, bottom left
    $4\times 1$ subdomains and bottom right $4\times 4$ subdomains}
  \label{Figure2}
\end{figure}
As in the implicit case shown earlier, the asymptotic behavior we
observe follows our analysis of the two subdomain case, also in the
experiments with many subdomains.

\section{Conclusion}

We provide in this paper the complete asymtotically optimized closed
form transmission conditions for optimized Schwarz waveform relaxation
algorithms applied to advection reaction diffusion problems in higher
dimensions. We showed the results for the case of two spatial
dimensions, but the extension to higher dimensions $d>2$ from there is
trivial, it suffices to replace the Fourier variable contributions
$k^2$ by $||\vec{k}||^2$, and $ck$ by $\vec{c}\cdot\vec{k}$, which
implies to replace in the asymptotic analysis the highest frequency
estimate $k_M=\frac{\pi}{h}$ by $k_M=\frac{\sqrt{d-1}\pi}{h}$, or
replacing $\pi$ by $\sqrt{d-1}\pi$ in the final asymptotically
optimized closed form formulas. The formulas for Robin and Vencel
conditions are derived such that limits to pure diffusion can be
taken, and therefore also the associated time dependent heat equation
optimization problems are solved by our formulas. The formulas are
equally good for advection dominated problems, although one has to pay
attention there to have fine enough mesh sizes to resolve boundary
layers, in order for the asymptotically optimized formulas to be
valid. We extensively tested our algorithms numerically, see also
\cite{Gander:2010:OSW} for more scaling experiments, and these tests
indicate that our theoretical asymtptotic formulas derived for two
subdomain decompositions are also very effective for more general
decompositions into many subdomains.


%
\bibliographystyle{plain}
\bibliography{paper} 


\end{document}